\RequirePackage{iftex}
\ifpTeX
  \documentclass[a4paper,dvipdfmx,11pt]{amsart}
\else
  \documentclass[a4paper,11pt]{amsart}
\fi
\usepackage{amsmath,amsthm,amssymb,mathrsfs,stmaryrd,mathtools}
\allowdisplaybreaks
\usepackage{bm}
\usepackage{url}
\usepackage{color}
\usepackage{eucal}
\usepackage{slashed}
\usepackage{physics}
\usepackage{graphicx}
\usepackage{tikz}
\usetikzlibrary{nfold,arrows}
\usepackage{tikz-cd}
\usetikzlibrary{intersections, calc, arrows, arrows.meta, shadows}
\usepackage[utf8]{inputenc} 
\usepackage[T1]{fontenc}
\usepackage[style=alphabetic,sorting=nty]{biblatex}
\usepackage{geometry} 
\usepackage[inline]{enumitem} 
\setlist[enumerate]{itemsep=2pt,parsep=2pt,before={\parskip=2pt}}
\setlist[description]{style=standard}
\usepackage[colorlinks=true, hyperindex, linkcolor=magenta, pagebackref=false, citecolor=cyan, pdfpagelabels]{hyperref} 
\usepackage{aliascnt}
\usepackage{cleveref}
\usepackage{quiver}
\mathtoolsset{showonlyrefs=true}

\newcommand{\calA}{\mathcal{A}}
\newcommand{\calC}{\mathcal{C}}
\newcommand{\calD}{\mathcal{D}}
\newcommand{\calE}{\mathcal{E}}
\newcommand{\calF}{\mathcal{F}}
\newcommand{\calH}{\mathcal{H}}
\newcommand{\calO}{\mathcal{O}}
\newcommand{\calP}{\mathcal{P}}

\newcommand{\frakC}{\mathfrak{C}}

\newcommand{\D}{\mathcal{D}}

\newcommand{\bbA}{\mathbb{A}}
\newcommand{\bbB}{\mathbb{B}}
\newcommand{\bbG}{\mathbb{G}}
\newcommand{\bbH}{\mathbb{H}}
\newcommand{\bbI}{\mathbb{I}}
\newcommand{\bbJ}{\mathbb{J}}
\newcommand{\bbN}{\mathbb{N}} 

\newcommand{\rmD}{\mathrm{D}} 

\newcommand{\rmob}{\mathrm{ob}}
\newcommand{\rmord}{\mathrm{ord}}

\newcommand{\bfL}{\mathbf{L}}
\newcommand{\bfR}{\mathbf{R}}
\newcommand{\bfz}{\mathbf{z}}
\newcommand{\bfone}{\mathbf{1}}

\newcommand{\FdagSSet}{F_{\dag}^{\sSet}}
\newcommand{\UdagSSet}{U_{\dag}^{\sSet}}
\newcommand{\FdagSCat}{F_{\dag}^{\sCat}}
\newcommand{\UdagSCat}{U_{\dag}^{\sCat}}

\newcommand{\Bergner}{\mathrm{Bergner}}

\newcommand{\Cat}{\mathrm{Cat}}

\newcommand{\Ex}{\mathrm{Ex}}

\newcommand{\Fun}{\mathrm{Fun}}

\newcommand{\RFib}{\mathrm{RFib}}

\newcommand{\sCat}{\mathrm{sCat}}

\newcommand{\Set}{\mathrm{Set}}

\newcommand{\sSpace}{\mathrm{sSpace}}

\newcommand{\sSet}{\mathrm{sSet}}

\newcommand{\Barconstruction}{\mathrm{Bar}}

\newcommand{\CSS}{\mathrm{CSS}}

\DeclareMathOperator{\colim}{colim}
\DeclareMathOperator*{\colimstar}{colim}

\renewcommand{\ev}{\mathrm{ev}} 
\newcommand{\fib}{\mathrm{fib}}

\newcommand{\Fill}{\mathrm{Fill}}
\DeclareMathOperator{\Fillop}{Fill}

\newcommand{\h}{\mathrm{h}} 
\newcommand{\Ho}{\mathrm{Ho}}
\newcommand{\hofib}{\mathrm{hofib}}

\DeclareMathOperator{\Hom}{Hom}

\newcommand{\inj}{\mathrm{inj}}
\DeclareMathOperator{\id}{id}

\newcommand{\Joyal}{\mathrm{Joyal}}

\DeclareMathOperator{\Map}{Map}
\DeclareMathOperator{\Mor}{Mor}

\newcommand{\N}{\mathrm{N}} 

\DeclareMathOperator{\Ob}{Ob}
\DeclareMathOperator{\Orb}{Orb}
\DeclareMathOperator{\myop}{op}

\newcommand{\Path}{\mathrm{Path}}

\newcommand{\proj}{\mathrm{proj}}
\newcommand{\res}{\mathrm{res}}

\newcommand{\rev}{\mathrm{rev}} 
\newcommand{\Reedy}{\mathrm{Reedy}}
\newcommand{\RMap}{\mathrm{RMap}}

\newcommand{\sk}{\mathrm{sk}} 
\newcommand{\Sp}{\mathrm{Sp}} 
\newcommand{\triv}{\mathrm{triv}} 

\newcommand{\UEquiv}{\mathrm{UEquiv}}

\usepackage{relsize}
\usepackage[bbgreekl]{mathbbol}
\usepackage{amsfonts}
\DeclareSymbolFontAlphabet{\mathbb}{AMSb} 
\DeclareSymbolFontAlphabet{\mathbbl}{bbold}
\newcommand{\prism}{{\mathlarger{\mathbbl{\Delta}}}}

\DeclareFontFamily{U}{dmjhira}{}
\DeclareFontShape{U}{dmjhira}{m}{n}{
  <-> dmjhira
}{}
\DeclareFontSubstitution{U}{dmjhira}{m}{n}
\newcommand{\yo}{\text{{\usefont{U}{dmjhira}{m}{n}\symbol{"48}}}}

\makeatletter
\renewcommand{\lim}{\mathop{\operator@font lim}\nolimits}
\makeatother

\theoremstyle{definition}
\newtheorem{theorem}{Theorem}[subsection]

\newaliascnt{conjecture}{theorem}

\aliascntresetthe{conjecture}
\crefname{conjecture}{conjecture}{conjectures}
\Crefname{conjecture}{Conjecture}{Conjectures}

\newaliascnt{construction}{theorem}
\newtheorem{construction}[construction]{Construction}
\aliascntresetthe{construction}
\crefname{construction}{construction}{constructions}
\Crefname{construction}{Construction}{Constructions}

\newaliascnt{corollary}{theorem}
\newtheorem{corollary}[corollary]{Corollary}
\aliascntresetthe{corollary}
\crefname{corollary}{corollary}{corollaries}
\Crefname{corollary}{Corollary}{Corollaries}

\newaliascnt{definition}{theorem}
\newtheorem{definition}[definition]{Definition}
\aliascntresetthe{definition}
\crefname{definition}{definition}{definitions}
\Crefname{definition}{Definition}{Definitions}

\newaliascnt{example}{theorem}

\aliascntresetthe{example}
\crefname{example}{example}{examples}
\Crefname{example}{Example}{Examples}

\newaliascnt{lemma}{theorem}
\newtheorem{lemma}[lemma]{Lemma}
\aliascntresetthe{lemma}
\crefname{lemma}{lemma}{lemmas}
\Crefname{lemma}{Lemma}{Lemmas}

\newaliascnt{notation}{theorem}
\newtheorem{notation}[notation]{Notation}
\aliascntresetthe{notation}
\crefname{notation}{notation}{notations}
\Crefname{notation}{Notation}{Notations}

\newaliascnt{observation}{theorem}

\aliascntresetthe{observation}
\crefname{observation}{observation}{observations}
\Crefname{observation}{Observation}{Observations}

\newaliascnt{proposition}{theorem}
\newtheorem{proposition}[proposition]{Proposition}
\aliascntresetthe{proposition}
\crefname{proposition}{proposition}{propositions}
\Crefname{proposition}{Proposition}{Propositions}

\newaliascnt{remark}{theorem}
\newtheorem{remark}[remark]{Remark}
\aliascntresetthe{remark}
\crefname{remark}{remark}{remarks}
\Crefname{remark}{Remark}{Remarks}

\newaliascnt{fact}{theorem}

\aliascntresetthe{fact}
\crefname{fact}{fact}{facts}
\Crefname{fact}{Fact}{Facts}

\newaliascnt{terminology}{theorem}

\aliascntresetthe{terminology}
\crefname{terminology}{terminology}{terminologies}
\Crefname{terminology}{Terminology}{Terminologies}

\theoremstyle{definition}
\newtheorem*{theorem*}{Theorem}
\newtheorem*{conjecture*}{Conjecture}
\newtheorem*{construction*}{Construction}
\newtheorem*{corollary*}{Corollary}
\newtheorem*{definition*}{Definition}
\newtheorem*{example*}{Example}
\newtheorem*{lemma*}{Lemma}
\newtheorem*{notation*}{Notation}
\newtheorem*{observation*}{Observation}
\newtheorem*{proposition*}{Proposition}
\newtheorem*{remark*}{Remark}
\newtheorem*{terminology*}{Terminology}

\newtheorem*{maintheorema}{Theorem A}
\newtheorem*{maintheoremb}{Theorem B}
\newtheorem*{maintheoremc}{Theorem C}

\title[Models for dagger $(\infty,1)$-categories II]{Models for dagger $(\infty,1)$-categories II:\\Dagger complete Segal spaces and the Joyal--Tierney equivalence}

\author{Keima Akasaka} 

\thanks{Graduate School of Science and Engineering, Chiba University.
  Email: \href{mailto:quasi.cosmoi@gmail.com}{\texttt{quasi.cosmoi@gmail.com}}.
}

\date{\today}

\begin{document}

\begin{abstract}
  Building on the dagger Bergner and dagger Joyal models constructed in Part I \cite{Aka26}, we develop simplicial-space models for dagger $(\infty,1)$-categories.
  We first construct a homotopically constant resolution model and a model structure for unitarily complete dagger Segal spaces.
  A change-of-index adjunction compares the two presentations of dagger simplicial spaces.
  We then prove a dagger analogue of the Joyal--Tierney Quillen equivalence.
  Together with Part I, these results give a chain of Quillen equivalences between the dagger Bergner, dagger Joyal, and dagger Rezk models.
\end{abstract}

\maketitle
\setcounter{tocdepth}{1}
\tableofcontents

\section{Introduction}

In Part I \cite{Aka26}, we constructed the dagger Bergner and dagger Joyal model structures and proved that the dagger rigidification--nerve adjunction is a Quillen equivalence.
The purpose of this second part is to construct simplicial-space presentations of the same homotopy theory and to compare them by Quillen equivalences.
These constructions are dagger analogues of Rezk's complete Segal spaces and the Joyal--Tierney comparison \cite{Rez01,JT06}.

The dagger structure can be recorded in a simplicial-space model in two related ways.
One may form simplicial objects in dagger simplicial sets, or one may enlarge the simplex category by order-reversing maps and consider simplicial-set-valued presheaves on $\prism_{\rev}$.
In the latter presentation, the Segal condition encodes composition, while the completeness map is determined by the walking unitary isomorphism and encodes unitary completeness.
We also compare $\prism_{\rev}$ with the category $\prism_{\dag}$ of free dagger simplices.

\subsection*{The main theorems}

Our first main result is a homotopically constant resolution model on $\sSpace^{\dag\vee}=\Fun(\prism^{\myop},\sSet^{\dag})$.

\begin{maintheorema}[\cref{cs.def.Sh,cs.thm.resolution,cs.thm.equivalence}]
  There exists a left proper combinatorial model structure $\sSpace^{\dag\vee}_{\CSS}:=L_{S_h}\sSpace^{\dag\vee}_{\Reedy}$.
  Its fibrant objects are precisely the Reedy fibrant homotopically constant objects.
  
  Moreover, the constant--evaluation adjunction induces a Quillen equivalence
  \begin{align}
    c:
    \sSet^{\dag}_{\Joyal}
    \rightleftarrows
    \sSpace^{\dag\vee}_{\CSS}
    :\ev_0.
  \end{align}
\end{maintheorema}

We next construct a dagger complete-Segal-space presentation on $\sSpace^{\dag}_{\rev}=\Fun(\prism_{\rev}^{\myop},\sSet)$.
Between fibrant objects, its weak equivalences admit a dagger Dwyer--Kan characterization.
Restriction along $u:\prism_{\rev}\to\prism_{\dag}$ then relates this model to $\sSpace^{\dag}=\Fun(\prism_{\dag}^{\myop},\sSet)$.

\begin{maintheoremb}[\cref{h1.lem.rev-models,h1.prop.fibrant,h1.thm.main}]
  The localization $L_{S_{\rev}}\sSpace^{\dag}_{\rev,\inj}$ is a left proper combinatorial simplicial model category.
  Its fibrant objects are precisely the injectively fibrant unitarily complete dagger Segal spaces.
  
  On the other hand, the fibrant objects of $L_{\widetilde S_{\rev}}\sSpace^{\dag}_{\rev,\proj}$ are precisely the levelwise fibrant unitarily complete dagger Segal spaces in the derived sense.
  
  Moreover, the change-of-index adjunction is a Quillen equivalence
  \begin{align}
    u_{!}:
    L_{\widetilde S_{\rev}}\sSpace^{\dag}_{\rev,\proj}
    \rightleftarrows
    L_{u_{!}(\widetilde S_{\rev})}\sSpace^{\dag}_{\proj}
    :u^{*}.
  \end{align}
  A fibrant object $W$ of $L_{u_{!}(\widetilde S_{\rev})}\sSpace^{\dag}_{\proj}$ is precisely a levelwise fibrant object whose restriction $u^{*}W$ is a unitarily complete dagger Segal space in the derived sense.
\end{maintheoremb}

The final comparison is the dagger analogue of the Joyal--Tierney theorem.
Its proof uses reversal-compatible cosimplicial frames, a comparison with the classical Joyal--Tierney realization, and a strictification theorem for local objects.

\begin{maintheoremc}[\cref{jt.thm.main,jt.cor.goal}]
  The constant--vertex adjunction induces a Quillen equivalence
  \begin{align}
    \delta:
    \sSet^{\dag}_{\Joyal}
    \rightleftarrows
    L_{S_{\rev}}\sSpace^{\dag}_{\rev,\inj}
    :i_0^{*}.
  \end{align}
  Together with the dagger Joyal--Bergner equivalence of Part I \cite{Aka26}, it gives a zigzag of Quillen equivalences whose two parts meet at the injective localization:
  \begin{align}
    \sCat^{\dag}_{\Bergner}
    \underset{\frakC_{\dag}}{\overset{\N_{\dag}}{\rightleftarrows}}
    \sSet^{\dag}_{\Joyal}
    \underset{i_0^{*}}{\overset{\delta}{\rightleftarrows}}
    L_{S_{\rev}}\sSpace^{\dag}_{\rev,\inj}
    \underset{\id}{\overset{\id}{\rightleftarrows}}
    L_{\widetilde S_{\rev}}\sSpace^{\dag}_{\rev,\proj}
    \underset{u^{*}}{\overset{u_{!}}{\rightleftarrows}}
    L_{u_{!}(\widetilde S_{\rev})}\sSpace^{\dag}_{\proj}.
  \end{align}
\end{maintheoremc}

\subsection*{Organization}

This second part is organized as follows:

\begin{itemize}
  \item \cref{review.section} recalls the dagger Bergner and dagger Joyal model structures constructed in Part I.
  It also reviews coherent unitarity and the dagger rigidification--nerve Quillen equivalence used below;
  \item \cref{cs.section} studies simplicial objects in dagger simplicial sets.
  It constructs the resolution model structure, characterizes its fibrant objects as the Reedy fibrant homotopically constant objects, and proves the constant--evaluation Quillen equivalence of Theorem~A;
  \item \cref{h1.section} constructs the dagger complete-Segal-space models.
  It characterizes their fibrant objects, compares their projective and injective presentations, and proves the change-of-index Quillen equivalence between the models over $\prism_{\rev}$ and $\prism_{\dag}$ stated in Theorem~B;
  \item \cref{jt.frames-section} constructs the reversal-compatible frame realization and computes the derived realization of the walking unitary;
  \item \cref{jt.section} proves the dagger Joyal--Tierney comparison.
  It establishes the dagger Rezk recognition criterion, strictifies local objects to frame nerves, and proves the Quillen equivalence of Theorem~C.
\end{itemize}

In Part~III \cite{Aka26-3}, we compare the model-categorical presentations constructed in Parts~I and~II with univalent flagged dagger $(\infty,1)$-categories.
Consequently, the model categories of Parts~I and~II present the $\infty$-category of univalent flagged dagger $(\infty,1)$-categories.

\subsection*{Notation}

Unless stated otherwise, weak equivalences and fibrations of simplicial sets refer to the Kan--Quillen model structure.

We use the following notational conventions throughout this paper:
\begin{itemize}
  \item we write $\prism$ for the simplex category, whose objects are denoted by $[n]$;
  \item we write $\prism_{\rev}$ for the category with the same objects as $\prism$ whose morphisms are the order-preserving and order-reversing maps;
  \item we let $C_2:=\{1,\omega\}$, with $\omega^2=1$, denote the cyclic group of order two;
  \item for $n\geq0$, we write $\rho_n:[n]\to[n]$ for the order-reversing automorphism in $\prism_{\rev}$ given by $i\mapsto n-i$;
  \item we write $\Set$, $\Cat$, $\sCat$, and $\sSet$ for the categories of sets, small categories, small simplicial categories, and simplicial sets, respectively;
  \item we write $\Cat^{\dag}$, $\sCat^{\dag}$, and $\sSet^{\dag}$ for the categories of small dagger $1$-categories, dagger simplicial categories, and dagger simplicial sets, respectively.
  The forgetful functors are denoted by 
  \begin{align}
    U_{\dag}^{\Cat}:\Cat^{\dag}\to\Cat, 
    \quad
    \UdagSCat:\sCat^{\dag}\to\sCat,
    \quad \text{and} \quad
    \UdagSSet:\sSet^{\dag}\to\sSet;
  \end{align}
  \item the free--forgetful adjunctions in the simplicial settings are denoted by 
  \begin{align}
    \FdagSCat:\sCat \rightleftarrows \sCat^{\dag}:\UdagSCat
    \quad \text{and} \quad
    \FdagSSet:\sSet \rightleftarrows \sSet^{\dag}:\UdagSSet.
  \end{align}
  For $A\in\sSet$, we have $\FdagSSet(A)=A\amalg_{\sk_0A}A^{\myop}$ with the involution exchanging the two factors;
  \item for a set $O$, we write $\sCat^{\dag}_O$ for the category of dagger simplicial categories with object set $O$ and dagger simplicial functors which are the identity on $O$.
  \item we write $\Delta^0_{\dag}:=\FdagSSet(\Delta^0)$ and write $\bfone_{\dag}$ for the one-object dagger simplicial category with one morphism;
  \item we write $\yo_{\rev}[n]$ for the representable presheaf of $[n]$ in $\Fun(\prism_{\rev}^{\myop},\Set)$.
  Its $C_2$-action is induced by $\rho_n$;
  \item for a simplicial category $\calC$, we write $\h\calC$ for its homotopy category.
  It has the same objects as $\calC$ and $\Hom_{\h\calC}(x,y)=\pi_0\calC(x,y)$;
  \item we write $\frakC:\sSet\rightleftarrows\sCat:N$ for the ordinary rigidification--nerve adjunction;
  \item for $K\in\sSet$, the symbol $\bbA_{\dag}(K)$ denotes the free dagger arrow on $K$;
  \item we write $\bbJ_{\dag}$ for the walking dagger arrow generated by $a:0\to1$ and its formal dagger.
  We write $\bbI_{\dag}$ for the walking unitary isomorphism generated by $u:0\to1$, subject to $u^{\dag}u=\id_0$ and $uu^{\dag}=\id_1$.
  The quotient which sends $a$ to $u$ is denoted by $q:\bbJ_{\dag}\to\bbI_{\dag}$;
  \item we fix the functorial object-preserving replacement $R_{\rmob}:\sCat^{\dag}\to\sCat^{\dag}$ defined by $(R_{\rmob}\calC)(x,y) := \Ex^{\infty}\calC(x,y)$.
\end{itemize}

\subsection*{AI declaration}

In preparing this paper, the author used GPT-5.6 Sol to improve the English and identify linguistic errors. 
It was also used to identify relevant references, allowing the author to shorten proofs developed independently by the author by citing existing results. 
The author also used it to assist in developing and verifying several proofs in Section~6.1.
The author assumes full responsibility for all mathematical content.

\subsection*{Acknowledgements}

This work was supported by JST SPRING, Grant Number JPMJSP2109.

\section{Review of the dagger Bergner and dagger Joyal models}\label{review.section}

We recall here all notation, definitions, and results from \cite{Aka26} which are used below.

\subsection{Dagger simplicial categories and dagger simplicial sets}

\begin{definition}
  We recall some notions of dagger simplicial categories and dagger simplicial sets:
  \begin{itemize}
    \item a \emph{dagger simplicial category} is a simplicial category $\calC$ equipped with an identity-on-objects involutive simplicial functor $\calC\to\calC^{\myop}$;
    \item a \emph{dagger simplicial functor} is a local weak equivalence (resp. a local fibration) if it induces a weak homotopy equivalence (resp. a Kan fibration) on every mapping space;
    \item a \emph{dagger simplicial set} is a simplicial set $X$ equipped with an involutive natural isomorphism $X\to X^{\myop}$ which is the identity on vertices;
    \item a \emph{dagger morphism} of dagger simplicial sets is a map of simplicial sets which is compatible with dagger structures.
  \end{itemize}
\end{definition}

\begin{proposition}
  We have:
  \begin{itemize}
    \item restriction to order-preserving maps gives an equivalence 
    \begin{align}
      \sSet^{\dag} \simeq \Fun(\prism_{\rev}^{\myop},\Set);
    \end{align}
    \item under this equivalence, there exists an isomorphism 
    \begin{align}
      \yo_{\rev}[n] \cong \FdagSSet(\Delta^n);
    \end{align}
    \item the ordinary rigidification--nerve adjunction lifts to an adjunction 
    \begin{align}
      \frakC_{\dag} : \sSet^{\dag} \rightleftarrows \sCat^{\dag} : \N_{\dag}.
    \end{align}
    It satisfies the natural identities 
    \begin{align}
      \UdagSCat\frakC_{\dag}=\frakC\UdagSSet, 
      \quad 
      \UdagSSet \N_{\dag}=N\UdagSCat, 
      \quad \text{and} \quad 
      \frakC_{\dag}\FdagSSet \cong \FdagSCat\frakC.
    \end{align}
  \end{itemize}
\end{proposition}

\subsection{Natural dagger intervals and the dagger Bergner model structure}

\begin{theorem}[Fixed-object model structure]
  For every object set $O$, the category $\sCat^{\dag}_O$ has a right proper simplicial model structure whose weak equivalences and fibrations are the local weak equivalences and local fibrations, respectively.
\end{theorem}

\begin{definition}
  We recall some notions of coherently unitary morphisms:
  \begin{itemize}
    \item a \emph{natural dagger interval} is a factorization $\bbJ_{\dag} \xrightarrow{j_{\bbH}} \bbH \xrightarrow{p_{\bbH}} \bbI_{\dag}$ of $q$ in $\sCat^{\dag}_{\{0,1\}}$ such that $j_{\bbH}$ is a fixed-object cofibration, $p_{\bbH}$ is a local weak equivalence, and every mapping space of $\bbH$ is countable.
    We write $h_{\bbH}:=j_{\bbH}(a)$;
    \item for $x,y\in\Ob(\calC)$, let $\calC\langle x,y\rangle$ denote the dagger simplicial category with two formal objects $0$ and $1$, whose mapping spaces and composition are inherited from $\calC$ by labelling them by $x$ and $y$;
    \item a class $[v]:x\to y$ in $\h\calC$ is \emph{coherently unitary} if there are a natural dagger interval $\bbH$ and a dagger functor $\Phi:\bbH \to R_{\rmob}(\calC\langle x,y\rangle)$ such that $[\Phi(h_{\bbH})]=[r_{x,y}(v)]$ in $\pi_0R_{\rmob}(\calC\langle x,y\rangle)(0,1)$.
    We call $x$ and $y$ coherently unitarily equivalent if there exists a coherently unitary class from $x$ to $y$;
    \item a dagger simplicial functor $f:\calC\to\calD$ is a \emph{dagger DK-equivalence} if it is a locally weak homotopy equivalence and essentially coherently unitary surjective.
  \end{itemize}
\end{definition}

\begin{theorem}[Dagger Bergner model structure]
  There is a left proper combinatorial model structure $\sCat^{\dag}_{\Bergner}$ whose weak equivalences are the dagger DK-equivalences.
\end{theorem}

\begin{proposition}
  We shall also use the following facts:
  \begin{itemize}
    \item the functor $\UdagSCat$ creates all small colimits;
    \item for every simplicial set $K$, the free dagger arrow $\bbA_{\dag}(K)$ is characterized by
    \begin{align}
      \Hom_{\sCat^{\dag}}(\bbA_{\dag}(K),\calC)
      \cong
      \coprod_{x,y\in\Ob(\calC)}
      \Hom_{\sSet}(K,\calC(x,y));
    \end{align}
    \item after forgetting the dagger, a pushout along $\bbA_{\dag}(K)\to\bbA_{\dag}(L)$ is the composite of the ordinary free attachment of $K\to L$ at $(x,y)$ and its dagger-conjugate attachment at $(y,x)$;
    \item the functor $\UdagSCat$ sends dagger Bergner cofibrations to ordinary Bergner cofibrations.
  \end{itemize}
\end{proposition}

\subsection{The dagger Joyal model structure and the dagger Joyal--Bergner equivalence}

\begin{definition}
  We recall some notions of dagger Joyal weak equivalences and cofibrations:
  \begin{itemize}
    \item a morphism of dagger simplicial sets is a \emph{dagger Joyal equivalence} if its image under $\frakC_{\dag}$ is a dagger DK-equivalence;
    \item a monomorphism $X\to Y$ of dagger simplicial sets is a \emph{free cofibration} if no new nondegenerate simplex of positive dimension is fixed by the dagger;
    \item for later use, we fix the following generating cofibration set for the dagger Joyal model structure:
    \begin{align}
      I^{\sSet}_{\dag}
      :=\{\FdagSSet(\partial\Delta^n) \to \FdagSSet(\Delta^n) ~|~ n\geq0\}.
    \end{align}
  \end{itemize}
\end{definition}

\begin{theorem}[Dagger Joyal--Bergner equivalence]
  We have:
  \begin{itemize}
    \item there is a left proper combinatorial model structure $\sSet^{\dag}_{\Joyal}$  whose cofibrations are the free cofibrations and whose weak equivalences are the dagger Joyal equivalences;
    \item the lifted adjunction induces a Quillen equivalence
    \begin{align}
      \frakC_{\dag}:
      \sSet^{\dag}_{\Joyal}
      \rightleftarrows
      \sCat^{\dag}_{\Bergner}
      :\N_{\dag};
    \end{align}
    \item if $\calC$ is dagger Bergner fibrant, then the counit $\frakC_{\dag}\N_{\dag}\calC\to\calC$ is a dagger DK-equivalence;
    \item if $X$ is dagger Joyal cofibrant and $\frakC_{\dag}X\to R_{\rmob}\frakC_{\dag}X$ is an object-preserving fibrant replacement, then its adjoint $X\to \N_{\dag}R_{\rmob}\frakC_{\dag}X$ is a dagger Joyal equivalence.
  \end{itemize}
\end{theorem}

\section{The model structures on \texorpdfstring{$\Fun(\prism^{\myop},\sSet^{\dag})$}{Fun(Δ,sSet†)}}\label{cs.section}

In this section, we construct a resolution presentation using simplicial objects in dagger simplicial sets (\cref{cs.thm.equivalence}).

We first equip $\sSpace^{\dag\vee}=\Fun(\prism^{\myop},\sSet^{\dag})$ with its Reedy model structure and study level-evaluation adjunctions (\cref{cs.lem.evaluation-quillen}).
We then localize at maps forcing simplicial diagrams to be homotopically constant;
the resulting resolution model structure is Quillen equivalent to $\sSet^{\dag}_{\Joyal}$ through the constant--evaluation adjunction (\cref{cs.thm.equivalence}).

\subsection{The Reedy structure}

We introduce the external tensor $A\boxtimes K$ and the associated adjunctions $(-)\boxtimes\Delta^k\dashv\ev_k$ for simplicial objects in dagger simplicial sets (\cref{cs.cor.evk}).
After constructing the Reedy model structure, we prove that these adjunctions are Quillen  (\cref{cs.lem.evaluation-quillen}).

\begin{notation}\label{cs.not.SS}
  We let $\sSpace^{\dag\vee}:=\Fun(\prism^{\myop},\sSet^{\dag})$ denote the category of simplicial objects in dagger simplicial sets;
  via $\sSet^{\dag}\simeq\Fun(\prism_{\rev}^{\myop},\Set)$ it is equivalent to $\Fun((\prism_{\rev}\times\prism)^{\myop},\Set)$,
  i.e.\ to bisimplicial sets equipped with levelwise reversal involutions in the first coordinate.
\end{notation}

\begin{notation}
  We fix some notational conventions:
  \begin{itemize}
    \item for $W\in\sSpace^{\dag\vee}$ and $\alpha:[k]\to[l]$ in $\prism$, we let $W_k\in\sSet^{\dag}$ denote the $k$-th level of $W$, and $W(\alpha):W_l\to W_k$ denote the structure map induced by $\alpha$;
    \item for $A\in\sSet^{\dag}$ and $K\in\sSet$, we define
    \begin{align}
      (A\boxtimes K)_k:=A\cdot K_k=\textstyle\coprod_{K_k}A,
    \end{align}
    with structure maps relabelling the summands along the maps $K(\alpha):K_l\to K_k$;
    \item for the constant and the evaluation, we write
    \begin{align}
      c:=(-)\boxtimes\Delta^0:\sSet^{\dag}\to\sSpace^{\dag\vee}
      \quad \text{and} \quad 
      \ev_0(W):=W_0;
    \end{align}
    \item for $A\in\sSet^{\dag}$ and $W\in\sSpace^{\dag\vee}$, we define a simplicial set
    \begin{align}
      \Map(A,W):=[k\mapsto\Hom_{\sSet^{\dag}}(A,W_k)]\in\sSet,
    \end{align}
    with simplicial operators given by postcomposition with the structure maps of $W$.
  \end{itemize}
\end{notation}

\begin{lemma}\label{cs.lem.adjunction}
  For every $A\in\sSet^{\dag}$, $K\in\sSet$ and $W\in\sSpace^{\dag\vee}$, there exists a natural bijection
  \begin{align}\label{cs.eq.boxtimes}
    \Hom_{\sSpace^{\dag\vee}}(A\boxtimes K,W)
    \simeq
    \Hom_{\sSet}(K,\Map(A,W)).
  \end{align}
\end{lemma}

\begin{proof}
  A morphism $f:A\boxtimes K\to W$ consists of morphisms $f_n:\coprod_{K_n}A\to W_n$ in $\sSet^{\dag}$, i.e.\ of families $(f_{n,x}:A\to W_n)_{x\in K_n}$.
  For $\alpha:[m]\to[n]$ the naturality square for $\alpha$ commutes if and only if $W(\alpha)\circ f_{n,x}=f_{m,K(\alpha)(x)}$ for all $x\in K_n$.

  Define a morphism of simplicial sets $\varphi:K\to\Map(A,W)$ by 
  \begin{align}
    \varphi_n(x):=f_{n,x}\in\Hom_{\sSet^{\dag}}(A,W_n)=\Map(A,W)_n.
  \end{align}
  By the definition of the simplicial operators of $\Map(A,W)$, the compatibility 
  \begin{align}
    \varphi_m(K(\alpha)(x))= \Map(A,W)(\alpha)(\varphi_n(x)) =W(\alpha)\circ f_{n,x}
  \end{align}
  of $\varphi$ with $\alpha$ is exactly the naturality.
  Thus $f\mapsto\varphi$ is a bijection onto the set of simplicial maps, clearly natural in $A$, $K$ and $W$.
\end{proof}

\begin{corollary}\label{cs.cor.evk}
  For every $k\geq0$ there exists an adjunction $(-)\boxtimes\Delta^k\dashv\ev_k$:
  \begin{align}
    \Hom_{\sSpace^{\dag\vee}}(A\boxtimes\Delta^k,W)
    \simeq
    \Map(A,W)_k
    =
    \Hom_{\sSet^{\dag}}(A,W_k).
  \end{align}
\end{corollary}

\begin{proof}
  Take $K=\Delta^k$ in \cref{cs.lem.adjunction} and apply the Yoneda lemma to $\Hom_{\sSet}(\Delta^k,-) \simeq (-)_k$.
\end{proof}

\begin{lemma}\label{cs.lem.reedy}
  There exists a Reedy model structure $\sSpace^{\dag\vee}_{\Reedy}$ on $\sSpace^{\dag\vee}$, which is combinatorial and left proper;
  its weak equivalences are the levelwise weak equivalences.
\end{lemma}

\begin{proof}
  The category $\prism^{\myop}$ is a small Reedy category, while $\sSet^{\dag}_{\Joyal}$ is left proper and combinatorial by \cite[Theorem~3.5.5]{Aka26}.
  Therefore the Reedy model structure exists and its weak equivalences are the levelwise weak equivalences.
  It is left proper and combinatorial.
\end{proof}

\begin{proposition}\label{cs.lem.evaluation-quillen}
  For every $k\geq0$, the adjunction $(-)\boxtimes\Delta^k\dashv\ev_k$ of \cref{cs.cor.evk} induces a Quillen adjunction
  \begin{align}
    (-)\boxtimes\Delta^k : \sSet^{\dag}_{\Joyal} \rightleftarrows \sSpace^{\dag\vee}_{\Reedy} : \ev_k.
  \end{align}
\end{proposition}

\begin{proof}
  If $i$ is a cofibration (resp. a trivial cofibration) in $\sSet^{\dag}_{\Joyal}$, then the identification $i\boxtimes\Delta^k \cong (\varnothing\to\Delta^k)\mathbin{\Box'}i$ and \cite[Proposition~7.36]{JT06} show that $i\boxtimes\Delta^k$ is a Reedy cofibration (resp. a Reedy trivial cofibration).
  Hence it is left Quillen.
\end{proof}

\subsection{The resolution model structure}

We localize the Reedy structure at the maps $A\boxtimes\Delta^k\to A\boxtimes\Delta^0$, which impose homotopical constancy on fibrant objects (\cref{cs.thm.resolution}).
By identifying this localization with Dugger's Reedy hocolim model structure, we characterize its fibrant objects and prove that $c\dashv\ev_0$ is a Quillen equivalence with the dagger Joyal model structure (\cref{cs.thm.equivalence}).

\begin{notation}\label{cs.not.Gh}
  We set
  \begin{align}
    G_{\dag}
    &:=\{\FdagSSet(\partial\Delta^n),\FdagSSet(\Delta^n)\}_{n\geq0},\\
    S_{h}
    &:=\{A\boxtimes\sigma_k : A\boxtimes\Delta^k\to A\boxtimes\Delta^0 \mid A\in G_{\dag},\ k\geq1\}.
  \end{align}
\end{notation}

\begin{remark}
  Every $A\in G_{\dag}$ is free, hence cofibrant in $\sSet^{\dag}_{\Joyal}$ by the recalled free-cofibration statement.
  Since $(-)\boxtimes\Delta^k$ and $(-)\boxtimes\Delta^0$ are left Quillen (by \cref{cs.lem.evaluation-quillen}), the domains and codomains of the maps of $S_h$ are Reedy cofibrant.
\end{remark}

\begin{lemma}[The resolution model structure]\label{cs.def.Sh}
  There exists a left proper combinatorial model structure $\sSpace^{\dag\vee}_{\CSS}:=L_{S_h}\sSpace^{\dag\vee}_{\Reedy}$ on $\sSpace^{\dag\vee}$: 
  it has the same cofibrations as the Reedy structure, its weak equivalences are the $S_h$-local equivalences, and its fibrant objects are the Reedy fibrant $S_h$-local objects.
\end{lemma}

\begin{proof}
  Let $H_h:=\{[s]\mid s\in S_h\} \subseteq \Mor\Ho(\sSpace^{\dag\vee}_{\Reedy})$ be the set of homotopy classes represented by $S_h$.
  This is a small set because $S_h$ is a set.
  The Reedy structure is left proper and combinatorial (by \cref{cs.lem.reedy}), so its left Bousfield localization at $H_h$ exists and is left proper and combinatorial by \cite[Theorem~4.7]{Bar10}.
  
  The $H_h$-local objects and equivalences are precisely the $S_h$-local objects and equivalences, since they depend only on the represented homotopy classes.
  Thus this localization is the stated model structure $L_{S_h}\sSpace^{\dag\vee}_{\Reedy}$.
\end{proof}

\begin{proposition}\label{cs.thm.resolution}
  An object $W\in\sSpace^{\dag\vee}_{\CSS}$ is fibrant if and only if it is Reedy fibrant and \emph{homotopically constant}:
  for every $k$, the morphism $W(\sigma_k):W_0\to W_k$ induced by the unique map $\sigma_k:[k]\to[0]$ is a weak equivalence in $\sSet^{\dag}_{\Joyal}$.
  
  In that case, the canonical morphism $\ell_W:cW_0\to W$ with $(\ell_W)_k=W(\sigma_k)$ is a levelwise weak equivalence, and $W_0$ is fibrant.
\end{proposition}

\begin{proof}
  By \cite[the proof of Proposition~A.5]{Dug01R}, the detecting set used in Dugger's construction can be chosen to be $G_{\dag}$: 
  its objects are the already-cofibrant domains and codomains of the generating cofibrations $I^{\sSet}_{\dag}$.
  Let
  \begin{align}
    S_{\D}
    :=
    \{A\boxtimes\alpha:
      A\boxtimes\Delta^j\to A\boxtimes\Delta^i
      \mid A\in G_{\dag},\ \alpha:[j]\to[i]\text{ in }\prism\}.
  \end{align}
  Thus $S_h$ is the subfamily of $S_{\rmD}$ corresponding to the nonidentity arrows $[0]\to[k]$ of $\prism^{\myop}$.

  We first show that $S_h$ and $S_{\rmD}$ define the same local objects.
  Let $W$ be Reedy fibrant and $S_h$-local.
  Fix $A\in G_{\dag}$ and $\alpha:[j]\to[i]$ in $\prism$.
  Since $\sigma_j=\sigma_i\circ\alpha$, we have $s_j=A\boxtimes\sigma_{i}\circ A\boxtimes\alpha$.
  Consequently, in $\Ho(\sSet)$ the morphisms induced on homotopy function complexes satisfy
  \begin{align}
    [A\boxtimes\sigma_{j}]^*=[A\boxtimes\alpha]^*\circ[A\boxtimes\sigma_{i}]^*.
  \end{align}
  By $S_h$-locality, the morphisms $[A\boxtimes\sigma_{i}]^*$ and $[A\boxtimes\sigma_{j}]^*$ are isomorphisms in $\Ho(\sSet)$.
  Hence $[A\boxtimes\alpha]^*$ is an isomorphism by two-out-of-three.
  Equivalently, the induced morphism of homotopy function complexes is a weak equivalence, so $W$ is $S_{D}$-local.

  The converse follows because $S_h\subseteq S_{D}$.
  Thus the two localizations have the same cofibrations and the same local objects, and therefore the same local equivalences.
  Consequently $L_{S_h}\sSpace^{\dag\vee}_{\Reedy}$ is the Reedy hocolim model structure of \cite[Theorem~5.7]{Dug01R}.

  By \cite[Theorem~5.7 (c)]{Dug01R}, its fibrant objects are precisely the Reedy fibrant objects for which every structure map is a weak equivalence in $\sSet^{\dag}_{\Joyal}$.
  This is equivalent to the condition in the statement.
  Indeed, one implication is immediate, while for $\alpha:[m]\to[n]$ in $\prism$ the identity $\sigma_m=\sigma_n\circ\alpha$ gives $W(\sigma_m)=W(\alpha)\circ W(\sigma_n)$, so the converse follows from two-out-of-three.

  Finally, the maps $(\ell_W)_k:=W(\sigma_k)$ define a morphism $\ell_W:cW_0\to W$ and make it a levelwise weak equivalence.
  Since $\ev_0$ is right Quillen (by \cref{cs.lem.evaluation-quillen}), $W_0$ is fibrant.
\end{proof}

\begin{theorem}\label{cs.thm.equivalence}
  The adjunction $c\dashv\ev_0$ induces a Quillen equivalence
  \begin{align}
    c:
    \sSet^{\dag}_{\Joyal}
    \rightleftarrows
    \sSpace^{\dag\vee}_{\CSS}
    :\ev_0 .
  \end{align}
\end{theorem}

\begin{proof}
  The proof of \cref{cs.thm.resolution} identifies $\sSpace^{\dag\vee}_{\CSS}$ with Dugger's Reedy hocolim model structure on simplicial objects in $\sSet^{\dag}_{\Joyal}$.
  Therefore the constant--evaluation adjunction is a Quillen equivalence by \cite[Theorem~6.1]{Dug01R}.
\end{proof}

\section{The model structures on \texorpdfstring{$\Fun(\prism_{\rev}^{\myop},\sSet)$}{Fun(Δ,sSet)} and \texorpdfstring{$\Fun(\prism_{\dag}^{\myop},\sSet)$}{Fun(Δ,sSet)}}\label{h1.section}

In this section, we construct two model-categorical presentations of unitarily complete dagger Segal spaces, indexed respectively by the reversal-extended simplex category $\prism_{\rev}$ and by the category $\prism_{\dag}$ of free dagger simplices (\cref{h1.thm.main}).

We first localize the projective and injective diagram model structures on $\sSpace^{\dag}_{\rev}$ at the dagger Segal and unitary completeness maps (\cref{h1.lem.rev-models}).
We then compare the two indexing categories through the change-of-index adjunction $u_{!}\dashv u^{*}$ (\cref{h1.thm.adjunction}).
The main result is that this adjunction is a Quillen equivalence (\cref{h1.thm.main}).

\subsection{The model structures on \texorpdfstring{$\Fun(\prism_{\rev}^{\myop},\sSet)$}{Fun(Δ,sSet)}}

We equip $\sSpace^{\dag}_{\rev}$ with its projective and injective model structures and localize them at the dagger Segal maps and the unitary completeness map.
We characterize the injectively fibrant local objects and the projectively fibrant local objects (\cref{h1.lem.rev-models}).

\begin{notation}\label{h1.not.discrete}
  We write 
  \begin{itemize}
    \item $\sSpace^{\dag}_{\rev}:=\Fun(\prism_{\rev}^{\myop},\sSet)$;
    \item $\delta:\sSet^{\dag}\hookrightarrow\sSpace^{\dag}_{\rev}$ for the postcomposition with the discrete embedding $\Set\hookrightarrow\sSet$.
    \item $\Delta^n_{\rev}:=\delta(\FdagSSet(\Delta^n))=\yo_{\rev}[n]$ for the representable dagger simplicial space;
    \item $\underline{\Map}(X,W):=[k\mapsto\Hom_{\sSpace^{\dag}_{\rev}}(X\times\Delta^k,W)]$ denotes the simplicial mapping space for the levelwise simplicial structure, $(X\times K)_{[n]}:=X_{[n]}\times K$.
  \end{itemize}
\end{notation}

\begin{remark}
  By definition, $\delta X_{[n]}$ is the set $X_n$ regarded as a discrete simplicial set.
  The functor $\delta$ is fully faithful and preserves limits and colimits (the discrete embedding has both adjoints, $\pi_0$ and the $0$-th level functor);
\end{remark}

\begin{lemma}\label{h1.lem.proj-rev}
  We have:
  \begin{enumerate}
    \item the projective model structure $\sSpace^{\dag}_{\rev,\proj}$ on $\sSpace^{\dag}_{\rev}$, with levelwise weak homotopy equivalences and levelwise Kan fibrations, exist and are combinatorial and left proper; 
    their generating (trivial) cofibrations are the maps $\yo(c)\times(\partial\Delta^k\to\Delta^k)$ (resp.\ $\yo(c)\times(\Lambda^k_i\to\Delta^k)$) for $c \in \prism_{\rev}^{\myop}$;
    \item the injective model structure $\sSpace^{\dag}_{\rev,\inj}$ on $\sSpace^{\dag}_{\rev}$, with levelwise weak equivalences and monomorphisms, exist and are combinatorial and left proper, with all objects cofibrant;
    \item both model structures are simplicial for the levelwise tensor, cotensor and mapping space.
  \end{enumerate}
\end{lemma}

\begin{proof}
  Regard $\prism_{\rev}^{\myop}$ as a simplicial category with discrete mapping spaces.
  The Kan model structure on $\sSet$ is excellent by \cite[Example~A.3.2.18]{HTT}.
  Existence, combinatoriality and simpliciality of the projective and injective diagram model structures follow from \cite[Proposition~A.3.3.2 and Remark~A.3.3.4]{HTT}.
  Their left properness follows from \cite[Remark~A.2.8.4]{HTT}.
  
  (1)
  By \cite[Remark~A.2.8.5]{HTT}, the projective cofibrations are generated by the maps
  $\yo(c)\times(\partial\Delta^k\to\Delta^k)$.
  Since these maps are levelwise monomorphisms, every projective cofibration is an injective cofibration.
  The same construction gives the displayed generating projective trivial cofibrations.
  
  (2)
  Every object is injectively cofibrant because $\varnothing\to X$ is levelwise a monomorphism, while
  $\varnothing\to\yo(c)=\yo(c)\times(\varnothing\to\Delta^0)$ is a generating projective cofibration;
  hence every representable is projectively cofibrant.
\end{proof}

\begin{definition}
  We write 
  \begin{align}
    \Sp[n]_{\rev}:=\delta(\FdagSSet(\Sp[n])),
    \quad \text{where} \quad 
    \Sp[n]:=\Delta^1\amalg_{\Delta^0}\cdots\amalg_{\Delta^0}\Delta^1\subseteq\Delta^n,
  \end{align}
  with the conventions $\Sp[0]=\Delta^0$ and $\Sp[1]=\Delta^1$.

  The \emph{Segal maps} are the monomorphisms $\Sp[n]_{\rev}\to\Delta^n_{\rev}$ ($n\geq2$) obtained by applying $\delta\circ \FdagSSet$ to the spine inclusions;
\end{definition}

\begin{definition}
  We write 
  \begin{align}
    E_{\rev}:=\delta(E_{\dag}),
    \quad \text{where} \quad 
    E_{\dag}:=\N_{\dag}(\bbI_{\dag})\in\sSet^{\dag}.
  \end{align}

  The \emph{completeness map} is $\delta(\epsilon_0):\Delta^0_{\rev}\to E_{\rev}$, where $\epsilon_0:\Delta^0_{\dag}\to E_{\dag}$ classifies, via the adjunction $\FdagSSet\dashv \UdagSSet$, the vertex $0$ of $\UdagSSet E_{\dag}=N(\UdagSCat\bbI_{\dag})$;
\end{definition}

\begin{notation}\label{h1.not.S}
  We let $S_{\rev}$ denote the set consisting of the Segal maps $\Sp[n]_{\rev}\to\Delta^n_{\rev}$ ($n\geq2$) and the completeness map $\Delta^0_{\rev}\to E_{\rev}$;
  
  We let $\widetilde S_{\rev}$ denote a set of \emph{projective cofibrant resolutions}: 
  for each $s:A\to B$ in $S_{\rev}$, a projective cofibration $\widetilde s:\widetilde A\to\widetilde B$ between projectively cofibrant objects together with levelwise weak equivalences $\widetilde A\to A$, $\widetilde B\to B$ commuting with $s$.
  Such resolutions are obtained by functorial cofibrant replacement of $s$ in the projective arrow model structure;
\end{notation}

\begin{remark}\label{h1.rem.S-comparison}
  An object is $\widetilde S_{\rev}$-local if and only if it is $S_{\rev}$-local, and the $\widetilde S_{\rev}$- and $S_{\rev}$-local equivalences coincide.  
  Indeed, for each $s:A\to B$ and its projective cofibrant resolution $\widetilde s:\widetilde A\to\widetilde B$, the defining commutative square has vertical levelwise weak equivalences.  
  
  After replacing the target by a levelwise fibrant object, invariance of homotopy function complexes in the source identifies the two induced maps.  
  The same identification, now used in the target variable, shows that the two classes of local equivalences agree (see \cite[\S17.4]{Hir02}).
\end{remark}

\begin{definition}\label{h1.def.segal-complete}
  Let $W\in\sSpace^{\dag}_{\rev}$ be injectively fibrant.
  We will say that it is 
  \begin{itemize}
    \item \emph{dagger Segal space} if for every $n\geq2$ the Segal map induces a weak equivalence
    \begin{align}
      W_n
      =
      \underline{\Map}(\Delta^n_{\rev},W)
      \to
      \underline{\Map}(\Sp[n]_{\rev},W)
      \simeq
      W_1\times_{W_0}\cdots\times_{W_0}W_1,
    \end{align}
    \item \emph{unitarily complete} if the completeness map induces a weak equivalence
    \begin{align}
      \underline{\Map}(E_{\rev},W)\to\underline{\Map}(\Delta^0_{\rev},W)=W_0 .
    \end{align}
  \end{itemize}
\end{definition}

\begin{definition}
  A levelwise fibrant object of $\sSpace^{\dag}_{\rev}$ is called a \emph{unitarily complete dagger Segal space in the derived sense} if some (equivalently, any) injectively fibrant replacement of it along a levelwise weak equivalence is one.
\end{definition}

\begin{lemma}\label{h1.lem.local-objects}
  We have:
  \begin{enumerate}
    \item for injectively fibrant $W \in \sSpace^{\dag}_{\rev,\inj}$ and any $X$, the simplicial set $\underline{\Map}(X,W)$ is a Kan complex computing the homotopy function complex, and a monomorphism $X\to Y$ induces a Kan fibration $\underline{\Map}(Y,W)\to\underline{\Map}(X,W)$;
    \item we have 
    \begin{align}
      \underline{\Map}(\Delta^n_{\rev},W)\simeq W_n 
      \quad \text{and} \quad 
      \underline{\Map}(\Sp[n]_{\rev},W)\simeq W_1\times_{W_0}\cdots\times_{W_0}W_1,
    \end{align}
    and the strict fiber products are homotopy fiber products when $W$ is injectively fibrant;
    \item an injectively fibrant $W$ is $S_{\rev}$-local if and only if it is a unitarily complete dagger Segal space;
    a levelwise fibrant $W$ is $S_{\rev}$-local with respect to the projective model structure, equivalently $\widetilde S_{\rev}$-local in that structure, if and only if it is one in the derived sense.
  \end{enumerate}
\end{lemma}

\begin{proof}
  (1)
  It follow from \cite[\S9.1, \S17.2]{Hir02}.

  (2)
  By the adjunction \eqref{cs.eq.boxtimes} for the levelwise simplicial structure, we have 
  \begin{align}
    \Hom(\Delta^n_{\rev}\times\Delta^k,W)\simeq(W_n)_k,
  \end{align}
  giving the first identification.
  The functor $\delta\circ \FdagSSet$ preserves colimits, and $\Sp[n]=\Delta^1\amalg_{\Delta^0}\cdots\amalg_{\Delta^0}\Delta^1$; 
  since $\underline{\Map}(-,W)$ carries colimits to limits, we have 
  \begin{align}
    \underline{\Map}(\Sp[n]_{\rev},W)\simeq W_1\times_{W_0}\cdots\times_{W_0}W_1.
  \end{align}
  When $W$ is injectively fibrant, the vertex evaluations $W_1\to W_0$ are Kan fibrations by (1), the vertex inclusions $\Delta^0_{\rev}\to\Delta^1_{\rev}$ being monomorphisms; a strict fiber product along Kan fibrations of Kan complexes is a homotopy fiber product.

  (3)
  For injectively fibrant $W$, locality with respect to a monomorphism between injectively cofibrant objects is tested by $\underline{\Map}(-,W)$ by (1).
  Locality with respect to the Segal maps is then exactly the dagger Segal condition, and locality with respect to the completeness map is unitary completeness by (2).
  
  For merely levelwise fibrant $W$, choose an injectively fibrant replacement $r:W\to W'$ along a levelwise weak equivalence; 
  homotopy function complexes out of any fixed object are invariant under $r$, so $W$ is $S_{\rev}$-local if and only if $W'$ is.
\end{proof}

\begin{proposition}\label{h1.lem.rev-models}
  We have:
  \begin{enumerate}
    \item the identity functors induce a Quillen equivalence
    \begin{align}
      \id:
      L_{\widetilde S_{\rev}}\sSpace^{\dag}_{\rev,\proj}
      \rightleftarrows
      L_{S_{\rev}}\sSpace^{\dag}_{\rev,\inj}
      :\id;
    \end{align}
    \item the fibrant objects of $L_{\widetilde S_{\rev}}\sSpace^{\dag}_{\rev,\proj}$ are the levelwise fibrant objects which are unitarily complete dagger Segal spaces in the derived sense;
    \item those of $L_{S_{\rev}}\sSpace^{\dag}_{\rev,\inj}$ are the injectively fibrant unitarily complete dagger Segal spaces.
  \end{enumerate}
\end{proposition}

\begin{proof}
  (1)
  By \cref{h1.lem.proj-rev}, the two underlying model structures are left proper, combinatorial and simplicial.
  Every map in $\widetilde S_{\rev}$ is a projective cofibration by construction, while every map in $S_{\rev}$ is a monomorphism and hence an injective cofibration.
  Therefore \cite[Proposition~A.3.7.3]{HTT} gives both localizations as left proper combinatorial simplicial model categories.

  The identity functor from the projective to the injective model structure is a Quillen equivalence by \cite[Remark~A.2.8.6]{HTT}.
  Every map $\widetilde s\in\widetilde S_{\rev}$ has projectively cofibrant source and target, so its left-derived image under the identity is represented by $\widetilde s$ itself.
  The defining square from $\widetilde s$ to the corresponding $s\in S_{\rev}$ has levelwise weak equivalences;
  hence the derived image of $\widetilde s$ and $s$ represent the same morphism in $\Ho(\sSpace^{\dag}_{\rev,\inj})$, compatibly with \cref{h1.rem.S-comparison}.
  Therefore \cite[Theorem~3.3.20 (1)(b)]{Hir02} induces the displayed Quillen equivalence between the localizations.

  (2) and (3)
  By \cite[Proposition~A.3.7.3 (3)]{HTT}, the fibrant objects of the localizations are the fibrant local objects.
  Their stated descriptions follow from \cref{h1.lem.local-objects} (3).
\end{proof}

\begin{corollary}\label{h1.lem.localized-simplicial}
  We have:
  \begin{enumerate}
    \item the localized injective model category $L_{S_{\rev}}\sSpace^{\dag}_{\rev,\inj}$ is a simplicial model category for the levelwise tensor, cotensor and mapping space of \cref{h1.not.discrete};
    \item every object is cofibrant in $L_{S_{\rev}}\sSpace^{\dag}_{\rev,\inj}$, and for every localized fibrant $Z$ and every $X$, the Kan complex $\underline{\Map}(X,Z)$ computes the localized homotopy function complex $\RMap(X,Z)$.
  \end{enumerate}
\end{corollary}

\begin{proof}
  (1)
  In the proof of \cref{h1.lem.rev-models}, $L_{S_{\rev}}\sSpace^{\dag}_{\rev,\inj}$ is obtained by applying \cite[Proposition~A.3.7.3]{HTT} to the levelwise simplicial model structure.
  It is therefore a simplicial model category for the same levelwise tensor, cotensor and mapping space.

  (2)
  Since the cofibrations are unchanged, every object remains cofibrant.
  Consequently, for localized fibrant $Z$, the simplicial mapping space $\underline{\Map}(X,Z)$ is a Kan complex and computes $\RMap(X,Z)$.
\end{proof}

\subsection{The model structures on \texorpdfstring{$\Fun(\prism_{\dag}^{\myop},\sSet)$}{Fun(Δ,sSet)}}

We introduce the category $\prism_{\dag}$ of free dagger simplices and the associated diagram category $\sSpace^{\dag}$.

\begin{notation}
  We fix some notational conventions:
  \begin{itemize}
    \item for $n\geq0$, let $[n]_{\dag}:=\FdagSCat([n])$ where $[n]$ is regarded as a discrete simplicial category.
    Thus $[n]_{\dag}$ is the free dagger category on the chain $0\to1\to\cdots\to n$;
    \item we let $\prism_{\dag}$ denote the full subcategory of $\sCat^{\dag}$ whose objects are the dagger categories $[n]_{\dag}$ for $n\geq0$.
    In particular, the morphisms of $\prism_{\dag}$ are all dagger functors between free dagger simplices;
    \item we write $\yo_{\dag}:\prism_{\dag} \to \Fun(\prism_{\dag}^{\myop},\Set)$ for its Yoneda embedding;
    \item we write $\sSpace^{\dag}:=\Fun(\prism_{\dag}^{\myop},\sSet)$;
  \end{itemize}
\end{notation}

\begin{lemma}\label{h1.lem.proj-dag}
  We have:
  \begin{enumerate}
    \item the projective model structure $\sSpace^{\dag}_{\proj}$ on $\sSpace^{\dag}$, with levelwise weak homotopy equivalences and levelwise Kan fibrations, exist and are combinatorial and left proper; 
    their generating (trivial) cofibrations are the maps $\yo(c)\times(\partial\Delta^k\to\Delta^k)$ (resp.\ $\yo(c)\times(\Lambda^k_i\to\Delta^k)$) for $c \in \prism_{\dag}$;
    \item the injective model structure $\sSpace^{\dag}_{\inj}$ on $\sSpace^{\dag}$, levelwise weak equivalences and monomorphisms, exist and are combinatorial and left proper, with all objects cofibrant;
    \item both model structures are simplicial for the levelwise tensor, cotensor and mapping space.
  \end{enumerate}
\end{lemma}

\begin{proof}
  We can prove it as \cref{h1.lem.proj-rev}.
\end{proof}

\subsection{The change-of-index functor}

We next compare the reversal category $\prism_{\rev}$ with the category $\prism_{\dag}$ of free dagger simplices.
We construct the localized adjunction $u_{!}\dashv u^{*}$ (\cref{h1.thm.adjunction}) and prove that its right derived functor is conservative (\cref{h1.thm.conservative}).

\begin{definition}\label{h1.not.u}
  We define a functor $u:\prism_{\rev}\to\prism_{\dag}$ as follows:
  \begin{itemize}
    \item on objects, $u([n])=[n]_{\dag}$;
    \item for an order-preserving $\alpha:[m]\to[n]$, the dagger functor $u(\alpha):[m]_{\dag}\to[n]_{\dag}$ acts as $\alpha$ on objects and sends the generator $a_i:i-1\to i$ to the forward word $a_{\alpha(i)}\cdots a_{\alpha(i-1)+1}$ (the identity if $\alpha(i-1)=\alpha(i)$);
    \item for an order-reversing $\gamma:[m]\to[n]$, the dagger functor $u(\gamma)$ acts as $\gamma$ on objects and sends $a_i$ to the dagger $(a_{\gamma(i-1)}\cdots a_{\gamma(i)+1})^{\dag}$ of the corresponding forward word.
  \end{itemize}
\end{definition}

\begin{remark}\label{h1.lem.u}
  The functor $u:\prism_{\rev}\to\prism_{\dag}$ is bijective on objects and faithful:
  its action on objects recovers the underlying map of finite ordinals, so faithfulness follows, and bijectivity on objects is immediate from the definition of $\prism_{\dag}$.
\end{remark}

\begin{lemma}\label{h1.lem.adjoints}
  We have:
  \begin{enumerate}
    \item the restriction $u^{*}:\sSpace^{\dag}\to\sSpace^{\dag}_{\rev}$ admits both adjoints, $u_{!}\dashv u^{*}\dashv u_{*}$.
    Moreover, $(u^{*}W)_{[n]}=W_{[n]_{\dag}}$; since $u$ is bijective on objects, $u^{*}$ preserves and reflects all levelwise notions (weak equivalences, Kan fibrations, monomorphisms):
    \item $u_{!}\circ\yo_{\rev}\simeq\yo_{\dag}\circ u$; in particular $u_{!}(\yo_{\rev}[n])\simeq\Delta^\N_{\dag}:=\yo_{\dag}([n]_{\dag})$.
  \end{enumerate}
\end{lemma}

\begin{proof}
  (1)
  The adjoints are the pointwise left and right Kan extensions along $u^{\myop}$.
  The formula $(u^{*}W)_{[n]}=W_{[n]_{\dag}}$ is the definition of restriction; 
  since every object of $\prism_{\dag}$ is of the form $u([n])$, a morphism $f \in \sSpace^{\dag}$ is a levelwise weak equivalence (fibration, monomorphism) if and only if $u^{*}f$ is.

  (2)
  For every $W\in\sSpace^{\dag}$, there is a bijection
  \begin{align}
    \Hom_{\sSpace^{\dag}}(u_{!}\yo_{\rev}[n],W)
    &\cong
    \Hom_{\sSpace^{\dag}_{\rev}}(\yo_{\rev}[n],u^{*}W) \\
    &\cong
    ((u^{*}W)_{[n]})_0 \\
    &=
    (W_{[n]_{\dag}})_0 \\
    &\cong
    \Hom_{\sSpace^{\dag}}(\yo_{\dag}([n]_{\dag}),W).
  \end{align}
  This bijection is natural in $W$.
  Hence the Yoneda lemma gives a canonical isomorphism $u_{!}\yo_{\rev}[n]\cong\yo_{\dag}([n]_{\dag})$.
\end{proof}

\begin{lemma}\label{h1.lem.quillen-proj}
  The adjunction $u_{!}\dashv u^{*}$ induces a Quillen adjunction
  \begin{align}
    u_{!}:\sSpace^{\dag}_{\rev,\proj}\rightleftarrows\sSpace^{\dag}_{\proj}:u^{*}.
  \end{align}
\end{lemma}

\begin{proof}
  By \cref{h1.lem.adjoints} (1), the functor $u^{*}$ preserves levelwise Kan fibrations and levelwise trivial Kan fibrations, 
  i.e. the fibrations and trivial fibrations of the projective structures.
\end{proof}

\begin{lemma}\label{h1.def.model}
  There exist left proper combinatorial simplicial model structures $L_{u_{!}(\widetilde S_{\rev})}\sSpace^{\dag}_{\proj}$ and $L_{u_{!}(\widetilde S_{\rev})}\sSpace^{\dag}_{\inj}$ on $\sSpace^{\dag}$.
\end{lemma}

\begin{proof}
  The projective and injective model structures $\sSpace^{\dag}_{\proj}$ and $\sSpace^{\dag}_{\inj}$ are left proper, combinatorial and simplicial by \cref{h1.lem.proj-dag}.
  Every map in $u_{!}(\widetilde S_{\rev})$ is a projective cofibration because the maps of $\widetilde S_{\rev}$ are projective cofibrations and $u_{!}$ is left Quillen (by \cref{h1.lem.quillen-proj}).
  It is therefore also an injective cofibration by \cref{h1.lem.proj-dag}.
  Hence \cite[Proposition~A.3.7.3]{HTT} gives both localizations as left proper combinatorial simplicial model categories.

  They have the same cofibrations as their respective underlying model structures.
  Their local objects and local equivalences are those defined using the displayed set of representatives.
  In particular, every object of $L_{u_{!}(\widetilde S_{\rev})}\sSpace^{\dag}_{\inj}$ is cofibrant.

  We note for later use that the domains and codomains of the maps of $u_{!}(\widetilde S_{\rev})$ are projectively cofibrant: those of $\widetilde S_{\rev}$ are, and $u_{!}$ is left Quillen (by \cref{h1.lem.quillen-proj}), hence preserves cofibrant objects.
\end{proof}

\begin{proposition}\label{h1.prop.fibrant}
  An object $W\in\sSpace^{\dag}$ is fibrant in $L_{u_{!}(\widetilde S_{\rev})}\sSpace^{\dag}_{\proj}$ if and only if it is levelwise fibrant and $u^{*}W$ is fibrant in $L_{\widetilde S_{\rev}}\sSpace^{\dag}_{\rev,\proj}$, 
  i.e.\ if and only if $W$ is levelwise fibrant and its $\prism_{\rev}$-restriction is a unitarily complete dagger Segal space in the derived sense. 
\end{proposition}

\begin{proof}
  Fibrancy in $L_{u_{!}(\widetilde S_{\rev})}\sSpace^{\dag}_{\proj}$ means being projectively fibrant, hence levelwise fibrant, and $u_{!}(\widetilde S_{\rev})$-local.
  For levelwise fibrant $W$, $u^{*}W$ is levelwise fibrant.
  
  Since $u_{!}\dashv u^{*}$ is a Quillen adjunction (by \cref{h1.lem.quillen-proj}) and every source and target of a map in $\widetilde S_{\rev}$ is projectively cofibrant, the derived adjunction for homotopy function complexes \cite[Proposition~17.4.16]{Hir02} gives natural weak equivalences
  \begin{align}
    \RMap_{\sSpace^{\dag}_{\proj}}(u_{!}\widetilde A,W)
    \simeq
    \RMap_{\sSpace^{\dag}_{\rev,\proj}}(\widetilde A,u^{*}W)
  \end{align}
  for every source or target $\widetilde A$ of a map in $\widetilde S_{\rev}$.
  These equivalences are natural in $\widetilde A$, so $W$ is $u_{!}(\widetilde S_{\rev})$-local if and only if $u^{*}W$ is $\widetilde S_{\rev}$-local.
  The claimed description now follows from \cref{h1.lem.local-objects} (3) and \cref{h1.lem.rev-models}.
\end{proof}

\begin{lemma}\label{h1.thm.adjunction}
  The adjunction $u_{!}\dashv u^{*}$ induces a Quillen adjunction
  \begin{align}
    u_{!}:
    L_{\widetilde S_{\rev}}\sSpace^{\dag}_{\rev,\proj}
    \rightleftarrows
    L_{u_{!}(\widetilde S_{\rev})}\sSpace^{\dag}_{\proj}
    :u^{*}.
  \end{align}
\end{lemma}

\begin{proof}
  The adjunction $u_{!}\dashv u^{*}$ is Quillen for $\sSpace^{\dag}_{\rev,\proj}$ and $\sSpace^{\dag}_{\proj}$ (by \cref{h1.lem.quillen-proj}).
  The sources and targets of the maps in $\widetilde S_{\rev}$ are projectively cofibrant, so their left-derived images are represented by the actual maps $u_{!}(\widetilde s)$.
  These maps are weak equivalences in $L_{u_{!}(\widetilde S_{\rev})}\sSpace^{\dag}_{\proj}$ by definition of the localization.
  Hence \cite[Theorem~3.3.20 (1)(a)]{Hir02} gives the displayed Quillen adjunction between the localized model structures.
\end{proof}

\begin{proposition}\label{h1.thm.conservative}
  The total derived right adjoint
  \begin{align}
    \bfR u^{*}:
    \Ho(L_{u_{!}(\widetilde S_{\rev})}\sSpace^{\dag}_{\proj})
    \to
    \Ho(L_{\widetilde S_{\rev}}\sSpace^{\dag}_{\rev,\proj})
  \end{align}
  is conservative.
\end{proposition}

\begin{proof}
  Let $\alpha:W\to W'$ be a morphism in $\Ho(L_{u_{!}(\widetilde S_{\rev})}\sSpace^{\dag}_{\proj})$ such that $\bfR u^{*}(\alpha)$ is invertible.
  Choose bifibrant representatives $\widehat W$ and $\widehat W'$ of $W$ and $W'$, and transport $\alpha$ along the resulting isomorphisms in the homotopy category.
  Since $\widehat W$ is cofibrant and $\widehat W'$ is fibrant, the transported morphism is represented by an actual map $f:\widehat W\to\widehat W'$.

  By \cref{h1.thm.adjunction}, the functor $u^{*}$ is right Quillen for the localized model structures.
  Hence $u^{*}\widehat W$ and $u^{*}\widehat W'$ are fibrant in $L_{\widetilde S_{\rev}}\sSpace^{\dag}_{\rev,\proj}$, and $u^{*}f$ represents $\bfR u^{*}(\alpha)$.
  Since this derived morphism is invertible, $u^{*}f$ is a weak equivalence in $L_{\widetilde S_{\rev}}\sSpace^{\dag}_{\rev,\proj}$.

  By \cref{h1.prop.fibrant}, the source and target of $u^{*}f$ are $\widetilde S_{\rev}$-local and fibrant in $\sSpace^{\dag}_{\rev,\proj}$.
  \Cite[Theorem~3.2.13 (1)]{Hir02} applied in $\sSpace^{\dag}_{\rev,\proj}$, therefore implies that $u^{*}f$ is a weak equivalence in that underlying model structure.
  Thus $u^{*}f$ is a levelwise weak equivalence.

  Since $u$ is bijective on objects, restriction along $u$ reflects levelwise weak equivalences by \cref{h1.lem.adjoints} (1).
  It follows that $f$ is a weak equivalence in $\sSpace^{\dag}_{\proj}$.
  Every weak equivalence of the underlying model structure remains a weak equivalence after left Bousfield localization, so $f$ is a weak equivalence in $L_{u_{!}(\widetilde S_{\rev})}\sSpace^{\dag}_{\proj}$.
  Hence $\alpha$ is invertible, and $\bfR u^{*}$ is conservative.
\end{proof}

\subsection{The comparison is a Quillen equivalence}

We have shown that the total derived right adjoint $\bfR u^{*}$ is conservative (\cref{h1.thm.conservative}).
To prove that $u_{!}\dashv u^{*}$ is a Quillen equivalence (\cref{h1.thm.main}), it remains to show that the derived unit is a local equivalence (\cref{h1.prop.derived-unit}).

The proof first analyzes the unit on representables (\cref{h1.prop.classifying}).
Tensor compatibility and a projective-cell induction then extend this result to every projectively cofibrant object (\cref{h1.lem.unit-cofibrant}).
Finally, the projective--injective comparison passes from the strict unit to the derived unit (\cref{h1.prop.derived-unit}).

\begin{lemma}\label{h1.lem.borel}
  Let $X$ be a levelwise discrete object of $\sSpace^{\dag}_{\rev}$ with a $C_2$-action, and let $A\subseteq X$ be an invariant subobject containing all levelwise fixed elements of the action.
  If $A\to X$ is a local equivalence, then $A/C_2\to X/C_2$ is a trivial cofibration of $L_{S_{\rev}}\sSpace^{\dag}_{\rev,\inj}$, where the quotients are the levelwise strict orbit objects.
\end{lemma}

\begin{proof}
  Equip $\Fun(BC_{2},L_{S_{\rev}}\sSpace^{\dag}_{\rev,\inj})$ with the projective model structure.
  This model structure exists because $L_{S_{\rev}}\sSpace^{\dag}_{\rev,\inj}$ is combinatorial and $BC_{2}$ is small \cite[Proposition~A.2.8.2]{HTT}.
  Its fibrations and weak equivalences are detected in $L_{S_{\rev}}\sSpace^{\dag}_{\rev,\inj}$.
  The strict-orbit functor
  \begin{align}
    (-)/C_2:
    \Fun(BC_{2},L_{S_{\rev}}\sSpace^{\dag}_{\rev,\inj})
    \to
    L_{S_{\rev}}\sSpace^{\dag}_{\rev,\inj}
  \end{align}
  is left Quillen, because its right adjoint is the trivial-action functor and projective fibrations and trivial fibrations are objectwise.

  Let $EC_2$ be the nerve of the indiscrete groupoid on the set $C_2$, with its free action, and put $E^{(-1)}:=\varnothing$ and $E^{(n)}:=\sk_nEC_2$.
  If $\calO_n$ is a set of representatives of the $C_2$-orbits of nondegenerate $n$-simplices, we obtain the following pushout square:
  \begin{align}
    &\begin{tikzpicture}[auto]
      \node (boundary) at (0,1.5) {$\displaystyle
        \coprod_{\sigma\in\calO_n}
        C_2\mathbin{\cdot}\partial\Delta^n$};
      \node (previous-skeleton) at (4,1.5) {$E^{(n-1)}$};
      \node (simplex) at (0,0) {$\displaystyle
        \coprod_{\sigma\in\calO_n}
        C_2\mathbin{\cdot}\Delta^n$};
      \node (skeleton) at (4,0) {$E^{(n)}$};
      \draw[->]
        (boundary) -- (previous-skeleton);
      \draw[->]
        (boundary) -- (simplex);
      \draw[->]
        (previous-skeleton) -- (skeleton);
      \draw[->]
        (simplex) -- (skeleton);
    \end{tikzpicture}
  \end{align}
  Here $C_2\mathbin{\cdot}(-)$ denotes the free $C_2$-object.
  The free $C_2$-object functor is left Quillen for the projective model structure, because its right adjoint is evaluation at the unique object of $BC_{2}$.
  
  For a $C_2$-object $V\in\Fun(BC_{2},L_{S_{\rev}}\sSpace^{\dag}_{\rev,\inj})$,
  tensoring the equivariant skeleton with $V$ gives relative attaching maps which are coproducts of $C_2\mathbin{\cdot}(V\times\partial\Delta^n) \to C_2\mathbin{\cdot}(V\times\Delta^n)$.
  Every object of $L_{S_{\rev}}\sSpace^{\dag}_{\rev,\inj}$ is cofibrant.
  Consequently, the underlying object of $V$ is cofibrant, and the map $V\times\partial\Delta^n \to V\times\Delta^n$ identifies with the pushout--product $(\varnothing\to V) \mathbin{\square} (\partial\Delta^n\to\Delta^n)$.
  It is therefore a cofibration in $L_{S_{\rev}}\sSpace^{\dag}_{\rev,\inj}$.
  Since the free $C_2$-object functor is left Quillen, every attaching map $C_2\mathbin{\cdot}(V\times\partial\Delta^n) \to  C_2\mathbin{\cdot}(V\times\Delta^n)$ is a projective cofibration.
  Starting from $V\times E^{(-1)}=\varnothing$ and using the displayed equivariant skeletal pushout at every stage shows that $V\times E^{(n)}$ is projectively cofibrant for every $n$.
  Passing to the sequential colimit gives that $V\times EC_2$ is projectively cofibrant.

  Let $i:A\to X$ denote the inclusion.
  Since $A\subseteq X$ is a levelwise monomorphism, $i$ is a cofibration in $\sSpace^{\dag}_{\rev,\inj}$.
  Left Bousfield localization does not change cofibrations, so $i$ is also a cofibration in $L_{S_{\rev}}\sSpace^{\dag}_{\rev,\inj}$.
  By hypothesis, $i$ is an $S_{\rev}$-local equivalence, hence a weak equivalence in $L_{S_{\rev}}\sSpace^{\dag}_{\rev,\inj}$.
  Therefore $i$ is a trivial cofibration in it.

  We next show that $A \times EC_{2} \to X \times EC_{2}$ is a projective trivial cofibration.
  Define
  \begin{align}
    Z_n
    :=
    (A\times EC_2)
    \amalg_{A\times E^{(n)}}
    (X\times E^{(n)})
    \quad \text{and} \quad
    D_n(i)
    :=
    (A\times\Delta^n)
      \amalg_{A\times\partial\Delta^n}
    (X\times\partial\Delta^n).
  \end{align}
  Then $Z_{-1}=A\times EC_2$ and $\colimstar_nZ_n=X\times EC_2$.
  The equivariant skeletal decomposition gives a pushout square
  \begin{align}
    &\begin{tikzpicture}[auto]
      \node (attaching-domain) at (0,1.5) {$\displaystyle
        \coprod_{\sigma\in\calO_n}
        C_2\mathbin{\cdot}D_n(i)$};
      \node (previous-stage) at (4,1.5) {$Z_{n-1}$};
      \node (attaching-codomain) at (0,0) {$\displaystyle
        \coprod_{\sigma\in\calO_n}
        C_2\mathbin{\cdot}(X\times\Delta^n)$};
      \node (stage) at (4,0) {$Z_n$};
      \draw[->]
        (attaching-domain) -- (previous-stage);
      \draw[->]
        (attaching-domain) -- (attaching-codomain);
      \draw[->]
        (previous-stage) -- (stage);
      \draw[->]
        (attaching-codomain) -- (stage);
    \end{tikzpicture}
  \end{align}
  Thus the left vertical map is a coproduct of the free $C_2$-object on $i\mathbin{\square}(\partial\Delta^n\to\Delta^n):D_n(i) \to X\times\Delta^n$.
  This is a trivial cofibration in $L_{S_{\rev}}\sSpace^{\dag}_{\rev,\inj}$ by the pushout--product axiom.
  Hence the left vertical map is a projective trivial cofibration.
  Therefore every map $Z_{n-1}\to Z_n$ is a projective trivial cofibration.
  Since projective trivial cofibrations are closed under sequential composition, $A\times EC_2 \to X\times EC_2$ is a projective trivial cofibration.
  Since the functor $(-)/C_{2}$ is left Quillen, 
  \begin{align}
    B(A) := (A\times EC_2)/C_{2} \to (X\times EC_2)/C_{2} =: B(X)
  \end{align}
  is a trivial cofibration in $L_{S_{\rev}}\sSpace^{\dag}_{\rev,\inj}$.

  Form the pushout $P:=B(X)\amalg_{B(A)}A/C_2$.
  We claim that the canonical collapse $P \to X/C_2$ is a levelwise weak equivalence.
  Fix $[m]\in\prism_{\rev}$, and let $\Orb_m(X)$ denote the set of $C_2$-orbits in the discrete set $X([m])$.
  Since $A([m])\subseteq X([m])$ is invariant, every orbit is either contained in $A([m])$ or disjoint from $A([m])$.
  We have a decomposition
  \begin{align}
    B(X)([m])
    \cong
    \coprod_{O\in\Orb_m(X)}
    (O\times EC_2)/C_2.
  \end{align}
  For an orbit $O\subseteq A([m])$, the corresponding component of the pushout defining $P([m])$ is
  \begin{align}
    (O\times EC_2)/C_2
    \amalg_{(O\times EC_2)/C_2}
    \Delta^0
    \cong
    \Delta^0.
  \end{align}
  For an orbit $O$ disjoint from $A([m])$, there is no contribution from $B(A)([m])$ or $(A/C_2)([m])$, so the corresponding component of $P([m])$ is $(O\times EC_2)/C_2$.
  Such an orbit is free, because $A([m])$ contains every fixed element of $X([m])$.
  After choosing one element of $O$, there is an equivariant isomorphism $O\cong C_2$, and hence an isomorphism
  \begin{align}
    (O\times EC_2)/C_2
    \cong
    EC_2.
  \end{align}
  Consequently,
  \begin{align}
    P([m])
    \cong
    \coprod_{\substack{O\in\Orb_m(X)\\ O\subseteq A([m])}}
    \Delta^0
    \amalg
    \coprod_{\substack{O\in\Orb_m(X)\\ O\cap A([m])=\varnothing}}
    EC_2
    \quad \text{and} \quad
    (X/C_2)([m])
    \cong
    \coprod_{O\in\Orb_m(X)}
    \Delta^0.
  \end{align}
  Under these identifications, $P([m])\to(X/C_2)([m])$ is the coproduct of identity maps $\Delta^0\to\Delta^0$ for the orbits contained in $A([m])$ and collapse maps $EC_2\to\Delta^0$ for the remaining orbits.
  Since $EC_2$ is contractible, each of these maps is a weak equivalence.
  Therefore $P([m])\to(X/C_2)([m])$ is a weak equivalence.
  Since $[m]$ was arbitrary, $P\to X/C_2$ is a levelwise weak equivalence.

  Finally, $A/C_2\to P$ is the cobase change of the trivial cofibration $B(A)\to B(X)$, and $P\to X/C_2$ is a levelwise weak equivalence.  
  Their composite $A/C_2\to X/C_2$ is therefore a local equivalence.  
  It is also a levelwise monomorphism.
  Indeed, if two elements of $A$ have the same orbit in $X$, one is carried to the other by an element of $C_2$;
  since $A$ is invariant, they already determine the same orbit in $A$.
  Hence it is a trivial cofibration of $L_{S_{\rev}}\sSpace^{\dag}_{\rev,\inj}$.
\end{proof}

\begin{lemma}\label{h1.lem.nerve}
  We have:
  \begin{enumerate}
    \item there exists a canonical bijection 
    \begin{align}
      \prism_{\dag}([m]_{\dag},[n]_{\dag}) \simeq \N_{\dag}([n]_{\dag})_m.
    \end{align}
    Consequently, $u^{*}\circ\yo_{\dag}\simeq\delta \N_{\dag}(-)$ on $\prism_{\dag}$; in particular $u^{*}(\Delta^\N_{\dag})=\delta \N_{\dag}([n]_{\dag})$.
    \item under the identification of (1), the unit 
    \begin{align}
      \eta:\yo_{\rev}[n]\to u^{*}u_{!}\yo_{\rev}[n]=\delta \N_{\dag}([n]_{\dag})
    \end{align}
    is the levelwise injection $\varphi\mapsto u(\varphi)$; 
    its image is the subobject $\delta \FdagSSet(\Delta^n)$ of chains all of whose entries are forward words (in the order-preserving levels) or all backward words (in the order-reversing levels).
  \end{enumerate}
\end{lemma}

\begin{proof}
  (1)
  Since $[m]_{\dag}$ is free on the generators $a_i$, a dagger functor $F:[m]_{\dag}\to[n]_{\dag}$ is the same as an arbitrary assignment of objects $f(i)$ together with morphisms $g_i:=F(a_i)\in[n]_{\dag}(f(i-1),f(i))$; 
  this is exactly an $m$-chain $(g_1,\dots,g_m)$ of $\N_{\dag}[n]_{\dag}$.
  
  This correspondence commutes with precomposition by every morphism $\alpha$ of $\prism_{\rev}$ after applying $u$.
  For order-preserving $\alpha$, this is the usual naturality of the nerve, whereas for order-reversing $\alpha$ both sides reverse the chain and apply the dagger to its entries.
  It is therefore natural in $[m]\in\prism_{\rev}$, which justifies the asserted isomorphism of presheaves.

  (2)
  The unit of $u_{!}\dashv u^{*}$ at a representable is the map 
  \begin{align}
    \prism_{\rev}([m],[n])\to\prism_{\dag}([m]_{\dag},[n]_{\dag}) : \varphi\mapsto u(\varphi),
  \end{align}
  which is injective by faithfulness of $u$ (\cref{h1.lem.u}).
  By \cref{h1.not.u}, $u(\alpha)$ for order-preserving $\alpha$ is the chain of forward words determined by $\alpha$, and $u(\gamma)$ for order-reversing $\gamma$ the chain of daggered forward; 
  these are exactly the elements of $\delta \FdagSSet(\Delta^n)\subseteq\delta \N_{\dag}([n]_{\dag})$.
\end{proof}

\begin{lemma}\label{h1.lem.outer}
  For $\ell\geq2$, let 
  \begin{align}
    A^{\ell}:=d_0\Delta^{\ell} \amalg_{d_0\Delta^{\ell} \cap d_{\ell}\Delta^{\ell}} d_{\ell}\Delta^{\ell}\subseteq\Delta^{\ell}
    \quad \text{and} \quad 
    A^{\ell}_{\rev}:=\delta \FdagSSet(A^{\ell})\subseteq\yo_{\rev}[\ell].
  \end{align}
  Then the inclusion $A^{\ell}_{\rev}\to\yo_{\rev}[\ell]$ is a trivial cofibration of $L_{S_{\rev}}\sSpace^{\dag}_{\rev,\inj}$.
\end{lemma}

\begin{proof}
  Let $\calA$ be the class of monomorphisms of simplicial sets $i:K\to L$ for which $\delta\FdagSSet(i)$ is a trivial cofibration of $L_{S_{\rev}}\sSpace^{\dag}_{\rev,\inj}$.
  The formula $\FdagSSet(K)=K\amalg_{K^{(0)}}K^{\myop}$ shows that $\delta\FdagSSet$ preserves monomorphisms and colimits.
  Hence $\calA$ is saturated.
  
  It also has the right cancellation property:
  if $i$ and $ji$ belong to $\calA$, then $\delta\FdagSSet(j)$ is a local equivalence by two-out-of-three and a monomorphism, hence a trivial cofibration.

  Every spine inclusion $\Sp[n]\to\Delta^n$ belongs to $\calA$ by the definition of $S_{\rev}$.
  The first part of the proof of \cite[Lemma~3.5]{JT06} therefore gives $\Sp[\ell]\to A^{\ell}$ in $\calA$.
  Its composite with $A^{\ell}\to\Delta^{\ell}$ is the spine inclusion, so right cancellation gives $A^{\ell}\to\Delta^{\ell}$ in $\calA$.
  This is precisely the asserted map $A^{\ell}_{\rev}\to\yo_{\rev}[\ell]$.
\end{proof}

\begin{remark}\label{h1.lem.fix}
  Let $C_2$ act on $\yo_{\rev}[\ell]$ through the Yoneda image of $\rho_{\ell}$, 
  i.e.\ by postcomposition $\varphi\mapsto \rho_{\ell}\circ\varphi$.
  The fixed elements of $\yo_{\rev}[\ell]([k])=\prism_{\rev}([k],[\ell])$ are: 
  none if $\ell$ is odd; 
  exactly the constant maps at $\ell/2$ if $\ell$ is even.
  In particular, for $\ell\geq2$ even, all fixed elements lie in $A^{\ell}_{\rev}$.
\end{remark}

\begin{remark}
  Recall that $[m]_{\dag}$ is the free dagger category on the chain $0\to1\to\cdots\to m$: 
  its underlying category is the free category on the doubled quiver $Q_m$, with vertices $0,\dots,m$ and edges $a_i:i-1\to i$ and $a_i^{\dag}:i\to i-1$, and its dagger reverses paths and exchanges $a_i\leftrightarrow a_i^{\dag}$.
\end{remark}

\begin{definition}\label{h1.not.paths}
  We define notions of letters and the length for $[m]_{\dag}$.
  \begin{itemize}
    \item we call the edges of $Q_m$ \emph{letters}; no letter is self-adjoint;
    \item a \emph{path of length $\ell$} is a composable sequence $p=(e_1,\dots,e_{\ell})$ of letters, with vertices $v_0,\dots,v_{\ell}$; the morphisms of $[m]_{\dag}$ are exactly the paths (including the empty paths, the identities), and every morphism has a \emph{unique} letter decomposition;
    \item the dagger path is $p^{\dag}=(e_{\ell}^{\dag},\dots,e_1^{\dag})$; we call $p$ \emph{palindromic} if $p^{\dag}=p$, i.e.\ $e_{\ell+1-j}=e_j^{\dag}$ for all $j$.
    Since no letter is self-adjoint, palindromic paths have \emph{even} length, and then $v_{\ell/2}$ is called the \emph{middle vertex};
  \end{itemize}
\end{definition}

\begin{proposition}\label{h1.prop.classifying}
  For every $m\geq0$, the inclusion $\delta \FdagSSet(\Sp[m])\to\delta \N_{\dag}([m]_{\dag})$ is a trivial cofibration of $L_{S_{\rev}}\sSpace^{\dag}_{\rev,\inj}$.
  Consequently the unit 
  \begin{align}
    \eta:\yo_{\rev}[m]\to u^{*}u_{!}\yo_{\rev}[m]=\delta \N_{\dag}([m]_{\dag})
  \end{align}
  is a trivial cofibration of $L_{S_{\rev}}\sSpace^{\dag}_{\rev,\inj}$; 
  in particular it is an $S_{\rev}$-local equivalence.
\end{proposition}

\begin{proof}
  The $k$-simplex of the nerve $\delta \N_{\dag}([m]_{\dag})$ is given by composable morphisms $c=(g_{1},\cdots,g_{k})$.
  we write $|c|:=\sum_j\ell(g_j)$ for the \emph{total length}, $\ell(g_{i})$ denoting the letter length of the path $g_{i}$.

  For $\ell\geq0$, let $X_{\ell}\subseteq \delta \N_{\dag}([m]_{\dag})$ be the subobject of chains $c$ with $|c|\leq\ell$.
  This is a subpresheaf: 
  the simplicial operators compose consecutive entries or delete outer entries, hence do not increase $|\mathrm{-}|$; degeneracies insert identities; the reversal preserves $|\mathrm{-}|$.

  We have $X_0=\delta\FdagSSet(\sk_0\Delta^m)$.
  A chain of total length at most one has at most one nonidentity entry, and that entry is a single letter $a_i$ or $a_i^{\dag}$.
  It therefore lies in the corresponding edge copy of $\FdagSSet(\Delta^1)$ inside $\FdagSSet(\Sp[m])$.
  Conversely, a composable chain in one of these edge copies contains at most one such letter.
  Hence $X_1=\delta\FdagSSet(\Sp[m])$.

  Finally $\delta \N_{\dag}([m]_{\dag})=\colimstar_{\ell}X_{\ell}$, every chain having finite total length.
  It therefore suffices to show that $X_{\ell-1}\to X_{\ell}$ is a trivial cofibration for every $\ell\geq2$.

  For a path $p=(e_1,\dots,e_{\ell})$ of length $\ell\geq2$, let 
  \begin{align}
    c_p:=(e_1,\dots,e_{\ell})\in \N_{\dag}([m]_{\dag})_{\ell}
  \end{align}
  be its letter chain, and let $\chi_p:\yo_{\rev}[\ell]\to \delta \N_{\dag}([m]_{\dag})$ be the classifying map of $c_p$; 
  explicitly $\chi_p(\varphi)=\varphi^{*}(c_p)$.
  Since the reversal acts on $c_p$ by $\rho_{\ell}^{*}(c_p)=c_{p^{\dag}}$, we have $\chi_{p^{\dag}}=\chi_p\circ(\rho_{\ell})_{*}$, where $(\rho_{\ell})_{*}$ is the flip of \cref{h1.lem.fix}.

  Call $\varphi:[k]\to[\ell]$ \emph{interior} if $\varphi(0)=0$ and $\varphi(k)=\ell$ in the order-preserving case, or if $\varphi(0)=\ell$ and $\varphi(k)=0$ in the order-reversing case.
  An element $\varphi$ belongs to $A^{\ell}_{\rev}$ if and only if its image omits $0$ or $\ell$.
  This is equivalent to saying that $\varphi$ is not interior.
  In that case $|\chi_p(\varphi)|<\ell$.
  Hence $\chi_p(A^{\ell}_{\rev})\subseteq X_{\ell-1}$.

  Let $\varphi$ be interior and write
  $\chi_p(\varphi)=(g_1,\ldots,g_k)$ and $d_i:=\ell(g_i)$.
  If $\varphi$ is order-preserving, the unique letter decompositions of the $g_i$ concatenate to $p$, and $\varphi(j)=\sum_{i=1}^{j}d_i$.
  For an order-reversing $\varphi$, put $\bar\varphi:=\rho_{\ell}\circ\varphi$.
  Then $\bar\varphi$ is order-preserving, and contravariance gives
  \begin{align}
    \chi_p(\varphi)
    =
    \bar\varphi^{*}(\rho_{\ell}^{*}c_p)
    =
    \bar\varphi^{*}(c_{p^{\dag}}).
  \end{align}
  Therefore its entries concatenate to $p^{\dag}$, and $ \varphi(j)=\ell-\sum_{i=1}^{j}d_i$.
  Consequently, for interior elements $\varphi$ and $\psi$,
  \begin{align}\label{h1.eq.interior-fibers}
    \chi_p(\varphi)=\chi_p(\psi)
    \quad\Longleftrightarrow\quad
    \varphi=\psi
    \quad\text{or}\quad
    (p=p^{\dag}\ \text{and}\ \psi=(\rho_{\ell})_{*}\varphi).
  \end{align}

  Thus the interior fibers of $\chi_p$ are singletons when $p\neq p^{\dag}$ and are exactly the $C_2$-orbits of the flip when $p=p^{\dag}$.
  The concatenation of an interior image is $p$ or $p^{\dag}$, so that image determines the unordered pair $\{p,p^{\dag}\}$.
  Therefore the interiors belonging to distinct unordered pairs are disjoint, and all are disjoint from $X_{\ell-1}$.
  If $p$ is palindromic, $\chi_p$ factors through the quotient $Q^{\ell}$.
  By \cref{h1.lem.fix}, a fixed element of the flip is constant at the middle vertex and is not interior;
  hence the flip acts freely on the interior, and \eqref{h1.eq.interior-fibers} shows that the induced map $\bar\chi_p$ is injective there.

  Fix $\ell\geq2$ and choose one representative $p$ from each unordered pair $\{p,p^{\dag}\}$ of paths of length $\ell$ in $[m]_{\dag}$.
  Then the squares
  \begin{align}
    &\begin{tikzpicture}[auto]
      \node (boundary-cells) at (0,1.5) {$\displaystyle
        \coprod_{p\neq p^{\dag}}A^{\ell}_{\rev}
        \amalg
        \coprod_{p=p^{\dag}}A^{\ell}_{\rev}/C_2$};
      \node (previous-length) at (4,1.5) {$X_{\ell-1}$};
      \node (cells) at (0,0) {$\displaystyle
        \coprod_{p\neq p^{\dag}}\yo_{\rev}[\ell]
        \amalg
        \coprod_{p=p^{\dag}}Q^{\ell}$};
      \node (length) at (4,0) {$X_{\ell}$};
      \draw[->]
        (boundary-cells) -- (previous-length);
      \draw[->]
        (boundary-cells) -- (cells);
      \draw[->]
        (previous-length) -- (length);
      \draw[->]
        (cells) -- (length);
    \end{tikzpicture}
  \end{align}
  are levelwise pushouts of sets: the attaching maps are the restrictions of $\chi_p$, $\bar\chi_p$ (which land in $X_{\ell-1}$, and factor through the quotient in the palindromic case), and $X_{\ell}\setminus X_{\ell-1}$ is, levelwise, the disjoint union of the interiors, onto which the cells map bijectively.

  For $p\neq p^{\dag}$, the map $A^{\ell}_{\rev}\to\yo_{\rev}[\ell]$ is a trivial cofibration (by \cref{h1.lem.outer}).
  For $p=p^{\dag}$, the path is palindromic, hence $\ell$ is even, the fixed elements of the flip are the middle-vertex constants and lie in $A^{\ell}_{\rev}$ (\cref{h1.lem.fix}); 
  by \cref{h1.lem.borel,h1.lem.outer}, $A^{\ell}_{\rev}/C_2\to Q^{\ell}$ is a trivial cofibration.

  Trivial cofibrations are closed under coproducts and cobase change, so $X_{\ell-1}\to X_{\ell}$ is a trivial cofibration for every $\ell\geq2$.
  They are also closed under transfinite composition, so $X_1\to\delta \N_{\dag}([m]_{\dag})$ is a trivial cofibration of $L_{S_{\rev}}\sSpace^{\dag}_{\rev,\inj}$.

  We can factor $X_1=\delta \FdagSSet(\Sp[m])\to \delta \N_{\dag}([m]_{\dag}) = \colimstar_{\ell}X_{\ell}$ as 
  \begin{align}
    \delta \FdagSSet(\Sp[m])
    \to
    \yo_{\rev}[m]
    = 
    \delta \FdagSSet(\Delta^{m})
    \xrightarrow{\eta} 
    \delta \N_{\dag}([m]_{\dag}).
  \end{align}
  By \cref{h1.lem.nerve} (2), the middle object is the subobject $\delta \FdagSSet(\Delta^m)\subseteq\delta \N_{\dag}([m]_{\dag})$, and the first map is $\delta \FdagSSet$ of the spine inclusion, 
  By two-out-of-three, $\eta$ is a local equivalence, and it is a monomorphism (by \cref{h1.lem.nerve} (2)), hence a trivial cofibration.
\end{proof}

\begin{lemma}\label{h1.lem.tensor}
  For every $V\in\sSpace^{\dag}_{\rev}$ and every $K\in\sSet$, there exists a natural equivalence 
  \begin{align}
    u_{!}(V\times K)\simeq u_{!}(V)\times K
    \quad \text{under which} \quad 
    \eta_{V\times K}=\eta_V\times K.
  \end{align}
\end{lemma}

\begin{proof}
  Regard $\prism_{\rev}^{\myop}$ and $\prism_{\dag}^{\myop}$ as simplicial categories with discrete mapping spaces.
  By \cite[Proposition~A.3.3.7]{HTT}, the adjunction $u_{!}\dashv u^{*}$ is a simplicial adjunction.
  Since $u^{*}$ preserves the levelwise tensors, the canonical tensor comparison of this simplicial adjunction is a natural isomorphism $ u_{!}(V\times K)\cong u_{!}(V)\times K$.
  Compatibility of the enriched adjunction unit with this canonical comparison gives $\eta_{V\times K}=\eta_V\times K$.
\end{proof}

\begin{proposition}\label{h1.lem.unit-cofibrant}
  For every projectively cofibrant $X\in\sSpace^{\dag}_{\rev}$, the unit $\eta_X:X\to u^{*}u_{!}X$ is an $S_{\rev}$-local equivalence.
\end{proposition}

\begin{proof}
  Since $u_{!}$ and $u^{*}$ are left adjoints, the composite $u^{*}u_{!}$ preserves colimits.
  The functor $u_{!}$ is left Quillen for the projective structures (by \cref{h1.lem.quillen-proj}), so it sends projective cofibrations to projective cofibrations, which are levelwise monomorphisms.
  The functor $u^{*}$ preserves levelwise monomorphisms;
  hence $u^{*}u_{!}$ sends projective cofibrations to levelwise monomorphisms.
  
  Let $C$ be the class of projectively cofibrant $X$ with $\eta_X$ a local equivalence.

  \emph{Cells:}
  The morphism $\eta_{\varnothing}$ is an isomorphism.
  For a cell $\yo_{\rev}[m]\times K$, \cref{h1.lem.tensor} gives $\eta_{\yo_{\rev}[m]\times K}=\eta_{\yo_{\rev}[m]}\times K$.
  By \cref{h1.prop.classifying}, $\eta_{\yo_{\rev}[m]}$ is a trivial cofibration.  
  Hence $\eta_{\yo_{\rev}[m]}\times K$ is its pushout--product with $\varnothing\to K$ and is a trivial cofibration (by \cref{h1.lem.localized-simplicial}).

  \emph{Cell attachments:}
  Let $X\in C$ and take the pushout 
  \begin{align}
    X'=X\amalg_{\yo_{\rev}[m]\times\partial\Delta^k}(\yo_{\rev}[m]\times\Delta^k).
  \end{align}
  Applying $u^{*}u_{!}$ gives the pushout with legs levelwise monomorphisms.
  The comparison of the two legs is a local equivalence on all three corners ($\eta_X$ by hypothesis; the other two by the previous argument).
  Since $L_{S_{\rev}}\sSpace^{\dag}_{\rev,\inj}$ is left proper, $\eta_{X'}$ is a local equivalence, so $X'\in C$.

  \emph{Transfinite composites:}
  Let $X_{\lambda}=\colimstar_{\beta<\lambda}X_{\beta}$ be the colimit of a transfinite sequence of such cell attachments, each $X_{\beta}\in C$.
  At the initial index the latching map $\emptyset \to X_{0}$ is a cofibration because every object is cofibrant in $L_{S_{\rev}}\sSpace^{\dag}_{\rev,\inj}$.
  In $L_{S_{\rev}}\sSpace^{\dag}_{\rev,\inj}$, the latching map of $(X_\beta)_\beta$ at a successor is the indicated monomorphism and hence a cofibration, while at a limit ordinal it is an isomorphism by continuity.
  The same is true for $(u^{*}u_{!}X_\beta)_\beta$, because $u^{*}u_{!}$ preserves colimits and sends the cell-attachment maps to levelwise monomorphisms.
  The maps $\eta_{X_\beta}$ are local equivalences by the induction hypothesis.
  Applying \cite[Corollary~A.2.9.25 (1)]{HTT} in $L_{S_{\rev}}\sSpace^{\dag}_{\rev,\inj}$ shows that $\eta_{X_\lambda}$ is a local equivalence.

  \emph{Retracts:}
  Since every projectively cofibrant object is a retract of a cell complex, $C$ contains all projectively cofibrant objects.
\end{proof}

\begin{lemma}\label{h1.thm.projection}
  We have:
  \begin{enumerate}
    \item the identity functors induce a Quillen equivalence 
    \begin{align}
      \id : L_{u_{!}(\widetilde S_{\rev})}\sSpace^{\dag}_{\proj} \rightleftarrows L_{u_{!}(\widetilde S_{\rev})}\sSpace^{\dag}_{\inj} : \id;
    \end{align}
    \item for every $\widetilde s : \widetilde A\to\widetilde B\in\widetilde S_{\rev}$, the map $u^{*}u_{!}(\widetilde s)$ is an $S_{\rev}$-local equivalence;
    \item the functor $u^{*}:L_{u_{!}(\widetilde S_{\rev})}\sSpace^{\dag}_{\inj}\to L_{S_{\rev}}\sSpace^{\dag}_{\rev,\inj}$ is a left Quillen functor;
    \item the functor $u^{*}$ sends weak equivalences in either of the model categories $L_{u_{!}(\widetilde S_{\rev})}\sSpace^{\dag}_{\inj}$ and $L_{u_{!}(\widetilde S_{\rev})}\sSpace^{\dag}_{\proj}$ to $S_{\rev}$-local equivalences.
  \end{enumerate}
\end{lemma}

\begin{proof}
  (1)
  Before localization, the identity functors form a Quillen equivalence from the projective to the injective model structure by \cite[Remark~A.2.8.6]{HTT}.
  The source and target of every map of $u_{!}(\widetilde S_{\rev})$ are projectively cofibrant.
  Hence the total left-derived image of the localizing set under the identity functor is represented by the same set $u_{!}(\widetilde S_{\rev})$.
  Since both localizations exist (by \cref{h1.def.model}), \cite[Theorem~3.3.20 (1)(b)]{Hir02} gives the asserted Quillen equivalence.

  (2)
  Naturality of the unit gives $u^{*}u_{!}(\widetilde s)\circ\eta_{\widetilde A}=\eta_{\widetilde B}\circ\widetilde s$.
  The maps $\eta_{\widetilde A}$ and $\eta_{\widetilde B}$ are local equivalences by \cref{h1.lem.unit-cofibrant}, and $\widetilde s$ is a local equivalence.
  By two-out-of-three, so is $u^{*}u_{!}(\widetilde s)$.

  (3)
  Here $u^{*}$ is the left adjoint in the adjunction $u^{*}\dashv u_{*}$ of \cref{h1.lem.adjoints} (1).
  Since $u^{*}$ preserves levelwise monomorphisms and levelwise weak equivalences, it is left Quillen in this adjunction.
  Composing with the identity left Quillen functor into $L_{S_{\rev}}\sSpace^{\dag}_{\rev,\inj}$ gives a left Quillen functor
  \begin{align}
    \sSpace^{\dag}_{\inj}
    \xrightarrow{u^{*}}
    \sSpace^{\dag}_{\rev,\inj}
    \xrightarrow{\id}
    L_{S_{\rev}}\sSpace^{\dag}_{\rev,\inj}.
  \end{align}
  Every object of $\sSpace^{\dag}_{\inj}$ is cofibrant, and the image $u^{*}u_{!}(\widetilde s)$ of every map of $u_{!}(\widetilde S_{\rev})$ is a weak equivalence in $L_{S_{\rev}}\sSpace^{\dag}_{\rev,\inj}$ by (2).
  Therefore \cite[Proposition~3.3.18 (1)]{Hir02} gives the asserted left Quillen functor from $L_{u_{!}(\widetilde S_{\rev})}\sSpace^{\dag}_{\inj}$.

  (4)
  Every object of $L_{u_{!}(\widetilde S_{\rev})}\sSpace^{\dag}_{\inj}$ is cofibrant, so by Ken Brown's lemma, $u^{*}$ sends all weak equivalences of $L_{u_{!}(\widetilde S_{\rev})}\sSpace^{\dag}_{\inj}$ to weak equivalences of $L_{S_{\rev}}\sSpace^{\dag}_{\rev,\inj}$, 
  i.e.\ to $S_{\rev}$-local equivalences (by \cref{h1.lem.rev-models}).

  Before localization, $\sSpace^{\dag}_{\proj}$ and $\sSpace^{\dag}_{\inj}$ have the same levelwise weak equivalences, and the identity Quillen equivalence identifies their homotopy categories and homotopy function complexes.
  Under this identification, the derived localizing morphisms are represented in both model structures by the same set $u_{!}(\widetilde S_{\rev})$.
  Hence the local objects correspond, the two notions of local equivalence coincide, and $L_{u_{!}(\widetilde S_{\rev})}\sSpace^{\dag}_{\proj}$ and $L_{u_{!}(\widetilde S_{\rev})}\sSpace^{\dag}_{\inj}$ have the same weak equivalences.
  Hence the same conclusion holds for every weak equivalence of $L_{u_{!}(\widetilde S_{\rev})}\sSpace^{\dag}_{\proj}$.
\end{proof}

\begin{notation}
  Fix a functorial fibrant-replacement
  \begin{align}
    R_{\dag}:
    L_{u_{!}(\widetilde S_{\rev})}\sSpace^{\dag}_{\proj}
    \to
    L_{u_{!}(\widetilde S_{\rev})}\sSpace^{\dag}_{\proj},
  \end{align}
  and denote its natural map by $r_V:V\to R_{\dag}V$.
\end{notation}

\begin{proposition}\label{h1.prop.derived-unit}
  For every projectively cofibrant $X\in\sSpace^{\dag}_{\rev}$, the derived unit $X\to u^{*}(R_{\dag}u_{!}X)$ is an $S_{\rev}$-local equivalence.
\end{proposition}

\begin{proof}
  The derived unit $X\to u^{*}(R_{\dag}u_{!}X)$ factors as 
  \begin{align}
    X\xrightarrow{\eta_X}u^{*}u_{!}X\xrightarrow{u^{*}(r)}u^{*}(R_{\dag}u_{!}X),
  \end{align}
  where $r:u_{!}X\to R_{\dag}u_{!}X$ is the fibrant replacement in $L_{u_{!}(\widetilde S_{\rev})}\sSpace^{\dag}_{\proj}$.
  The first map is an $S_{\rev}$-local equivalence (by \cref{h1.lem.unit-cofibrant}), and the second (by \cref{h1.thm.projection} (4)).
  So the composite is an $S_{\rev}$-local equivalence.
\end{proof}

\begin{theorem} \label{h1.thm.main}
  The Quillen adjunction of \cref{h1.thm.adjunction}
  \begin{align}
    u_{!}:
    L_{\widetilde S_{\rev}}\sSpace^{\dag}_{\rev,\proj}
    \rightleftarrows
    L_{u_{!}(\widetilde S_{\rev})}\sSpace^{\dag}_{\proj}
    :u^{*}
  \end{align}
  is a Quillen equivalence.
\end{theorem}

\begin{proof}
  By \cite[Corollary 1.3.16]{Hov99}, it suffices to show that the right adjoint $u^{*}$ reflects weak equivalences between fibrant objects and that, for every cofibrant $X$, the derived unit $X\to u^{*}R_{\dag}u_{!}X$ is a weak equivalence.
  The reflection condition follows from \cref{h1.thm.conservative}, and the required derived unit is an $S_{\rev}$-local equivalence (by \cref{h1.prop.derived-unit}).  
  These are precisely the weak equivalences of $L_{\widetilde S_{\rev}}\sSpace^{\dag}_{\rev,\proj}$ (by \cref{h1.rem.S-comparison}).
\end{proof}

\section{Frame realization and the derived unitary interval}\label{jt.frames-section}

In this section, we construct a reversal-compatible frame realization (\cref{jt.not.frames}).
The resulting realization--nerve adjunction is Quillen.
We compare its derived behavior with the ordinary Joyal--Tierney realization (\cref{jt.lem.transport-realization}) and compute the derived realization of the walking unitary (\cref{jt.thm.interval}).
These results provide the frame-theoretic input for the strictification and Quillen-equivalence arguments of \cref{jt.section}.

\subsection{The frame nerve and the realization functors}

We construct reversal-compatible cosimplicial frames in $\sCat^{\dag}$ and use them to define a frame realization--nerve adjunction (\cref{jt.not.frames}).

\begin{definition}
  For every $W \in \sSpace^{\dag}_{\rev}$, we define a functor 
  \begin{align}
    i_{0}^{*}(W) : \prism^{\myop}_{\rev} \to \Set
    \quad \text{by} \quad
    (i_{0}^{*}W)_m:=W_{m,0}.
  \end{align}
\end{definition}

\begin{lemma}\label{jt.lem.c-delta}
  We have:
  \begin{enumerate}
    \item there exists an adjunction 
    \begin{align}
      \delta : \sSet^{\dag} \rightleftarrows \sSpace^{\dag}_{\rev} : i_{0}^{*};
    \end{align}
    \item under the identification $\sSpace^{\dag\vee}\simeq\sSpace^{\dag}_{\rev}$, we have $\delta=c$ and $i_{0}^{*}=\ev_0$;
    in particular, \cref{cs.thm.equivalence} states that $\delta\dashv i_{0}^{*}$ is a Quillen equivalence onto the resolution model structure $\sSpace^{\dag\vee}_{\CSS}$.
  \end{enumerate}
\end{lemma}

\begin{proof}
  (1)
  A morphism $\delta X\to W$ is determined by a family $f_m: X_m \to W_{m,0}$ ($[m]\in\prism_{\rev}$) which is natural in $[m]$.
  Thus, for every $\sigma:[m]\to[n]$ in $\prism_{\rev}$, it satisfies
  \begin{align}
    W(\sigma,\id_{[0]})
    \circ
    f_n
    =
    f_m
    \circ
    X(\sigma).
  \end{align}
  The component $X_m=(\delta X)_{m,k}\to W_{m,k}$ is then determined by $f_m$ and the unique map $[k]\to[0]$.
  Hence these natural families are precisely the morphisms $X \to i_0^*W$.

  (2) is clear.
\end{proof}

\begin{remark}\label{jt.rem.two-localizations}
  Under the identification $ \sSpace^{\dag\vee} \simeq \sSpace^{\dag}_{\rev}$, the resolution model structure $\sSpace^{\dag\vee}_{\CSS}$ and the localized injective model structure $L_{S_{\rev}}\sSpace^{\dag}_{\rev,\inj}$ arise from different constructions:
  the former imposes homotopical constancy in the simplicial-object direction, while the latter imposes the Segal and unitary-completeness conditions in the $\prism_{\rev}$-direction.

  The identity functors are \emph{not} automatically a Quillen comparison between them.
  The comparison used below is instead obtained by proving separately that the same underlying adjunction $\delta=c \dashv i_{0}^{*}=\ev_{0}$ is a Quillen equivalence for each model structure.
\end{remark}

\begin{notation}\label{jt.not.tdag}
  We fix some notational conventions:
  \begin{itemize}
    \item let $j:\prism\hookrightarrow\prism_{\rev}$ be the inclusion and write $\res W:=j^*W$ for the ordinary simplicial space obtained by forgetting the horizontal reversal;
    \item let $\bbG[k]$ be the groupoid completion of $[k]$ and put $J^k=N\bbG[k]$ ($J^k$ is the simplicial set denoted $\Delta'[k]$ in \cite[\S2]{JT06});
    \item for $\xi=(a_0\xrightarrow{g_1}\cdots\xrightarrow{g_n}a_n)\in J^k_n$, we write $\xi^\vee=(a_n\xrightarrow{g_n^{-1}}\cdots \xrightarrow{g_1^{-1}}a_0)$;
  \end{itemize}
\end{notation}

\begin{construction}
  Define a functor
  \begin{align}
    t:\prism\times\prism\to\sSet
    \quad \text{by} \quad
    t([m],[k])=\Delta^m\times J^k,
  \end{align}
  and let $t_!:\sSpace\to\sSet$ be the left Kan extension of $t$ along the Yoneda embedding $\prism\times\prism\to\sSpace$;
  equivalently,
  \begin{align}
    t_!(W)
    :=
    \int^{([m],[k])\in\prism\times\prism}(\Delta^m\times J^k)\cdot W_{m,k}.
  \end{align}
  For $W\in\sSpace^{\dag}_{\rev}$, an $n$-simplex of $t_!(\res W)$ has a representative $[\alpha,\xi,w]$, where $\alpha:[n]\to[m]$ is order-preserving, $\xi\in J^k_n$, and $w\in W_{m,k}$.
  Define
  \begin{align}\label{jt.eq.tdag-formula}
    [\alpha,\xi,w]^{\dag}
      :=[\rho_m\alpha\rho_n,\xi^\vee,W(\rho_m)(w)].
  \end{align}
  This is compatible with both coend relations and fixes vertices.
  Indeed, 
  \begin{itemize}
    \item for an order-preserving $\theta:[m']\to[m]$, putting $\theta^\vee=\rho_m\theta\rho_{m'}$ gives $\rho_m\theta=\theta^\vee\rho_{m'}$;
    \item inversion in $\bbG[-]$ is natural for every functor $\bbG[k']\to\bbG[k]$;
    \item the same identities show compatibility with the simplicial operators and show that \eqref{jt.eq.tdag-formula} squares to the identity;
    \item the class of $(i,j,w)$ is sent to that of $(m-i,j,W(\rho_m)w)$, and the horizontal coend relation identifies the latter with $W(i)w$ because $\rho_m(m-i)=i$; inversion fixes the objects of $\bbG[k]$;
  \end{itemize}

  Thus \eqref{jt.eq.tdag-formula} defines a dagger structure on $t_!(\res W)$.
  We denote the resulting dagger simplicial set by $t_{\dag,!}W$;
  its underlying simplicial set is $\UdagSSet(t_{\dag,!}W) = t_!(\res W)$.
\end{construction}

\begin{construction}\label{jt.not.frames}
  For $[m]\in\prism_{\rev}$, we put
  \begin{align}\label{jt.eq.compatible-frames}
    A_m:=\frakC_{\dag}\FdagSSet(\Delta^m)
    \cong \FdagSCat\frakC(\Delta^m)
    \quad \text{and} \quad
    \Xi^k_{\dag}[m]
    :=\frakC_{\dag}t_{\dag,!}(\yo_{\rev}[m]\times\Delta^k).
  \end{align}
  These constructions are natural in $[m]\in\prism_{\rev}$.
  By co-Yoneda and \cite[Lemma 2.11]{JT06},
  \begin{align}\label{jt.eq.Xi-underlying}
    \UdagSCat\Xi^k_{\dag}[m] \cong\frakC(\UdagSSet\FdagSSet(\Delta^m)\times J^k)
    \quad \text{and} \quad
    \Xi^0_{\dag}[m]&=A_m.
  \end{align}
  The projections $J^k\to\Delta^0$ give an augmentation $\Xi^\bullet_{\dag}[m]\to cA_m$.
  It is degreewise a dagger DK-equivalence:
  after forgetting the dagger this follows because $J^k$ is a contractible Kan complex, the Joyal model structure is cartesian \cite[Theorem 1.9]{JT06}, and $\frakC$ is left Quillen;
  projection is surjective on vertex objects (choose $(i,0)$ over $i$), and the resulting identities witness coherent-unitary essential surjectivity.
  
  Consequently, $\UdagSCat\Xi^\bullet_{\dag}[m]$ is the homotopically constant cosimplicial object induced by the Joyal--Tierney total-space construction.

  Apply the functorial relative Reedy replacement of \cite[Proposition 16.6.8 (1)]{Hir02} to the functor $\Xi^\bullet_{\dag}[-]:\prism_{\rev}\to(\sCat^{\dag})^{\prism}$.
  More precisely, that proposition gives functorially, for every cosimplicial object $Y^\bullet$, a Reedy trivial fibration $FY^\bullet\to Y^\bullet$ which is an isomorphism in degree zero and for which $FY^\bullet$ is Reedy cofibrant whenever $Y^0$ is cofibrant.
  Applying this functor objectwise to the $\prism_{\rev}$-diagram preserves its full reversal naturality.

  The object $\FdagSSet(\Delta^m)$ is cofibrant by \cite[Corollary~3.4.4]{Aka26}, and $\frakC_{\dag}$ is left Quillen by \cite[Theorem~3.5.6]{Aka26}.
  Hence $A_m=\frakC_{\dag}\FdagSSet(\Delta^m)$ is cofibrant.
  Since $\Xi^0_{\dag}[m]=A_m$, the relative Reedy replacement gives, after identifying its degree-zero isomorphism with the identity, a Reedy trivial fibration
  \begin{align}\label{jt.eq.lambda-frame}
    \lambda_m^\bullet:\Gamma^\bullet_{\dag}[m]\xrightarrow{\ \sim\ }\Xi^\bullet_{\dag}[m],
    \quad
    \Gamma^0_{\dag}[m]=A_m,
    \quad
    \lambda_m^0=\id_{A_m}.
  \end{align}
  Functoriality makes this map natural for both order-preserving and order-reversing morphisms.
  The composite 
  \begin{align}
    \Gamma^\bullet_{\dag}[m]\to\Xi^\bullet_{\dag}[m]\to cA_m
  \end{align}
  is a Reedy weak equivalence, so $\Gamma^\bullet_{\dag}[m]$ is a cosimplicial frame in $\sCat^{\dag}_{\Bergner}$.

  By the coend--Hom adjunction, the functor $([m],[k])\mapsto\Gamma^k_{\dag}[m]$ defines an adjunction
  \begin{align}
    \mathfrak{C}^{\prism}_{\dag} : \sSpace^{\dag}_{\rev} \rightleftarrows \sCat^{\dag} : \N^{\prism}_{\dag}, 
  \end{align}
  given by
  \begin{align}
    \mathfrak{C}^{\prism}_{\dag}(W)
    &= 
    \int^{([m],[k]) \in \prism_{\rev} \times \prism}\Gamma^{k}_{\dag}[m]\cdot W([m],[k]), \\
    \N^{\prism}_{\dag}(\calC)
    &=
    (([m],[k])\mapsto\Hom_{\sCat^{\dag}}(\Gamma^k_{\dag}[m],\calC)).
  \end{align}
\end{construction}

\begin{lemma}\label{jt.lem.quillen}
  The adjunction $\mathfrak{C}^{\prism}_{\dag} \dashv \N^{\prism}_{\dag}$ induces a Quillen adjunction
  \begin{align}
    \mathfrak{C}^{\prism}_{\dag}:\sSpace^{\dag}_{\rev,\proj}
    \rightleftarrows
    \sCat^{\dag}_{\Bergner}:\N^{\prism}_{\dag}.
  \end{align}
\end{lemma}

\begin{proof}
  For every $[m]\in\prism_{\rev}$, $\Gamma^\bullet_{\dag}[m]$ is a cosimplicial frame on the cofibrant object $A_m$.
  Hence the functor
  \begin{align}
    \Gamma^\bullet_{\dag}[m]\otimes (-):\sSet\to\sCat^\dag_{\Bergner}
  \end{align}
  is left Quillen by \cite[Corollary~5.4.4]{Hov99}.
  Fix $[m]\in\prism_{\rev}$ and $K\in\sSet$.
  We use $[p]\in\prism_{\rev}$ for the variable in the dagger-simplicial direction and $[k]\in\prism$ for the variable in the simplicial direction.
  Then we have 
  \begin{align}
    \frakC^\prism_{\dag}(\yo_{\rev}[m]\times K)
    &\cong
    \int^{([p],[k])\in\prism_{\rev}\times\prism} 
    \Gamma^k_{\dag}[p]
    \cdot
    (\Hom_{\prism_{\rev}}([p],[m])\times K_k) \\
    &\cong
    \int^{[k]}
    \int^{[p]}
    \Gamma^k_{\dag}[p]
    \cdot
    \Hom_{\prism_{\rev}}([p],[m])
    \cdot K_k \\
    &\cong
    \int^{[k]\in\prism}
    \Gamma^k_{\dag}[m]\cdot K_k \\
    &=
    \Gamma^\bullet_{\dag}[m]\otimes K.
  \end{align}
  If $i:K\to L$ is a generating cofibration or a generating trivial cofibration of $\sSet$, then the corresponding projective generating map is
  \begin{align}
    \yo_{\rev}[m]\times i:
    \yo_{\rev}[m]\times K
    \to
    \yo_{\rev}[m]\times L.
  \end{align}
  Under the preceding natural isomorphism, its image under $\frakC^\prism_{\dag}$ is
  \begin{align}
    \Gamma^\bullet_{\dag}[m]\otimes i:
    \Gamma^\bullet_{\dag}[m]\otimes K
    \to
    \Gamma^\bullet_{\dag}[m]\otimes L.
  \end{align}
  This is a cofibration or a trivial cofibration, respectively, because $\Gamma^\bullet_{\dag}[m]\otimes (-)$ is left Quillen.
  Thus $\frakC^\prism_{\dag}$ sends the generating projective cofibrations and generating projective trivial cofibrations to cofibrations and trivial cofibrations, respectively.
\end{proof}

\begin{lemma}\label{jt.lem.frame-comparison}
  The maps \eqref{jt.eq.lambda-frame} induce a natural dagger functor
  \begin{align}
    \Lambda_W: \frakC^{\prism}_{\dag}(W) \to \frakC_{\dag}(t_{\dag,!}W).
  \end{align}
  If $W$ is projectively cofibrant, it is a dagger DK-equivalence.
\end{lemma}

\begin{proof}
  Co-Yoneda and preservation of coends by the left adjoint $\frakC_{\dag}$ give a natural identification
  \begin{align}
    \frakC_{\dag}(t_{\dag,!}W)
    \cong
    \int^{([m],[k])}
    \Xi^k_{\dag}[m]\cdot W_{m,k}.
  \end{align}
  Under this identification, $\Lambda_W$ is the coend of the actual natural transformation $\lambda:\Gamma^\bullet_{\dag}[-]\to\Xi^\bullet_{\dag}[-]$.

  We first show that the underlying comparison is a comparison of ordinary frames.
  Since $\Gamma^\bullet_{\dag}[m]$ is Reedy cofibrant in $\sCat^{\dag}_{\Bergner}$, it follows that $\UdagSCat\Gamma^\bullet_{\dag}[m]$ is Reedy cofibrant in $\sCat_{\Bergner}$.
  
  Since $\lambda_m^\bullet$ is a Reedy trivial fibration, every $\UdagSCat\lambda_m^k$ is an ordinary DK-equivalence.
  Composing it with the degreewise ordinary DK-equivalence $\UdagSCat\Xi^k_{\dag}[m]\to\UdagSCat A_m$ shows that the augmentation of $\UdagSCat\Gamma^\bullet_{\dag}[m]$ is degreewise an ordinary DK-equivalence.

  Put $X_m:=\UdagSSet\FdagSSet(\Delta^m)$.
  By \eqref{jt.eq.Xi-underlying}, and because cartesian product with $X_m$ and the left adjoint $\frakC$ preserve colimits, the $k$-th latching map of $\UdagSCat\Xi^\bullet_{\dag}[m]$ is $\frakC(X_m\times L^kJ)\to\frakC(X_m\times J^k)$.
  The map $L^kJ\to J^k$ is a monomorphism.
  Indeed, $L^kJ$ is the union of the nerves of the full subgroupoids of $\bbG[k]$ on the proper subsets of $\{0,\ldots,k\}$, and is therefore a simplicial subset of $J^k$.
  Hence $X_m\times L^kJ\to X_m\times J^k$ is a Joyal cofibration.
  
  Since $\frakC$ is left Quillen, the displayed latching map is an ordinary Bergner cofibration.
  Moreover, each projection $J^k\to\Delta^0$ is a categorical equivalence.
  Since $\sSet_{\Joyal}$ is cartesian and every simplicial set is cofibrant, $X_m\times J^k\to X_m$ is a Joyal equivalence.
  The left Quillen functor $\frakC$ therefore sends it to an ordinary DK-equivalence.
  Thus $\UdagSCat\Xi^\bullet_{\dag}[m]$ is an ordinary cosimplicial frame on $\UdagSCat A_m$.
  Consequently, the source and target of $\UdagSCat\lambda_m^\bullet$ are ordinary cosimplicial frames on the same cofibrant object, and $\UdagSCat\lambda_m^0=\id$.

  We first verify the object condition, which does not require cofibrancy of $W$.
  Every target object has a representative $(b,w)$ with $b\in\Ob\Xi^k_{\dag}[m]$ and $w\in W_{m,k}$.
  Since $\lambda_m^k$ is a dagger DK-equivalence, $b$ is coherently unitarily equivalent to $\lambda_m^k(a)$ for some $a\in\Ob\Gamma^k_{\dag}[m]$.
  The coend coprojection indexed by $w$ carries this witness to a coherent unitary from $\Lambda_W(a,w)$ to $(b,w)$.
  Thus $\Lambda_W$ is coherently unitarily essentially surjective for every $W$.

  We now apply the universal description of realization functors associated to cosimplicial resolutions.
  Let $\gamma:\prism_{\rev}\to\sCat$ be the functor given by $\gamma([m]):=\UdagSCat A_m$.

  The preceding paragraphs show that $\UdagSCat\Gamma^\bullet_{\dag}[-]$ and $\UdagSCat\Xi^\bullet_{\dag}[-]$ are cosimplicial resolutions of the same cofibrant-valued diagram $\gamma$.
  Moreover,
  $\UdagSCat\lambda$ is an augmentation-compatible natural weak equivalence between these resolutions.
  By \cite[Proposition~3.4 and \S9.5]{DugUHT01}, the two resolutions correspond to the coend functors
  \begin{align}
    W \mapsto \UdagSCat\frakC^{\prism}_{\dag}(W)
    \quad \text{and} \quad 
    W \mapsto \UdagSCat\frakC_{\dag}(t_{\dag,!}W),
  \end{align}
  and the induced natural transformation is precisely $\UdagSCat\Lambda$.
  By \cite[the proof of Lemma~9.7]{DugUHT01}, this natural transformation is a Quillen homotopy.
  Consequently, $\UdagSCat\Lambda_W$ is an ordinary DK-equivalence for every projectively cofibrant $W$.
  Together with the coherent-unitary essential-surjectivity established above, this proves that $\Lambda_W$ is a dagger DK-equivalence.
\end{proof}

\begin{lemma}\label{jt.lem.transport}
  There exist natural isomorphisms of underived functors
  \begin{align}
    \Theta:\frakC^{\prism}_{\dag}\delta
    \xrightarrow{\cong}\frakC_{\dag}
    \quad\text{and}\quad
    \vartheta:
    i_0^*\N^{\prism}_{\dag}
    \xrightarrow{\cong}
    \N_{\dag}.
  \end{align}
  The inverse $\vartheta^{-1}$ is the mate of $\Theta$;
  equivalently, $\vartheta$ is the mate of $\Theta^{-1}$.
\end{lemma}

\begin{proof}
  Since $\delta X$ is constant in its second simplicial direction, the coend formula and the identity $\Gamma^{\bullet}_{\dag}[m]\otimes\Delta^0=\Gamma^0_{\dag}[m]=A_m$ give
  \begin{align}
    \frakC^{\prism}_{\dag}(\delta X)
    &\cong \int^{[m]\in\prism_{\rev}} A_m\cdot X_m \\
    &= \int^{[m]}\frakC_{\dag}\FdagSSet(\Delta^m)\cdot X_m \\
    &\cong \frakC_{\dag}(\int^{[m]}\FdagSSet(\Delta^m)\cdot X_m) \\
    &\cong \frakC_{\dag}X.
  \end{align}
  The inverse of the mate of $\Theta$ is given degreewise by
  \begin{align}
    (i_0^*\N^{\prism}_{\dag}\calC)_m
    =\Hom_{\sCat^{\dag}}(A_m,\calC)
    \cong\Hom_{\sSet^{\dag}}(\FdagSSet\Delta^m,\N_{\dag}\calC)
    =(\N_{\dag}\calC)_m.
  \end{align}
  These degreewise isomorphisms are natural in $[m]\in\prism_{\rev}$ and in $\calC$, so they define $\vartheta$.
  The displayed adjunction calculation shows that $\Theta$ and $\vartheta^{-1}$ form a mate pair.
\end{proof}

\begin{definition}
  For $X\in\sSet^{\dag}$, define a simplicial object in $\sSpace^{\dag}_{\rev}$ by
  \begin{align}
    \Barconstruction_r(X)
    :=
    \coprod_{\substack{[m_0]\xrightarrow{u_1}[m_1]\to\cdots\to[m_r]}}\yo_{\rev}[m_0]\cdot X_{m_r}.
  \end{align}
  Here the $u_i$ are morphisms of $\prism_{\rev}$ and $\cdot$ denotes the copower by a set.
  The face maps compose adjacent morphisms, with the first face using the Yoneda action and the last face using the presheaf action on $X$, while the degeneracy maps insert identity morphisms.
\end{definition}

\begin{notation}
  For $X\in\sSet^{\dag}$, let
  \begin{align}
    \varepsilon_X:
    B_{\proj}(\delta X)
    :=
    |\Barconstruction_{\bullet}(X)|
    =
    \int^{[r]\in\prism}\Barconstruction_r(X)\times\Delta^r
    \to
    \delta X
  \end{align}
  be the canonical bar augmentation.
\end{notation}

\begin{proposition}\label{jt.lem.density}
  The frame augmentations induce a natural dagger DK-equivalence
  \begin{align}
    \beta_X:\frakC^{\prism}_{\dag}B_{\proj}(\delta X)\xrightarrow{\simeq}\frakC_{\dag}X.
  \end{align}
  The map $\varepsilon_X$ exhibits $B_{\proj}(\delta X)$ as a projective cofibrant replacement of $\delta X$.
  Moreover, $\beta$ is compatible with reversal and represents, at $\delta X$, the projectively derived comparison induced by the strict co-Yoneda isomorphism constructed below.
\end{proposition}

\begin{proof}
  Under the equivalence $\sSet^{\dag}\simeq\Fun(\prism_{\rev}^{\myop},\Set)$, the realization $B_{\proj}(\delta X)$ is Dugger's simplicial resolution by representables of the presheaf $[m]\mapsto X_m$ on $\prism_{\rev}$.
  Therefore \cite[Lemma~2.7 and \S9.1]{DugUHT01} shows that $\varepsilon_X$ is a rowwise weak equivalence and that $B_{\proj}(\delta X)$ is projectively cofibrant.
  The length-zero summand $(\id_{[m]},x)$ shows in addition that $\varepsilon_X$ is surjective in every bidegree $(m,0)$, and in particular in bidegree $(0,0)$.

  There exists a natural strict dagger isomorphism
  \begin{align}
    \tau_X:t_{\dag,!}(\delta X)\xrightarrow{\cong}X.
  \end{align}
  After forgetting reversal, $\res(\delta X)$ is the simplicial space constant in its second direction at $\UdagSSet X$.
  Thus \cite[Lemma~2.11]{JT06}, applied with its second variable equal to $\Delta^0$, gives the underlying isomorphism 
  \begin{align}
    t_!(\res(\delta X))\cong\UdagSSet X.
  \end{align}
  In the coend description, it sends a representative $[\alpha,\xi,x]$ to $X(\alpha)(x)$.
  Formula \eqref{jt.eq.tdag-formula} gives
  \begin{align}
    \tau_X([\alpha,\xi,x]^{\dag})
    =
    X(\rho_m\alpha\rho_n)X(\rho_m)(x) 
    =
    X(\rho_n)X(\alpha)(x)
    =
    \tau_X([\alpha,\xi,x])^{\dag}.
  \end{align}
  Hence $\tau_{X}$ is an isomorphism in $\sSet^{\dag}$.

  Since the bar resolution is projectively cofibrant, the comparison of \cref{jt.lem.frame-comparison} gives a dagger DK-equivalence
  \begin{align}
    \Lambda_{B_{\proj}(\delta X)}:
    \frakC^\prism_{\dag}B_{\proj}(\delta X)
    \xrightarrow{\simeq}
    \frakC_{\dag}t_{\dag,!}B_{\proj}(\delta X).
  \end{align}
  The rowwise weak equivalence $\res\varepsilon_X$ is a Rezk weak equivalence.
  Every bisimplicial set is cofibrant in $\sSpace_{\CSS}$, and $t_!$ is left Quillen to $\sSet_{\Joyal}$ by \cite[Theorems~4.1 and~4.12]{JT06}.
  Hence Ken Brown's lemma shows that $t_!(\res\varepsilon_X)$ is a Joyal equivalence.
  By the rigidification identity recalled above, its rigidification is the underlying ordinary functor of
  \begin{align}
    \frakC_{\dag}t_{\dag,!}(\varepsilon_X):
    \frakC_{\dag}t_{\dag,!}B_{\proj}(\delta X)
    \to
    \frakC_{\dag}t_{\dag,!}(\delta X).
  \end{align}
  Since every simplicial set is Joyal cofibrant and $\frakC$ is left Quillen, Ken Brown's lemma shows that this functor is a local weak equivalence.
  Its object map is $\varepsilon_{X,0,0}$, which is surjective, so identities witness coherent-unitary essential surjectivity.
  Hence $\frakC_{\dag}t_{\dag,!}(\varepsilon_X)$ is a dagger DK-equivalence.

  Define $\beta_X$ to be the composite
  \begin{align}
    \frakC^{\prism}_{\dag}B_{\proj}(\delta X)
      &\xrightarrow[\Lambda]{\simeq}
    \frakC_{\dag}t_{\dag,!}B_{\proj}(\delta X)
      \xrightarrow{\ \frakC_{\dag}t_{\dag,!}(\varepsilon_X)\ }
    \frakC_{\dag}t_{\dag,!}(\delta X)
      \xrightarrow[\frakC_{\dag}(\tau_X)]{\cong}
    \frakC_{\dag}X.
  \end{align}
  All three maps are natural in morphisms $X\to Y$ of $\sSet^{\dag}$.
  The bar construction and $\lambda$ are defined using the full category $\prism_{\rev}$, and $\tau_X$ is dagger-compatible by the calculation above.
  Thus $\beta$ is natural and compatible with reversal.

  Let $\Theta_X:\frakC^\prism_{\dag}(\delta X)\xrightarrow{\cong}\frakC_{\dag}X$ be the strict co-Yoneda comparison of \cref{jt.lem.transport}.
  Naturality of $\Lambda$ and the identity $\lambda_m^0=\id_{A_m}$ give
  \begin{align}
    \frakC_{\dag}t_{\dag,!}(\varepsilon_X)
      \circ\Lambda_{B_{\proj}(\delta X)}
    =
    \Lambda_{\delta X} \circ\frakC^\prism_{\dag}(\varepsilon_X)
    \quad \text{and} \quad
    \frakC_{\dag}(\tau_X)\circ\Lambda_{\delta X}
    =
    \Theta_X.
  \end{align}
  Therefore, we have 
  \begin{align}
    \beta_X =\Theta_X\circ\frakC^\prism_{\dag}(\varepsilon_X)
    :\frakC^{\prism}_{\dag}B_{\proj}(\delta X) \to \frakC_{\dag}X
  \end{align}
  If $QX\to X$ is a dagger-Joyal cofibrant replacement, then $\frakC_{\dag}QX\to\frakC_{\dag}X$ is a dagger DK-equivalence by the definition of dagger-Joyal weak equivalences.
  Hence $\frakC_{\dag}X$ represents $\bfL\frakC_{\dag}(X)$, and $\beta_X$ is the projectively derived comparison induced by $\Theta$ at $\delta X$.
\end{proof}

\begin{corollary}\label{jt.lem.transport-realization}
  For every $X\in\sSet^{\dag}$, there exists a natural zigzag of dagger DK-equivalences
  \begin{align}\label{jt.eq.transport-realization}
    \bfL\frakC^{\prism}_{\dag}(\delta X)
    &\simeq
    \frakC^{\prism}_{\dag}B_{\proj}(\delta X)
    \xrightarrow[\beta_X]{\simeq}
    \frakC_{\dag}X.
  \end{align}
\end{corollary}

\begin{proof}
  By \cref{jt.lem.density}, $B_{\proj}(\delta X)$ is a projective cofibrant replacement and its frame realization maps by the dagger DK-equivalence $\beta_X$ directly to $\frakC_{\dag}X$.
\end{proof}

\begin{corollary}\label{jt.cor.transport-mapping}
  For every cofibrant $X \in \sSet^{\dag}_{\Joyal}$ and every fibrant $\calC \in \sCat^{\dag}_{\Bergner}$, there exist natural weak equivalences
  \begin{align}
    \RMap_{\sSpace^{\dag}_{\rev,\proj}}(\delta X,\N^{\prism}_{\dag}(\calC))
    \simeq
    \RMap_{\sCat^{\dag}_{\Bergner}}(\frakC_{\dag}X,\calC)
    \simeq
    \RMap_{\sSet^{\dag}_{\Joyal}}(X,\N_{\dag}(\calC)).
  \end{align}
\end{corollary}

\begin{proof}
  The first equivalence is the derived adjunction of \cref{jt.lem.quillen} followed by \cref{jt.lem.transport-realization};
  the second is the derived adjunction $\frakC_{\dag}\dashv \N_{\dag}$.
\end{proof}

\subsection{The reduction principle and the derived interval}

This subsection provides the local computations used in the strictification argument.
The transport theorem computes the derived realization of the walking unitary (\cref{jt.thm.interval}).
It follows that frame nerves of fibrant dagger simplicial categories satisfy the Segal and unitary-completeness local conditions (\cref{jt.cor.local}).

\begin{lemma}\label{jt.lem.free-inner-trivial}
  If $0<k<n$, the free dagger inner-horn inclusion $\FdagSSet(\Lambda^n_k)\to\FdagSSet(\Delta^n)$ is a dagger Joyal trivial cofibration.
  Consequently, every relative cell complex generated by the free dagger inner-horn inclusions is a dagger Joyal trivial cofibration.
\end{lemma}

\begin{proof}
  Put $i:\Lambda^n_k\to\Delta^n$.
  The map $i$ is a Joyal trivial cofibration.
  Since $\frakC$ is left Quillen, $\frakC(i)$ is a Bergner trivial cofibration.
  It is bijective on objects, since both its source and target have object set $\{0,\ldots,n\}$.

  The natural compatibility of free dagger completion with rigidification gives $\frakC_{\dag}\FdagSSet(i) \cong \FdagSCat\frakC(i)$.
  \Cite[Lemma~3.2.7 (2)]{Aka26} proves that $\FdagSCat$ sends an object-bijective Bergner trivial cofibration to a dagger Bergner trivial cofibration.
  Hence $\frakC_{\dag}\FdagSSet(i)$ is a dagger DK-equivalence, and therefore $\FdagSSet(i)$ is a dagger Joyal equivalence by definition.
  Since $i$ is a monomorphism, $\FdagSSet(i)$ is also a free cofibration by the dagger Joyal model theorem recalled in \cref{review.section}.
  Thus $\FdagSSet(i)$ is a dagger Joyal trivial cofibration.

  Finally, trivial cofibrations form the left class of a weak factorization system and are therefore closed under coproducts, pushouts, and transfinite compositions.
  Every relative cell complex generated by the free dagger inner-horn inclusions is consequently a dagger Joyal trivial cofibration.
\end{proof}

\begin{lemma}\label{jt.lem.free-inner-underlying}
  If $0<k<n$, the underlying map of the free dagger inner-horn inclusion $\FdagSSet(\Lambda^n_k) \to \FdagSSet(\Delta^n)$ is inner anodyne.
  Consequently, the underlying map of every relative cell complex generated by the free dagger inner-horn inclusions is inner anodyne.
\end{lemma}

\begin{proof}
  Since an inner horn contains every vertex, the explicit formula for free dagger completion gives
  \begin{align}
    \UdagSSet\FdagSSet(\Lambda^n_k)
    =
    \Lambda^n_k\amalg_{\sk_0\Delta^n}(\Lambda^n_k)^{\myop}
    \quad \text{and} \quad
    \UdagSSet\FdagSSet(\Delta^n)
    =
    \Delta^n\amalg_{\sk_0\Delta^n}(\Delta^n)^{\myop}.
  \end{align}
  The map between them factors as
  \begin{align}
    \Lambda^n_k\amalg_{\sk_0\Delta^n}(\Lambda^n_k)^{\myop}
    \to 
    \Delta^n\amalg_{\sk_0\Delta^n}(\Lambda^n_k)^{\myop}
    \to 
    \Delta^n\amalg_{\sk_0\Delta^n}(\Delta^n)^{\myop}.
  \end{align}
  The first arrow is a pushout of $\Lambda^n_k\to\Delta^n$.
  After identifying opposite simplices by order reversal, the second is a pushout of
  $\Lambda^n_{n-k}\to\Delta^n$.
  Both arrows are inner anodyne, and so is their composite.

  Under $\sSet^{\dag} \cong \Fun(\prism_{\rev}^{\myop},\Set)$, the forgetful functor $\UdagSSet$ corresponds to restriction along
  $\prism\to\prism_{\rev}$ and therefore preserves colimits.
  It consequently carries pushouts and transfinite composites of free dagger inner-horn inclusions to pushouts and transfinite composites of the inner anodyne maps just considered.
  Since the inner anodyne maps are weakly saturated, the final assertion follows.
\end{proof}

\begin{remark}
  For an injectively fibrant $S_{\rev}$-local $W$, choose a functorial projective trivial fibration $p:Q_{\proj}W\to W$ with projectively cofibrant source;
  We factor $t_{\dag,!}Q_{\proj}W\to *$ by the small object argument with respect to the family $\{\FdagSSet(\Lambda^n_k) \to \FdagSSet(\Delta^n) ~|~ 0<k<n\}$.
  
  We denote the left factor by $t_{\dag,!}Q_{\proj}W \to T_W$.
  It is a relative free-dagger inner-horn cell complex, hence a dagger Joyal trivial cofibration by \cref{jt.lem.free-inner-trivial}.
  By \cref{jt.lem.free-inner-underlying}, its underlying map is inner anodyne.
  Since inner horns contain all vertices, the map is vertex-preserving.

  The right factor $T_W\to *$ has the right lifting property against all free dagger inner-horn inclusions.
  By the adjunction $\FdagSSet\dashv\UdagSSet$, this says exactly that
  $\UdagSSet T_W\to *$ has the right lifting property against all ordinary inner horns.
  Therefore $\UdagSSet T_W$ is a quasi-category;
\end{remark}

\begin{notation}\label{jt.not.Vred}
  We fix some notational conventions:
  \begin{itemize}
    \item for $W \in \sSpace^{\dag}_{\rev}$ and for vertices $a,b\in(W_0)_0$, put 
    \begin{align}
      \Map_W(a,b):=\hofib_{(a,b)}(W_1\to W_0\times W_0);
    \end{align}
    \item for a quasi-category $K$ and vertices $x,y\in K_0$, $\Map_K(x,y)$ denotes any standard homotopy function complex from $x$ to $y$;
    \item for a simplicial set $S$ with an ordered pair of vertices $(x,y)$, write $S_{x,y}$ for the corresponding object $\partial\Delta^1\to S$, and put 
    \begin{align}
      C_E^k
      &:=(\Delta^1\times J^k)\amalg_{\partial\Delta^1\times J^k}\partial\Delta^1,\\
      \Map_E(S;x,y)
      &:=[k \mapsto \Hom_{\partial\Delta^1\mathbin{\downarrow}\sSet}(C_E^k,S_{x,y})].
    \end{align}
  \end{itemize}
\end{notation}

\begin{remark}\label{jt.not.ordered-mapping}
  We have some remarks:
  \begin{itemize}
    \item the object $J^k$ is the nerve of the indiscrete groupoid on $\{0,\ldots,k\}$, so $C_E^\bullet$ is the cosimplicial resolution denoted by the same letter in \cite[Proposition~4.5 (d)]{DS11};
    \item if $S$ is a quasi-category, $\Map_E(S;x,y)$ computes the ordered fixed-endpoint homotopy function complex by \cite[Corollary~4.8]{DS11};
    \item the formula for the ordinary right adjoint $t^!$ gives a strict natural isomorphism
    \begin{align}
      \fib_{(x,y)}
      ((t^!S)_1\to(t^!S)_0\times(t^!S)_0)
      \cong
      \Map_E(S;x,y).
    \end{align}
    Indeed, in vertical degree $k$ the condition of lying over $(x,y)$ says precisely that a map $\Delta^1\times J^k\to S$ is constant with values $x$ and $y$ on the two components of $\partial\Delta^1\times J^k$.
  \end{itemize}
\end{remark}

\begin{proposition}\label{jt.thm.interval}
  Let $c:\delta E_{\dag}\to\Delta^0_{\rev}$ be the canonical collapse.
  Then there exists an isomorphism in $\Ho(\sCat^{\dag}_{\Bergner})$
  \begin{align}
    \bfL\frakC^{\prism}_{\dag}(c):
    \bfL\frakC^{\prism}_{\dag}(\delta E_{\dag})
    \xrightarrow{\cong}
    \bfL\frakC^{\prism}_{\dag}(\Delta^0_{\rev})
    \simeq 
    \bfone_{\dag}.
  \end{align}
\end{proposition}

\begin{proof}
  By \cref{jt.lem.transport-realization} and $E_{\dag}=\N_{\dag}(\bbI_{\dag})$, there exists a natural zigzag
  \begin{align}
    \bfL\frakC^{\prism}_{\dag}(\delta E_{\dag})\simeq \frakC_{\dag}\N_{\dag}(\bbI_{\dag}) \xrightarrow{\varepsilon_{\bbI_{\dag}}} \bbI_{\dag}.
  \end{align}
  The walking unitary is fibrant in $\sCat^{\dag}_{\Bergner}$, since all its mapping spaces are points.
  The counit is a dagger DK-equivalence by the recalled dagger Joyal--Bergner equivalence;
  after forgetting the dagger, it is the ordinary Joyal--Bergner counit.
  The collapse $\bbI_{\dag}\to\bfone_{\dag}$ is homwise an isomorphism and coherently unitarily essentially surjective, since its two objects are joined by the generating unitary.
  It is therefore a dagger DK-equivalence.

  It remains to identify the resulting composite.
  Let $p:\bbI_{\dag}\to\bfone_{\dag}$ be the collapse and consider $\N_{\dag}(p):E_{\dag}\to\Delta^0_{\dag}$.
  Naturality of \cref{jt.lem.density} identifies $\bfL\frakC^{\prism}_{\dag}(\delta \N_{\dag}(p))$ with $\frakC_{\dag}\N_{\dag}(p)$.
  Naturality of the counit gives 
  \begin{align}
    p\circ\varepsilon_{\bbI_{\dag}}=\varepsilon_{\bfone_{\dag}}\circ\frakC_{\dag}\N_{\dag}(p).
  \end{align}
  Hence the composite above is exactly the morphism induced by the canonical collapse $\delta E_{\dag}\to\Delta^0_{\rev}$, under the natural identification $\bfL\frakC^{\prism}_{\dag}(\Delta^0_{\rev})\simeq\bfone_{\dag}$.
  Thus two-out-of-three proves the theorem.
\end{proof}

\begin{proposition}\label{jt.cor.local}
  For every fibrant $\calC\in\sCat^{\dag}_{\Bergner}$, the frame nerve $\N^{\prism}_{\dag}(\calC)$ is $S_{\rev}$-local in the derived sense, and every injectively fibrant replacement is an $S_{\rev}$-local fibrant object.
\end{proposition}

\begin{proof}
  By \cref{jt.lem.quillen}, $\frakC^{\prism}_{\dag}\dashv\N^{\prism}_{\dag}$ is a Quillen adjunction before localization.
  Hence \cite[Theorem~3.1.6 (1)(a)--(c)]{Hir02} reduces the derived $S_{\rev}$-locality of $\N^{\prism}_{\dag}(\calC)$ to proving that $\bfL\frakC^{\prism}_{\dag}$ sends every element of $S_{\rev}$ to a dagger DK-equivalence.

  First let $\sigma_m:\Sp[m]_{\rev}\to\Delta^m_{\rev}$ be a Segal generator, and let $s_m:\FdagSSet(\Sp[m])\to \FdagSSet(\Delta^m)$ be the corresponding free dagger spine inclusion.
  The ordinary spine inclusion is inner anodyne.
  Recall that free dagger completion sends the generating inner horns, and hence their weakly saturated closure to dagger Joyal trivial cofibrations; 
  so $s_m$ is a dagger Joyal trivial cofibration.
  Since $\frakC_{\dag}$ is left Quillen, $\frakC_{\dag}(s_m)$ is a dagger Bergner trivial cofibration.
  Naturality of \cref{jt.lem.transport-realization} identifies $\bfL\frakC^{\prism}_{\dag}(\sigma_m)$ with $\frakC_{\dag}(s_m)$ in $\Ho(\sCat^{\dag}_{\Bergner})$.

  For the completeness generator $e:\Delta^0_{\rev}\to E_{\rev}$, let $c:E_{\rev}\to\Delta^0_{\rev}$ be the canonical collapse.
  Under the natural identification $\bfL\frakC^{\prism}_{\dag}(\Delta^0_{\rev})\simeq\bfone_{\dag}$, the morphism $\bfL\frakC^{\prism}_{\dag}(c)$ is the dagger DK-equivalence of \cref{jt.thm.interval}.
  Since $ce=\id_{\Delta^0_{\rev}}$, $\bfL\frakC^{\prism}_{\dag}(e)$ is its inverse in $\Ho(\sCat^{\dag}_{\Bergner})$ and is also a dagger DK-equivalence.

  Thus \cite[Theorem~3.1.6 (1)]{Hir02} gives the asserted derived locality.
  Locality is invariant under the levelwise weak equivalence to an injectively fibrant replacement by \cref{h1.rem.S-comparison,h1.lem.local-objects} (3), so every such replacement is $S_{\rev}$-local fibrant.
\end{proof}

\begin{theorem}\label{jt.cor.localized-adjunction}
  The Quillen adjunction of \cref{jt.lem.quillen} descends to a Quillen adjunction
  \begin{align}
    \frakC^{\prism}_{\dag}:
    L_{\widetilde S_{\rev}}\sSpace^{\dag}_{\rev,\proj}
    \rightleftarrows
    \sCat^{\dag}_{\Bergner}:
    \N^{\prism}_{\dag}.
  \end{align}
\end{theorem}

\begin{proof}
  By \cref{jt.lem.quillen}, $\frakC^{\prism}_{\dag}\dashv\N^{\prism}_{\dag}$ is a Quillen adjunction before localization.
  Every $\widetilde s\in\widetilde S_{\rev}$ is a projectively cofibrant approximation to an element $s\in S_{\rev}$, and the proof of \cref{jt.cor.local} shows that $\frakC^{\prism}_{\dag}(\widetilde s)$ is a dagger DK-equivalence.
  Therefore \cite[Proposition~3.3.18 (1)]{Hir02} gives the asserted Quillen adjunction.
\end{proof}

\begin{corollary}\label{jt.cor.transport-local}
  For every cofibrant $X\in\sSet^{\dag}_{\Joyal}$ and every fibrant $\calC\in\sCat^{\dag}_{\Bergner}$, there exists a natural weak equivalence
  \begin{align}
    \RMap_{L_{S_{\rev}}\sSpace^{\dag}_{\rev,\inj}}
    (\delta X,\N^{\prism}_{\dag}(\calC))
    \simeq
    \RMap_{\sSet^{\dag}_{\Joyal}}(X,\N_{\dag}(\calC)).
  \end{align}
\end{corollary}

\begin{proof}
  By \cref{jt.cor.local}, an injectively fibrant replacement of $\N^{\prism}_{\dag}(\calC)$ is $S_{\rev}$-local.
  Mapping into this local object is unchanged by localization.
  The identity Quillen equivalence of \cref{h1.lem.rev-models} identifies the resulting projective and injective derived mapping spaces.
  The assertion now follows from \cref{jt.cor.transport-mapping}.
\end{proof}

\section{The dagger Joyal--Tierney theorem}\label{jt.section}

In this section, we prove that the constant--vertex adjunction induces a Quillen equivalence $\delta: \sSet^{\dag}_{\Joyal} \rightleftarrows L_{S_{\rev}}\sSpace^{\dag}_{\rev,\inj} :i_{0}^{*}$ (\cref{jt.thm.main}).
We first use the frame realization of \cref{jt.frames-section} to strictify local objects to frame nerves (\cref{jt.thm.strict}).
The final derived-unit argument then reduces to the dagger Joyal--Bergner equivalence proved above (\cref{jt.lem.unit}).
Combined with the change-of-index Quillen equivalence, this completes the chain of Quillen equivalences between $\sSet^{\dag}_{\Joyal}$ and $L_{u_{!}(\widetilde S_{\rev})}\sSpace^{\dag}_{\proj}$ (\cref{jt.cor.goal}).

\subsection{Strictification of local objects}

We show that every $S_{\rev}$-local object is connected by local equivalences to the frame nerve of a fibrant dagger simplicial category (\cref{jt.thm.strict}).

\begin{notation}
  We fix the following notational conventions:
  \begin{itemize}
    \item for an injectively fibrant object $Z\in\sSpace^{\dag}_{\rev}$ and vertices $x,y\in Z_0$, put
    \begin{align}
      \UEquiv_Z(x,y)
      :=
      \hofib_{(x,y)}(\underline{\Map}(E_{\rev},Z) \to Z_0\times Z_0).
    \end{align}
    If $Z$ is moreover $S_{\rev}$-local, we call $x$ and $y$ \emph{coherently unitarily equivalent} if $\UEquiv_Z(x,y)$ is nonempty;
    \item for $n>0$, let $\partial\yo_{\rev}[n]\subseteq\yo_{\rev}[n]$ denote the boundary;
    The reversal $\rho_n$ induces a $C_2$-action on this pair, and hence, by precomposition, on the associated mapping spaces;
    \item for an injectively fibrant $Z\in\sSpace^{\dag}_{\rev}$ and a $C_2$-fixed map $a:\partial\yo_{\rev}[n]\to Z$, put
    \begin{align}
      \Fill_Z(a)
      :=
      \fib_a(\underline{\Map}(\yo_{\rev}[n],Z)\to\underline{\Map}(\partial\yo_{\rev}[n],Z)).
    \end{align}
    Since $a$ is fixed, the space $\Fill_Z(a)$ inherits a $C_2$-action.
  \end{itemize}
\end{notation}

\begin{lemma}\label{jt.lem.fixed-cell}
  The filler space $\Fill_Z(a)$ is fibrant in $\Fun(BC_{2},\sSet)_{\inj}$.
  Consequently, its strict fixed points compute its homotopy fixed points.
\end{lemma}

\begin{proof}
  Let $i:A\to B$ be a trivial cofibration in $\Fun(BC_{2},\sSet)_{\inj}$.
  An equivariant lifting problem for $i$ against $\Fill_Z(a)\to *$ is adjoint to a lifting problem in $\sSpace^{\dag}_{\rev}$ against
  \begin{align}
    (
      (\yo_{\rev}[n]\times A)
      \amalg_{\partial\yo_{\rev}[n]\times A}
      (\partial\yo_{\rev}[n]\times B)
    )/C_2
    \to
    (\yo_{\rev}[n]\times B)/C_2.
  \end{align}
  The $C_2$-action on $\yo_{\rev}[n]\setminus\partial\yo_{\rev}[n]$ is free.
  Indeed, its simplices are represented by surjective morphisms $\alpha:[m]\to[n]$ in $\prism_{\rev}$.
  The equality $\rho_n\alpha=\alpha$ and surjectivity would imply $\rho_n=\id_{[n]}$, which is impossible for $n>0$.
  At each level, choose one representative from every free $C_2$-orbit in $\yo_{\rev}[n]\setminus\partial\yo_{\rev}[n]$.
  On the orbit generated by such a representative, the displayed orbit pushout--product is canonically isomorphic to the underlying map $A\to B$.
  Consequently, for every $[m]\in\prism_{\rev}$, after these choices the $[m]$-component of the displayed map is isomorphic to
  \begin{align}
    \id_{((\partial\yo_{\rev}[n])([m])\times B)/C_2}
    \amalg
    \coprod_{
      ((\yo_{\rev}[n])([m])
      \setminus
      (\partial\yo_{\rev}[n])([m]))/C_2
    }
    (A\to B).
  \end{align}
  It is a levelwise trivial monomorphism, so injective fibrancy of $Z$ supplies the required lift.
  Hence $\Fill_Z(a)$ is injectively fibrant as a $C_2$-simplicial set.

  The trivial-action functor $\triv : \sSet\to\Fun(BC_{2},\sSet)$ preserves cofibrations and trivial cofibrations for the injective structure.
  Its right adjoint, the fixed-point functor $(-)^{C_{2}} : \Fun(BC_{2},\sSet) \to \sSet$ , is therefore right Quillen.
  Since $\Fill_Z(a)$ is injectively fibrant, the strict fixed points of $\Fill_Z(a)$ compute homotopy fixed points.
\end{proof}

\begin{lemma}\label{jt.lem.uequiv-invariance}
  Let $f:Z\to Z'$ be a map between injectively fibrant $S_{\rev}$-local objects.
  If $\Map_Z(x,y)\to\Map_{Z'}(fx,fy)$ is a weak equivalence for every $x,y$, then $\UEquiv_Z(x,y)\to\UEquiv_{Z'}(fx,fy)$ is a weak equivalence for every $x,y$.
\end{lemma}

\begin{proof}
  Put $E_{\rev}^{(0)} =\Delta^0_{\rev}\amalg\Delta^0_{\rev}$ and $E_{\rev}^{(n)} := \delta(\sk_nE_{\dag})$.
  Then $E_{\rev}=\colim_nE_{\rev}^{(n)}$ is the filtration by dagger skeleta.
  For every $n>0$, the two nondegenerate $n$-simplices of $E_{\dag}=\N_{\dag}(\bbI_{\dag})$ are
  \begin{align}
    \sigma_n^0:=(0,1,0,1,\ldots)
    \quad \text{or} \quad
    \sigma_n^1:=(1,0,1,0,\ldots)
  \end{align}
  Reversal gives
  \begin{align}
    (\sigma_n^i)^{\dag}
    =
    \begin{cases}
      \sigma_n^{1-i},&n\text{ odd},\\
      \sigma_n^i,&n\text{ even}.
    \end{cases}
  \end{align}
  Let $F_n$ be a set of representatives of the free orbits and let $H_n$ be the set of fixed cells.
  Thus $F_n$ has one element and $H_n$ is empty when $n$ is odd, while $F_n$ is empty and $H_n$ has two elements when $n$ is positive and even.

  Put $Q^n:=\yo_{\rev}[n]/C_2$ and $\partial Q^n:=\partial\yo_{\rev}[n]/C_2$, where $C_2$ acts by the reversal $\rho_n$.
  The characteristic map of a fixed cell is $C_2$-invariant and therefore factors through $Q^n$.
  For every $n>0$, the following square is a pushout square:
  \begin{align*}
    &\begin{tikzpicture}[auto]
      \node (cell-boundaries) at (0,1.5) {$\displaystyle
        \coprod_{\sigma\in F_n}\partial\yo_{\rev}[n]
        \amalg
        \coprod_{\tau\in H_n}\partial Q^n$};
      \node (previous-skeleton) at (4,1.5) {$E_{\rev}^{(n-1)}$};
      \node (cells) at (0,0) {$\displaystyle
        \coprod_{\sigma\in F_n}\yo_{\rev}[n]
        \amalg
        \coprod_{\tau\in H_n}Q^n$};
      \node (skeleton) at (4,0) {$E_{\rev}^{(n)}$};
      \draw[->]
        (cell-boundaries) -- (previous-skeleton);
      \draw[->]
        (cell-boundaries) -- (cells);
      \draw[->]
        (previous-skeleton) -- (skeleton);
      \draw[->]
        (cells) -- (skeleton);
    \end{tikzpicture}
  \end{align*}

  For vertices $x,y\in Z_0$, put
  \begin{align}
    P_n^Z(x,y)
    &:=
    \fib_{(x,y)}
    (\underline{\Map}(E_{\rev}^{(n)},Z)
    \to
    \underline{\Map}(E_{\rev}^{(0)},Z)),
  \end{align}
  and define $P_n^{Z'}(fx,fy)$ similarly.
  Since the skeletal inclusions are monomorphisms, injective fibrancy and \cref{h1.lem.local-objects}~(1) show that
  \begin{align}
    P_n^Z(x,y)
    \to P_{n-1}^Z(x,y)
    \quad \text{and} \quad
    P_n^{Z'}(fx,fy)
    \to P_{n-1}^{Z'}(fx,fy)
  \end{align}
  are Kan fibrations.
  Moreover, we have
  \begin{align}
    P_0^Z(x,y)\cong P_0^{Z'}(fx,fy)\cong *.
  \end{align}

  We first compare the ordinary fixed-vertex filler spaces.
  For a vertex string $\bfz=(z_0,\ldots,z_n)$, put
  \begin{align}
    Z_n(\bfz)
    :=
    \fib_{\bfz}
    (Z_n\to Z_0^{n+1})
    \quad \text{and} \quad
    B_n^Z(\bfz)
    :=
    \fib_{\bfz}
    (\underline{\Map}(\partial\yo_{\rev}[n],Z)
      \to Z_0^{n+1}).
  \end{align}
  The Segal equivalences and the hypothesis give
  \begin{align}
    Z_n(\bfz)
    \simeq
    \prod_{j=1}^{n}\Map_Z(z_{j-1},z_j)
    \xrightarrow{\simeq}
    \prod_{j=1}^{n}\Map_{Z'}(fz_{j-1},fz_j)
    \simeq
    Z'_n(f\bfz),
  \end{align}
  where $f\bfz=(fz_0,\ldots,fz_n)$.

  Let $\calP_n$ be the category of nonempty proper subsets $S\subsetneq\{0,\ldots,n\}$, with a morphism $S\to T$ when $T\subseteq S$.
  Restriction to proper faces gives
  \begin{align}
    B_n^Z(\bfz)
    \cong
    \lim_{S\in\calP_n}
    Z_{|S|-1}(\bfz|_S),
  \end{align}
  and similarly for $Z'$.
  These face diagrams are Reedy fibrant, since their matching maps are fixed-vertex pullbacks of the Kan fibrations
  \begin{align}
    \underline{\Map}(\yo_{\rev}[|S|-1],Z)
    \to
    \underline{\Map}(\partial\yo_{\rev}[|S|-1],Z),
  \end{align}
  and similarly for $Z'$.
  The Segal equivalences and the hypothesis give an objectwise weak equivalence between these diagrams.
  Since $\calP_n$ is inverse Reedy, it has cofibrant constants and its limit functor is right Quillen by \cite[Proposition~15.10.2 (1) and Theorem~15.10.8 (1)]{Hir02}.
  Therefore
  \begin{align}
    B_n^Z(\bfz)
    \xrightarrow{\simeq}
    B_n^{Z'}(f\bfz).
  \end{align}
  The resulting square of boundary-restriction fibrations
  \begin{align*}
    &\begin{tikzpicture}[auto]
      \node (simplices) at (0,1.5) {$Z_n(\bfz)$};
      \node (boundaries) at (4,1.5) {$B_n^Z(\bfz)$};
      \node (target-simplices) at (0,0) {$Z'_n(f\bfz)$};
      \node (target-boundaries) at (4,0) {$B_n^{Z'}(f\bfz)$};
      \draw[->]
        (simplices) -- (boundaries);
      \draw[->]
        (simplices) -- (target-simplices);
      \draw[->]
        (boundaries) -- (target-boundaries);
      \draw[->]
        (target-simplices) -- (target-boundaries);
    \end{tikzpicture}
  \end{align*}
  has weak equivalences as its vertical maps.
  For a boundary value $a\in B_n^Z(\bfz)$, put
  \begin{align}
    \Fillop^{\rmord}_Z(a)
    &:=
    \fib_a
    (Z_n(\bfz)\to B_n^Z(\bfz)),
  \end{align}
  and define $\Fillop^{\rmord}_{Z'}(fa)$ analogously.
  Homotopy invariance of fibers gives a weak equivalence
  \begin{align}
    \Fillop^{\rmord}_Z(a)
    \xrightarrow{\simeq}
    \Fillop^{\rmord}_{Z'}(fa).
  \end{align}
  This is the fixed-endpoint analogue of the ordinary fiberwise argument in \cite[the proof of Theorem~6.2 in \S11]{Rez01}.

  We prove by induction on $n$ that $P_n^Z(x,y) \to P_n^{Z'}(fx,fy)$ is a weak equivalence.
  The assertion holds for $n=0$ because both sides are points.
  Fix $u\in P_{n-1}^Z(x,y)$ and denote the induced boundary values on the free and fixed cells by $a_{\sigma}$ and $a_{\tau}$.
  Applying $\underline{\Map}(-,Z)$ to the skeletal pushout gives
  \begin{align}
    \fib_u(P_n^Z(x,y)\to P_{n-1}^Z(x,y))
    \cong
    \prod_{\sigma\in F_n} \Fillop^{\rmord}_Z(a_{\sigma})
    \times
    \prod_{\tau\in H_n}\Fill_Z(a_{\tau})^{C_2}.
  \end{align}
  Indeed, on a free orbit a filler on one cell uniquely determines the filler on its dagger conjugate.
  On a fixed cell, the boundary value $a_{\tau}$ is $C_2$-fixed, and an extension through $Q^n$ is precisely a fixed point of $\Fill_Z(a_{\tau})$.

  The ordinary filler comparison gives weak equivalences on the factors indexed by $F_n$.
  For $\tau\in H_n$, the boundary value determines its vertex string, so there are canonical identifications
  \begin{align}
    \Fill_Z(a_{\tau})
    \cong
    \Fillop^{\rmord}_Z(a_{\tau})
    \quad \text{and} \quad
    \Fill_{Z'}(fa_{\tau})
    \cong
    \Fillop^{\rmord}_{Z'}(fa_{\tau}).
  \end{align}
  Hence the induced map $\Fill_Z(a_{\tau}) \to \Fill_{Z'}(fa_{\tau})$ is a $C_2$-equivariant underlying weak equivalence.
  Both sides are fibrant in $\Fun(BC_{2},\sSet)_{\inj}$ by \cref{jt.lem.fixed-cell}.
  Since strict fixed points compute homotopy fixed points, the induced map $\Fill_Z(a_{\tau})^{C_2} \to \Fill_{Z'}(fa_{\tau})^{C_2}$ is a weak equivalence.
  Since finite products preserve weak equivalences, every fiber map of the skeletal restriction is a weak equivalence.

  Consider the commutative square of Kan fibrations
  \begin{align*}
    &\begin{tikzpicture}[auto]
      \node (stage) at (0,1.5) {$P_n^Z(x,y)$};
      \node (previous-stage) at (4,1.5) {$P_{n-1}^Z(x,y)$};
      \node (target-stage) at (0,0) {$P_n^{Z'}(fx,fy)$};
      \node (target-previous-stage) at (4,0)
        {$P_{n-1}^{Z'}(fx,fy)$};
      \draw[->]
        (stage) -- (previous-stage);
      \draw[->]
        (stage) -- (target-stage);
      \draw[->]
        (previous-stage) -- (target-previous-stage);
      \draw[->]
        (target-stage) -- (target-previous-stage);
    \end{tikzpicture}
  \end{align*}
  The right vertical map is a weak equivalence by induction, and all the fiber maps are weak equivalences by the preceding calculation.
  Then the left vertical map is also a weak equivalence.

  Finally, the endpoint restrictions from $E_{\rev}$ to $E_{\rev}^{(0)}$ are Kan fibrations, so their strict fibers represent $\UEquiv$.
  Mapping spaces carry $E_{\rev}=\colim_nE_{\rev}^{(n)}$ to an inverse limit, and strict fibers commute with limits.
  Hence
  \begin{align}
    \UEquiv_Z(x,y)
    \cong
    \lim_nP_n^Z(x,y)
    \quad \text{and} \quad
    \UEquiv_{Z'}(fx,fy)
    \cong
    \lim_nP_n^{Z'}(fx,fy).
  \end{align}
  The two towers are Reedy fibrant because their matching maps are the displayed restriction Kan fibrations.
  The map between the two towers is a levelwise weak equivalence by the skeletal induction.
  Since $\bbN^{\myop}$ is inverse Reedy, its limit functor is right Quillen by \cite[Proposition~15.10.2 (1) and Theorem~15.10.8 (1)]{Hir02}.
  Taking inverse limits therefore gives
  \begin{align}
    \UEquiv_Z(x,y)
    \xrightarrow{\simeq}
    \UEquiv_{Z'}(fx,fy).
  \end{align}
\end{proof}

\begin{lemma}\label{jt.lem.uequiv-paths}
  If $Z$ is injectively fibrant and $S_{\rev}$-local, then for every $x,y\in Z_0$ there are natural weak equivalences 
  \begin{align}
    \UEquiv_Z(x,y) \simeq \Path_{Z_0}(x,y).
  \end{align}
\end{lemma}

\begin{proof}
  Evaluation at the first endpoint gives a Kan fibration $s_Z: \underline{\Map}(E_{\rev},Z) \to Z_0$, which is a weak equivalence by unitary completeness.
  The proof of \cite[Proposition~6.4]{Rez01}, with $\UEquiv$ in place of the ordinary homotopy-equivalence space, consequently gives the desired natural weak equivalences.
\end{proof}

\begin{proposition}\label{jt.lem.rezk-criterion}
  Let $f:Z\to Z'$ be a map between injectively fibrant $S_{\rev}$-local objects.
  Then $f$ is a weak equivalence in $L_{S_{\rev}}\sSpace^{\dag}_{\rev,\inj}$ if and only if:
  \begin{enumerate}
    \item $\Map_Z(x,y)\to\Map_{Z'}(fx,fy)$ is a weak equivalence for every $x,y$;
    \item every component of $Z'_0$ contains a vertex coherently unitarily equivalent to one in the image of $f$.
  \end{enumerate}
\end{proposition}

\begin{proof}
  For every $m\geq0$, write $f_m:=f_{[m]}:Z_m\to Z'_m$.

  Assume first that conditions (1) and (2) hold.
  We adapt the fiberwise argument in \cite[the proof of Proposition~7.6]{Rez01}.
  By \cref{jt.lem.uequiv-paths}, condition~(2) makes $\pi_0(f_0)$ surjective.
  By condition~(1) and \cref{jt.lem.uequiv-invariance,jt.lem.uequiv-paths}, there are natural weak equivalences
  \begin{align}
    \Path_{Z_0}(x,y)
    \simeq
    \UEquiv_Z(x,y)
    \xrightarrow{\simeq}
    \UEquiv_{Z'}(fx,fy)
    \simeq
    \Path_{Z'_0}(fx,fy).
  \end{align}
  By naturality, the displayed zigzag represents the map induced by $f_0$ on path spaces.
  It gives injectivity of $\pi_0(f_0)$ and, by taking $y=x$, weak equivalences on all based loop spaces.
  Since $Z_0$ and $Z'_0$ are Kan complexes, $f_0$ is a weak equivalence.

  The endpoint maps are Kan fibrations by \cref{h1.lem.local-objects}~(1).
  In the factorization
  \begin{align}
    Z_1
    \to
    (Z_0\times Z_0)\times_{Z'_0\times Z'_0}Z'_1
    \to
    Z'_1,
  \end{align}
  the first map is a weak equivalence by condition~(1) and the fiberwise criterion for Kan fibrations.
  The second map is the pullback of $f_0\times f_0$ along the right-hand endpoint fibration, and hence is a weak equivalence by right properness of $\sSet$.
  Thus $f_1$ is a weak equivalence.
  The Segal equivalences and \cref{h1.lem.local-objects}~(2) show that $f_m$ is a weak equivalence for every $m\geq2$.
  Hence $f$ is an underlying injective weak equivalence, and therefore a weak equivalence in $L_{S_{\rev}}\sSpace^{\dag}_{\rev,\inj}$.

  Conversely, suppose that $f$ is a weak equivalence in $L_{S_{\rev}}\sSpace^{\dag}_{\rev,\inj}$.
  Since $Z$ and $Z'$ are local fibrant objects, \cite[Theorem~3.2.13 (1)]{Hir02} implies that $f$ is an underlying injective weak equivalence.
  Homotopy invariance of the fibers of the endpoint fibrations gives condition~(1).
  Moreover, $f_0$ is surjective on components.
  Thus every vertex $z'\in(Z'_0)_0$ lies in the same component as some $fx$, and
  \cref{jt.lem.uequiv-paths} shows that $z'$ is coherently unitarily equivalent to $fx$.
  This proves condition~(2).
\end{proof}

\begin{notation}\label{jt.not.lambda-mate}
  Let $t_{\dag}^{!} : \sSet \to \sSpace$ denote the right adjoint of $t_{\dag,!}$, and let $\Lambda^{\flat}: t_{\dag}^{!}\N_{\dag} \to \N^{\prism}_{\dag}$ be the right mate of $\Lambda$.
  Forgetting the dagger gives a natural comparison
  \begin{align}\label{jt.eq.tdag-right-underlying}
    \vartheta_X:
    \res(t_{\dag}^{!}X)
    \to
    t^!(\UdagSSet X).
  \end{align}
\end{notation}

\begin{lemma}\label{jt.lem.tdag-ordered-fibers}
  For every $X\in\sSet^{\dag}$ and every ordered pair of vertices $(x,y)$, the map
  $\vartheta_X$ induces a strict isomorphism on the fibers over the corresponding constant endpoint simplices.
\end{lemma}

\begin{proof}
  In bidegree $(m,k)$, the two adjunctions identify
  \begin{align}
    (\res t_{\dag}^{!}X)_{m,k}
    &\cong
    \Hom_{\sSet^{\dag}}
    (t_{\dag,!}(\yo_{\rev}[m]\times\Delta^k),X),\\
    (t^!\UdagSSet X)_{m,k}
    &\cong
    \Hom_{\sSet}(\Delta^m\times J^k,\UdagSSet X).
  \end{align}
  Under these identifications, the component of
  \eqref{jt.eq.tdag-right-underlying} is
  \begin{align}
    \Hom_{\sSet^{\dag}}
    (t_{\dag,!}(\yo_{\rev}[m]\times\Delta^k),X)
    \to
    \Hom_{\sSet}(\Delta^m\times J^k,\UdagSSet X): 
    F\mapsto\UdagSSet(F)\circ j_+^k,
  \end{align}
  where
  \begin{align}
    j_+^k:
    \Delta^m\times J^k
    \to
    \UdagSSet t_{\dag,!}(\yo_{\rev}[m]\times\Delta^k)
  \end{align}
  is the canonical forward coprojection.
  These maps are compatible with the two adjunction bijections by construction.

  Since
  $\res(\yo_{\rev}[m])\cong\UdagSSet\FdagSSet(\Delta^m)$,
  \cite[Lemma~2.11]{JT06} and the formula for free dagger completion give
  \begin{align}
    \UdagSSet t_{\dag,!}
    (\yo_{\rev}[m]\times\Delta^k)
    \cong
    \UdagSSet\FdagSSet(\Delta^m)\times J^k
    \quad \text{and} \quad
    \UdagSSet\FdagSSet(\Delta^m)
    \cong
    \Delta^m
    \amalg_{\sk_0\Delta^m}
    (\Delta^m)^{\myop}.
  \end{align}
  Since $\sk_0\Delta^1=\partial\Delta^1$, we obtain
  \begin{align}\label{jt.eq.tdag-two-copies}
    \UdagSSet t_{\dag,!}(\yo_{\rev}[1]\times\Delta^k)
    \cong
    (\Delta^1\times J^k)_+
    \amalg_{\partial\Delta^1\times J^k}
    ((\Delta^1)^{\myop}\times J^k)_-.
  \end{align}

  For a simplicial set $S$ and a vertex $x\in S_0$, let
  $c_x^k:J^k\to S$ be the constant map.
  Thus, for $\xi\in J_n^k$, we have $(c_x^k)_n(\xi)=s^{(n)}x$, where $s^{(n)}x$ denotes the unique totally degenerate $n$-simplex at $x$.
  When $S=\UdagSSet X$, the dagger fixes vertices and hence also fixes every $s^{(n)}x$.

  Write
  \begin{align}
    P_{x,y}(X)
    :=
    \fib_{(x,y)}
    ((\res t_{\dag}^{!}X)_1
      \to(\res t_{\dag}^{!}X)_0\times
      (\res t_{\dag}^{!}X)_0).
  \end{align}
  The identification $\UdagSSet t_{\dag,!}(\yo_{\rev}[0]\times\Delta^k) \cong J^k$ shows that the $k$-simplices of this strict fiber are
  \begin{align}
    P_{x,y}(X)_k
    =
    \left\{
      F:t_{\dag,!}(\yo_{\rev}[1]\times\Delta^k)\to X
      \ \middle|\ 
      \begin{array}{l}
        F\text{ is a dagger simplicial map},\\
        \UdagSSet(F)|_{\{0\}\times J^k}=c_x^k,\\
        \UdagSSet(F)|_{\{1\}\times J^k}=c_y^k
      \end{array}
    \right\}.
  \end{align}

  Put
  \begin{align}
    Q_{x,y}(\UdagSSet X)
    :=
    \fib_{(x,y)}
    ((t^!\UdagSSet X)_1
      \to(t^!\UdagSSet X)_0\times
      (t^!\UdagSSet X)_0).
  \end{align}
  The formula for $t^!$ in \cite[\S2]{JT06} and the fixed-boundary identification of \cref{jt.not.ordered-mapping} give a strict natural isomorphism
  \begin{align}
    Q_{x,y}(\UdagSSet X)
    \cong
    \Map_E(\UdagSSet X;x,y).
  \end{align}
  Consequently,
  \begin{align}
    Q_{x,y}(\UdagSSet X)_k
    =
    \{f:\Delta^1\times J^k\to\UdagSSet X
      ~|~
      f|_{\{0\}\times J^k}=c_x^k, f|_{\{1\}\times J^k}=c_y^k\}.
  \end{align}

  The map induced by $\vartheta_X$ on the two fibers is the forward restriction
  \begin{align}
    \Theta_{x,y,k}:
    P_{x,y}(X)_k
    \to
    Q_{x,y}(\UdagSSet X)_k
    :
    F\mapsto\UdagSSet(F)\circ j_+^k.
  \end{align}
  We construct its inverse explicitly.
  In simplicial degree $n$, write the forward and reverse coprojections as
  \begin{align}
    j_{k,n}^+(\alpha,\xi)
    :=[\alpha,\xi,\id_{[1]}]
    \quad \text{and} \quad 
    j_{k,n}^-(\beta,\xi)
    :=[\beta,\xi,\rho_1],
  \end{align}
  where $\alpha,\beta:[n]\to[1]$ are order-preserving and $\xi\in J_n^k$.
  Formula \eqref{jt.eq.tdag-formula} becomes
  \begin{align}\label{jt.eq.tdag-copy-dagger}
    j_{k,n}^+(\alpha,\xi)^{\dag}
    =
    j_{k,n}^-(\rho_1\alpha\rho_n,\xi^{\vee})
    \quad \text{and} \quad 
    j_{k,n}^-(\beta,\xi)^{\dag}
    =
    j_{k,n}^+(\rho_1\beta\rho_n,\xi^{\vee}).
  \end{align}
  For the constant map $\underline i:[n]\to[1]$ with value $i$, the overlap relation in \eqref{jt.eq.tdag-two-copies} is
  \begin{align}\label{jt.eq.tdag-copy-overlap}
    j_{k,n}^+(\underline i,\xi)
    =
    j_{k,n}^-(\underline{1-i},\xi).
  \end{align}

  Let $f\in Q_{x,y}(\UdagSSet X)_k$.
  Define $\widetilde f: \UdagSSet t_{\dag,!}(\yo_{\rev}[1]\times\Delta^k) \to \UdagSSet X$ on the two copies by
  \begin{align}\label{jt.eq.tdag-extension}
    \widetilde f_n
    (j_{k,n}^+(\alpha,\xi))
    :=f_n(\alpha,\xi)
    \quad \text{and} \quad
    \widetilde f_n
    (j_{k,n}^-(\beta,\xi))
    :=
    [f_n(\rho_1\beta\rho_n,\xi^{\vee})]^{\dag}.
  \end{align}
  We verify that these formulas agree on the overlap.
  Put $x_0:=x$ and $x_1:=y$.
  By \eqref{jt.eq.tdag-copy-overlap} and the endpoint conditions on $f$, we have
  \begin{align}
    \widetilde f_n
    (j_{k,n}^-(\underline{1-i},\xi))
    =
    [f_n(\underline i,\xi^{\vee})]^{\dag}
    =
    (s^{(n)}x_i)^{\dag}
    =
    s^{(n)}x_i
    =
    f_n(\underline i,\xi)
    =
    \widetilde f_n
    (j_{k,n}^+(\underline i,\xi)).
  \end{align}
  The forward formula is simplicial because $f$ is simplicial.
  On the reverse copy, the formula is the composite of the domain dagger, $f$, and the codomain dagger.
  Hence it is also simplicial, and \eqref{jt.eq.tdag-extension} defines a simplicial map on the pushout
  \eqref{jt.eq.tdag-two-copies}.

  We show that it is dagger-equivariant.
  On a forward simplex, \eqref{jt.eq.tdag-copy-dagger} and \eqref{jt.eq.tdag-extension} give
  \begin{align}
    \widetilde f_n
    (j_{k,n}^+(\alpha,\xi)^{\dag})
    =
    \widetilde f_n
    (j_{k,n}^-(\rho_1\alpha\rho_n,\xi^{\vee}))
    =
    [f_n(\rho_1(\rho_1\alpha\rho_n)\rho_n,(\xi^{\vee})^{\vee})]^{\dag}
    =
    f_n(\alpha,\xi)^{\dag}
    =
    \widetilde f_n(j_{k,n}^+(\alpha,\xi))^{\dag}.
  \end{align}
  The reverse-copy identity follows by involutivity.
  Thus $\widetilde f$ is the underlying map of a dagger simplicial map $\Psi_{x,y,k}(f): t_{\dag,!}(\yo_{\rev}[1]\times\Delta^k) \to X$.

  This extension is unique.
  Indeed, if $F:t_{\dag,!}(\yo_{\rev}[1]\times\Delta^k) \to X$ is a dagger map with forward restriction $f$, then for every reverse simplex $z$ we have
  \begin{align}
    \UdagSSet(F)(z)
    =
    \UdagSSet(F)(z^{\dag})^{\dag}
    =
    f(z^{\dag})^{\dag}.
  \end{align}
  Therefore $\Theta_{x,y,k}\Psi_{x,y,k} = \id$ and $\Psi_{x,y,k}\Theta_{x,y,k} = \id$.

  For every $\theta:[\ell]\to[k]$, the induced functor of indiscrete groupoids commutes with inversion, so $J(\theta)(\xi^{\vee}) = J(\theta)(\xi)^{\vee}$.
  Hence the restriction maps $\Theta_{x,y,k}$ and the inverse formulas \eqref{jt.eq.tdag-extension} commute with all simplicial operators in $k$.
  They therefore assemble to a strict natural isomorphism
  \begin{align}\label{jt.eq.tdag-fixed-fiber}
    \fib_{(x,y)}
    ((\res t_{\dag}^{!}X)_1
      \to(\res t_{\dag}^{!}X)_0\times
      (\res t_{\dag}^{!}X)_0)
    &\cong
    \fib_{(x,y)}
    ((t^!\UdagSSet X)_1
      \to(t^!\UdagSSet X)_0\times
      (t^!\UdagSSet X)_0)\\
    &\cong
    \Map_E(\UdagSSet X;x,y).
  \end{align}
  The second isomorphism is the cited fixed-boundary identification of \cref{jt.not.ordered-mapping}.
\end{proof}

\begin{notation}\label{jt.not.relative-endpoint-data}
  Let $\calC$ be fibrant in $\sCat^{\dag}_{\Bergner}$ and fix objects $c,d\in\calC$.
  \begin{itemize}
    \item put $\bbB_{\dag}:=\bfone_{\dag}\amalg\bfone_{\dag}$, and regard $\calC$ as the object $(c,d):\bbB_{\dag}\to\calC$ of the ordered endpoint undercategory;
    \item let $\iota:\bbB_{\dag}\to\bbJ_{\dag}$ be the ordered endpoint inclusion;
    \item put $\calE^{\bullet}_{\dag}:=\FdagSCat\frakC(C_E^{\bullet})$, retaining the two endpoint objects, and write
    \begin{align}
      \calH_m(\calC)
      &:=[k\mapsto \Hom_{\sCat^{\dag}} (\Xi^k_{\dag}[m],\calC)],
      \\
      \calE_{\calC;c,d}
      &:=[k\mapsto \Hom_{\bbB_{\dag}\mathbin{\downarrow}\sCat^{\dag}}(\bbB_{\dag}\to\calE^k_{\dag},\bbB_{\dag}\xrightarrow{(c,d)}\calC)].
    \end{align}
    \item define augmented cosimplicial arrows
    \begin{align}
      G^{\bullet}
      :=(\Gamma^{\bullet}_{\dag}[0]\amalg\Gamma^{\bullet}_{\dag}[0]
        \to\Gamma^{\bullet}_{\dag}[1])
      \quad \text{and} \quad 
      X^{\bullet}
      :=(\Xi^{\bullet}_{\dag}[0]\amalg\Xi^{\bullet}_{\dag}[0]
        \to\Xi^{\bullet}_{\dag}[1]).
    \end{align}
    Then the natural transformation $\lambda^{\bullet}$ induces an augmentation-compatible morphism $\lambda^{\bullet}:G^{\bullet}\to X^{\bullet}$;
  \end{itemize}
\end{notation}

\begin{lemma}\label{jt.lem.relative-endpoint-frame}
  Let $\calC$ be fibrant in $\sCat^{\dag}_{\Bergner}$, and let $c,d\in\calC$.
  Then there exists a natural isomorphism
  \begin{align}
    \fib_{(c,d)}(\calH_1(\calC)\to\calH_0(\calC)^2)
    \cong
    \calE_{\calC;c,d}.
  \end{align}
\end{lemma}

\begin{proof}
  For every $m,k\geq0$, the defining adjunctions give natural bijections
  \begin{align}
    \calH_m(\calC)_k
    &=
    \Hom_{\sCat^{\dag}}(\frakC_{\dag}t_{\dag,!} (\yo_{\rev}[m]\times\Delta^k), \calC)
    \\
    &\cong
    \Hom_{\sSet^{\dag}}(t_{\dag,!}(\yo_{\rev}[m]\times\Delta^k),\N_{\dag}\calC)
    \\
    &=
    (\res t_{\dag}^{!}\N_{\dag}\calC)_{m,k}.
  \end{align}
  These bijections identify the two endpoint maps of $\calH_1(\calC)\to\calH_0(\calC)^2$ with those of
  \begin{align}
    (\res t_{\dag}^{!}\N_{\dag}\calC)_1
    \to
    (\res t_{\dag}^{!}\N_{\dag}\calC)_0 \times (\res t_{\dag}^{!}\N_{\dag}\calC)_0,
  \end{align}
  and carry $(c,d)$ to the corresponding pair of constant endpoint simplices.
  Hence \cref{jt.lem.tdag-ordered-fibers} and $\UdagSSet \N_{\dag}=N\UdagSCat$ give
  \begin{align}
    [\fib_{(c,d)}(\calH_1(\calC)\to\calH_0(\calC)^2)]_k
    &\cong
    \Map_E(N\UdagSCat\calC;c,d)_k\\
    &=
    \Hom_{\partial\Delta^1\mathbin{\downarrow}\sSet}(\partial\Delta^1\to C_E^k,\partial\Delta^1\xrightarrow{(c,d)}N\UdagSCat\calC) \\
    &\cong
    \Hom_{\frakC(\partial\Delta^1)\mathbin{\downarrow}\sCat}(\frakC(\partial\Delta^1)\to\frakC(C_E^k),\frakC(\partial\Delta^1)\xrightarrow{(c,d)}\UdagSCat\calC)\\
    &\cong
    \Hom_{\bbB_{\dag}\mathbin{\downarrow}\sCat^{\dag}}(\bbB_{\dag} \to \FdagSCat\frakC(C_E^k),\bbB_{\dag}\xrightarrow{(c,d)}\calC)\\
    &=
    (\calE_{\calC;c,d})_k.
  \end{align}
  The last two isomorphisms are induced by the adjunctions $\frakC\dashv N$ and $\FdagSCat\dashv\UdagSCat$.
  All these bijections commute with the cosimplicial operators in $k$.
\end{proof}

\begin{lemma}\label{jt.lem.relative-endpoint-cosimplicial-frame}
  The augmented cosimplicial object $(\bbB_{\dag}\to\calE^{\bullet}_{\dag})\to c(\iota)$ is a cosimplicial frame on $\iota$ in $\bbB_{\dag}\mathbin{\downarrow}\sCat^{\dag}$.
\end{lemma}

\begin{proof}
  By \cite[Proposition~4.5 (b)--(d)]{DS11},
  the augmentation $C_E^{\bullet}\to c\Delta^1$ is a cosimplicial resolution in $\partial\Delta^1\mathbin{\downarrow}\sSet_{\Joyal}$.
  In particular, each relative latching map $\ell_k:L^kC_E^{\bullet}\to C_E^k$ is a Joyal cofibration.

  Since $\frakC$ and $\FdagSCat$ preserve colimits, we have 
  \begin{align}
    \FdagSCat\frakC(\ell_k):
    L^k\calE^{\bullet}_{\dag} 
    = 
    \FdagSCat\frakC(L^kC_E^{\bullet})
    \to
    \FdagSCat\frakC(C_E^k)
    = 
    \calE^{\bullet}_{\dag}.
  \end{align}
  The functor $\frakC$ is left Quillen, so $\frakC(\ell_k)$ is a Bergner cofibration.
  \Cite[Lemma~3.2.7 (1)]{Aka26} then shows that $\FdagSCat\frakC(\ell_k)$ is a dagger Bergner cofibration.
  Hence $\calE^{\bullet}_{\dag}$ is Reedy cofibrant in the undercategory.

  Write $\epsilon_k:C_E^k\to\Delta^1$ for the augmentation.
  Choosing an object $j\in\{0,\ldots,k\}$ of the indiscrete groupoid defining $J^k$ gives a monomorphic section $s_{k,j}: \Delta^1 \to C_E^k$ with $\epsilon_ks_{k,j} = \id_{\Delta^1}$.
  Since $\epsilon_k$ is a Joyal equivalence, two-out-of-three shows that $s_{k,j}$ is a Joyal trivial cofibration.
  Both $\Delta^1$ and $C_E^k$ have vertex set $\{0,1\}$.
  Therefore $\frakC(s_{k,j})$ is an object-bijective Bergner trivial cofibration, \cite[Lemma~3.2.7 (2)]{Aka26} shows that $\FdagSCat\frakC(s_{k,j}):\bbJ_{\dag}\to\calE^k_{\dag}$ is a dagger Bergner trivial cofibration.
  Its retraction is the augmentation $\FdagSCat\frakC(\epsilon_k): \calE^k_{\dag}\to\bbJ_{\dag}$.
  Hence the augmentation is a dagger DK-equivalence by two-out-of-three.

  Finally, $C_E^0=\Delta^1$ and $\epsilon_0=\id_{\Delta^1}$.
  Thus the degree-zero object is $\iota:\bbB_{\dag}\to\bbJ_{\dag}$, and the assertion follows.
\end{proof}

\begin{lemma}\label{jt.lem.lambda-relative-arrows}
  The morphism
  $\lambda^{\bullet}:G^{\bullet}\to X^{\bullet}$
  admits an augmentation-compatible relative Reedy cofibrant replacement fixed in degree zero.
  For such a replacement, write
  $Q\lambda^{\bullet}:QG^{\bullet}\to QX^{\bullet}$.
  Write the source-to-target arrows of $QG^{\bullet}$ and $QX^{\bullet}$ as $QG_s^{\bullet}\to QG_t^{\bullet}$ and $QX_s^{\bullet}\to QX_t^{\bullet}$.
  Moreover, we put
  \begin{align}
    \overline G^{\bullet}
    :=c\bbB_{\dag} \amalg_{QG_s^{\bullet}} QG_t^{\bullet}
    \quad \text{and} \quad
    \overline X^{\bullet}
    := 
    c\bbB_{\dag} \amalg_{QX_s^{\bullet}} QX_t^{\bullet}.
  \end{align}
  Then these endpoint collapses are cosimplicial frames on $\iota$, and the induced comparison $\overline\lambda^{\bullet}: \overline G^{\bullet} \to \overline X^{\bullet}$ is an augmentation-compatible Reedy weak equivalence.
\end{lemma}

\begin{proof}
  Equip $\Fun([1],\sCat^{\dag}_{\Bergner})$ with its projective model structure.
  The object $\bbB_{\dag}$ is cofibrant, and the ordered endpoint inclusion $\iota:\bbB_{\dag}\to\bbJ_{\dag}$ is a dagger Bergner cofibration.
  Thus $\iota$ is a cofibrant arrow in $\Fun([1],\sCat^{\dag}_{\Bergner})$.
  Moreover, $\lambda^{\bullet}:G^{\bullet}\to X^{\bullet}$ is an augmented Reedy weak equivalence with degree-zero object $\iota$.

  Applying the functorial relative Reedy replacement of \cite[Proposition~16.6.8 (1)]{Hir02} in $\Fun([1],\sCat^{\dag}_{\Bergner})$ gives
  \begin{align}\label{jt.eq.relative-arrow-replacement}
    &\begin{tikzpicture}[auto]
      \node (replacement-source) at (0,1.5) {$QG^{\bullet}$};
      \node (replacement-target) at (4,1.5) {$QX^{\bullet}$};
      \node (source) at (0,0) {$G^{\bullet}$};
      \node (target) at (4,0) {$X^{\bullet}$};
      \draw[->]
        (replacement-source) --
        node[above] {$Q\lambda^{\bullet}$}
        (replacement-target);
      \draw[->]
        (replacement-source) --
        node[left] {$\sim$}
        (source);
      \draw[->]
        (replacement-target) --
        node[right] {$\sim$}
        (target);
      \draw[->]
        (source) --
        node {$\lambda^{\bullet}$}
        (target);
    \end{tikzpicture}
  \end{align}
  fixed in degree zero.
  Here $QG^{\bullet}$ and $QX^{\bullet}$ are Reedy cofibrant, and $Q\lambda^{\bullet}$ is a Reedy weak equivalence by two-out-of-three.

  It remains to check the endpoint collapses.
  By \cite[Theorem~15.5.2]{Hir02}, the maps $QG_s^{\bullet}\to QG_t^{\bullet}$ and $QX_s^{\bullet}\to QX_t^{\bullet}$ are Reedy cofibrations.
  Hence their cobase changes $c\bbB_{\dag}\to\overline G^{\bullet}$ and $c\bbB_{\dag}\to\overline X^{\bullet}$ are Reedy cofibrations.
  They are degreewise cofibrations by \cite[Proposition~15.3.11 (1)]{Hir02}.
  The source augmentations are degreewise weak equivalences, so left properness of $\sCat^{\dag}_{\Bergner}$ gives degreewise weak equivalences $QG_t^{\bullet}\xrightarrow{\sim}\overline G^{\bullet}$ and $QX_t^{\bullet}\xrightarrow{\sim}\overline X^{\bullet}$.
  The target augmentations are also degreewise weak equivalences to $c\bbJ_{\dag}$.
  Thus the induced augmentations $\overline G^{\bullet}\to c\bbJ_{\dag}$ and $\overline X^{\bullet}\to c\bbJ_{\dag}$ are degreewise weak equivalences by two-out-of-three.
  Since the replacement is fixed in degree zero, we have
  \begin{align}
    \overline G^0
    \cong 
    \overline X^0
    \cong
    \bbB_{\dag}
    \amalg_{\bbB_{\dag}}
    \bbJ_{\dag}
    \cong
    \bbJ_{\dag},
  \end{align}
  and the degree-zero structure map is $\iota$.
  They are therefore cosimplicial frames on $\iota$ in $\bbB_{\dag}\mathbin{\downarrow} \sCat^{\dag}_{\Bergner}$.

  Functoriality of the cobase change gives the augmentation-compatible map $\overline\lambda^{\bullet}:\overline G^{\bullet}\to\overline X^{\bullet}$.
  It is a Reedy weak equivalence by two-out-of-three.
  The choice-independence used below is \cite[Theorem~16.6.18 (1)]{Hir02}, applied in $\bbB_{\dag}\mathbin{\downarrow}\sCat^{\dag}_{\Bergner}$;
  equality of the induced total derived morphisms follows from \cite[Lemmas~5.5.1 and~5.5.2]{Hov99}.
\end{proof}

\begin{lemma}\label{jt.lem.xi-endpoint-collapse}
  There is a strict isomorphism
  \begin{align}\label{jt.eq.raw-Xi-collapse}
    c\bbB_{\dag}
    \amalg_{
      \Xi^{\bullet}_{\dag}[0]
      \amalg\Xi^{\bullet}_{\dag}[0]}
      \Xi^{\bullet}_{\dag}[1]
    \cong
    \calE^{\bullet}_{\dag}.
  \end{align}
  The induced morphism from $\overline X^{\bullet}$ to the left-hand side is a Reedy weak equivalence.
\end{lemma}

\begin{proof}
  The fixed-boundary calculation gives an isomorphism \eqref{jt.eq.raw-Xi-collapse}.
  The right vertical map of \eqref{jt.eq.relative-arrow-replacement} induces a morphism from $\overline X^{\bullet}$ to the left-hand side of \eqref{jt.eq.raw-Xi-collapse}.
  By \cref{jt.lem.relative-endpoint-cosimplicial-frame,jt.lem.lambda-relative-arrows}, both objects augment by degreewise weak equivalences to $c\bbJ_{\dag}$, so this morphism is a Reedy weak equivalence by two-out-of-three.
\end{proof}

\begin{proposition}\label{jt.lem.ordered-lambda-mate}
  If $\calC$ is fibrant in $\sCat^{\dag}_{\Bergner}$, then, for every $c,d\in\Ob(\calC)$, the map $\Lambda^{\flat}_{\calC}$ induces an isomorphism on the derived ordered fixed-endpoint fibers over $(c,d)$.
\end{proposition}

\begin{proof}
  Regard $\calC_{c,d}:=(\bbB_{\dag}\xrightarrow{(c,d)}\calC)$ as a fibrant object of $\bbB_{\dag}\mathbin{\downarrow}\sCat^{\dag}_{\Bergner}$.
  By \cref{jt.lem.relative-endpoint-frame,jt.lem.relative-endpoint-cosimplicial-frame,jt.lem.lambda-relative-arrows,jt.lem.xi-endpoint-collapse},
  $\overline G^{\bullet}$ and $\overline X^{\bullet}$ are cosimplicial frames on $\iota$ in this undercategory, and $\overline\lambda^{\bullet}:\overline G^{\bullet}\to\overline X^{\bullet}$ is an augmentation-compatible Reedy weak equivalence.
  Hence \cite[Corollary~16.5.5 (1)]{Hir02} gives a weak equivalence
  \begin{align}
    (\overline\lambda^{\bullet})^{*}:
    [k\mapsto
      \Hom_{\bbB_{\dag}\mathbin{\downarrow}\sCat^{\dag}}
      (\overline X^k,\calC_{c,d})]
    \xrightarrow{\simeq}
    [k\mapsto
      \Hom_{\bbB_{\dag}\mathbin{\downarrow}\sCat^{\dag}}
      (\overline G^k,\calC_{c,d})].
  \end{align}

  It remains only to identify this derived morphism.
  Under the adjunction identifications, the component of the right mate in bidegree $(m,k)$ is
  \begin{align}
    (\Lambda^{\flat}_{\calC})_{m,k}:
    \Hom_{\sCat^{\dag}}(\Xi^k_{\dag}[m],\calC)
    \to
    \Hom_{\sCat^{\dag}}(\Gamma^k_{\dag}[m],\calC) : 
    \widehat F
    \mapsto
    \widehat F\circ\lambda_m^k.
  \end{align}
  Thus, for $m=1$ and the fixed ordered endpoints $(c,d)$, $\Lambda^{\flat}_{\calC}$ is precomposition with the augmented morphism of cosimplicial arrows $\lambda^{\bullet}:G^{\bullet}\to X^{\bullet}$.
  The functorial relative Reedy replacement of \cite[Proposition~16.6.8 (1)]{Hir02}, followed by the functorial endpoint cobase change, carries this morphism to $\overline\lambda^{\bullet}$.
  Consequently, the total derived ordered fixed-endpoint morphism induced by $\Lambda^{\flat}_{\calC}$ is represented by $(\overline\lambda^{\bullet})^{*}$.
  It is therefore an isomorphism in $\Ho(\sSet)$.
\end{proof}

\begin{notation}\label{jt.not.ordered-frame-data}
  Let $W$ be injectively fibrant and $S_{\rev}$-local:
  \begin{itemize}
    \item use the functorial projective trivial fibration $p:Q_{\proj}W\to W$ with projectively cofibrant source and write the frame unit
    \begin{align}
      \varphi:
      Q_{\proj}W
      \to
      \N^{\prism}_{\dag}
      (R_{\rmob}\frakC^{\prism}_{\dag}(Q_{\proj}W));
    \end{align}
    \item let $t_{\dag,!}(Q_{\proj}W) \xrightarrow{g_W} T_W \to *$ be the functorial free-dagger factorization;
    \item let $r_W: \frakC_{\dag}(T_W) \to R_{\rmob}\frakC_{\dag}(T_W)$ be the chosen object-preserving fibrant-replacement map, and put 
    \begin{align}
      a_W:
      \frakC_{\dag}t_{\dag,!}(Q_{\proj}W)
      \xrightarrow{\frakC_{\dag}(g_W)}
      \frakC_{\dag}(T_W)
      \xrightarrow{r_W}
      R_{\rmob}\frakC_{\dag}(T_W);
    \end{align}
    \item under the rigidification identity, the underlying functor of $r_W$ is an object-preserving functor $\overline{\eta}_W: \frakC(\UdagSSet T_W) \to \UdagSCat(R_{\rmob}\frakC_{\dag}(T_W))$;
    \item let $\eta_W: \UdagSSet T_W \to N\UdagSCat(R_{\rmob}\frakC_{\dag}(T_W))$ be its adjoint.
    Since $r_W$ is the identity on objects, it is the identity on vertices;
    \item let $\widetilde g_W: \res(Q_{\proj}W) \to t^!(\UdagSSet T_W)$ be the actual adjoint of the underlying map $t_!(\res(Q_{\proj}W))\to\UdagSSet T_W$, and write $a_W^{\flat}: Q_{\proj}W \to t_{\dag}^{!}\N_{\dag}(R_{\rmob}\frakC_{\dag}(T_W))$ for the adjoint of $a_W$.
  \end{itemize}
\end{notation}

\begin{lemma}\label{jt.lem.ordered-rigidification-unit}
  In the situation of \cref{jt.not.ordered-frame-data}, for every ordered pair of vertices $x,y\in(\UdagSSet T_W)_0$, the map $\eta_W$ induces an isomorphism on derived ordered mapping spaces over $(x,y)$.
\end{lemma}

\begin{proof}
  The simplicial set $\UdagSSet T_W$ is a quasi-category, and $\UdagSCat(R_{\rmob}\frakC_{\dag}(T_W))$ is Bergner fibrant.
  Hence $N\UdagSCat(R_{\rmob}\frakC_{\dag}(T_W))$ is also a quasi-category.
  By \cref{jt.not.ordered-mapping}, the derived ordered mapping spaces are computed by $\Map_E(\UdagSSet T_W;x,y)$ and $\Map_E(N\UdagSCat(R_{\rmob}\frakC_{\dag}(T_W));x,y)$.
  
  For a quasi-category $S$ and vertices $u,v$, let $d_{S;u,v}$ denote the natural isomorphism from
  $\Map_E(S;u,v)$ to $\frakC(S)(u,v)$ in $\Ho(\sSet)$ supplied by \cite[Corollary~5.3]{DS11}.

  Naturality of this comparison and the adjunction identity $\varepsilon_{\UdagSCat(R_{\rmob}\frakC_{\dag}(T_W))} \circ \frakC(\eta_W) = \overline\eta_W$ give the following commutative diagram in $\Ho(\sSet)$:
  \begin{align}
    &\begin{tikzpicture}[auto]
        \node (source-mapping) at (0,3) {$
          \Map_E(\UdagSSet T_W;x,y)$};
        \node (target-mapping) at (7,3) {$
          \Map_E
          (N\UdagSCat(R_{\rmob}\frakC_{\dag}(T_W));x,y)$};
        \node (source-rigidification) at (0,1.5) {$
          \frakC(\UdagSSet T_W)(x,y)$};
        \node (target-rigidification) at (7,1.5) {$
          \frakC
          (N\UdagSCat(R_{\rmob}\frakC_{\dag}(T_W)))(x,y)$};
        \node (target-category) at (7,0) {$
          \UdagSCat
          (R_{\rmob}\frakC_{\dag}(T_W))(x,y)$};
        \draw[->]
          (source-mapping) --
          node[above] {$(\eta_W)_*$}
          (target-mapping);
        \draw[->]
          (source-mapping) --
          node[left] {$d_{\UdagSSet T_W;x,y}$}
          (source-rigidification);
        \draw[->]
          (target-mapping) --
          node[right] {$d_{N\UdagSCat(R_{\rmob}\frakC_{\dag}(T_W));x,y}$}
          (target-rigidification);
        \draw[->]
          (source-rigidification) --
          node[above] {$\frakC(\eta_W)_{x,y}$}
          (target-rigidification);
        \draw[->]
          (source-rigidification) --
          node[below] {$(r_W)_{x,y}$}
          (target-category);
        \draw[->]
          (target-rigidification) --
          node[right] {$
            (\varepsilon_{\UdagSCat(R_{\rmob}\frakC_{\dag}(T_W))})_{x,y}$}
          (target-category);
      \end{tikzpicture}%
  \end{align}
  The comparison morphisms and the counit morphism are isomorphisms by \cite[Corollary~5.3 and Proposition~5.9]{DS11}.
  The morphism $(r_W)_{x,y}$ is represented by a weak equivalence.
  Hence $(\eta_W)_*$ is an isomorphism.
\end{proof}

\begin{lemma}\label{jt.lem.strict-ordered-frame-mate}
  In the situation of \cref{jt.not.ordered-frame-data}, for every $x,y\in((Q_{\proj}W)_0)_0$, the maps
  \begin{align}
    a_W^{\flat}
    :
    Q_{\proj}W \to t_{\dag}^{!}\N_{\dag}(R_{\rmob}\frakC_{\dag}(T_W))
    \quad \text{and} \quad
    \Lambda^{\flat}_{R_{\rmob}\frakC_{\dag}(T_W)}
    :
    t_{\dag}^{!}\N_{\dag}
    (R_{\rmob}\frakC_{\dag}(T_W)) \to \N^{\prism}_{\dag}(R_{\rmob}\frakC_{\dag}(T_W))
  \end{align}
  induce isomorphisms on the derived ordered endpoint fibers over $(x,y)$.

  Moreover, there exists a commutative diagram
  \begin{align}\label{jt.eq.ordered-frame-mate}
    &\begin{tikzpicture}[auto]
        \node (source) at (0,1.5) {$Q_{\proj}W$};
        \node (intermediate) at (6,1.5) {$
          t_{\dag}^{!}\N_{\dag}
          (R_{\rmob}\frakC_{\dag}(T_W))$};
        \node (source-nerve) at (0,0) {$
          \N^{\prism}_{\dag}
          \frakC^{\prism}_{\dag}(Q_{\proj}W)$};
        \node (target-nerve) at (6,0) {$
          \N^{\prism}_{\dag}
          (R_{\rmob}\frakC_{\dag}(T_W))$};
        \draw[->]
          (source) --
          node[above] {$a_W^{\flat}$}
          (intermediate);
        \draw[->]
          (source) --
          node[left] {$\eta^{\prism}_{Q_{\proj}W}$}
          (source-nerve);
        \draw[->]
          (intermediate) --
          node[right] {$\Lambda^{\flat}_{R_{\rmob}\frakC_{\dag}(T_W)}$}
          (target-nerve);
        \draw[->]
          (source-nerve) --
          node {$\N^{\prism}_{\dag}(a_W\Lambda_{Q_{\proj}W})$}
          (target-nerve);
      \end{tikzpicture}
  \end{align}
  This diagram preserves the specified ordered pair of vertices $(x,y)$.
\end{lemma}

\begin{proof}
  Since $\UdagSSet t_{\dag,!}=t_!\res$, the underlying ordinary adjoint of $a_W$ is the composite
  \begin{align}
    \res(Q_{\proj}W)
    \xrightarrow{\widetilde g_W}
    t^!(\UdagSSet T_W)
    \xrightarrow{t^!(\eta_W)}
    t^!N\UdagSCat
    (R_{\rmob}\frakC_{\dag}(T_W)).
  \end{align}
  For $Y \in \sSpace$, define its derived ordered endpoint fiber by
  \begin{align}
    \calF_{x,y}(Y)
    :=
    \RFib_{(x,y)}(
      Y_1
      \to
      Y_0\times Y_0)
    \in
    \Ho(\sSet).
  \end{align}
  It is therefore invariant under Rezk weak equivalences which preserve the specified endpoints.

  By construction and \cref{jt.lem.free-inner-underlying}, the adjoint map $ t_!(\res(Q_{\proj}W)) \to \UdagSSet T_W$ is a vertex-preserving Joyal fibrant replacement.
  Hence $\widetilde g_W$ represents the derived unit of the Joyal--Tierney Quillen equivalence \cite[Theorem~4.12]{JT06}, and is a weak equivalence in $\sSpace_{\CSS}$.
  Since it preserves $x$ and $y$, homotopy invariance in the target variable of the ordered-endpoint homotopy function complex shows that $\widetilde g_W$ induces an isomorphism on the derived ordered endpoint fiber over $(x,y)$.
  By \cref{jt.not.ordered-mapping,jt.lem.ordered-rigidification-unit}, the same is true of $t^!(\eta_W)$.

  Compatibility of \eqref{jt.eq.tdag-right-underlying} with the two adjunction bijections, together with $\UdagSSet t_{\dag,!}=t_!\res$ and $\UdagSSet \N_{\dag}=N\UdagSCat$, gives 
  \begin{align}\label{jt.eq.underlying-adjoint-compatibility}
    \vartheta_{\N_{\dag}(R_{\rmob}\frakC_{\dag}(T_W))}
    \circ
    \res(a_W^{\flat})
    =
    t^!(\eta_W)
    \circ
    \widetilde g_W.
  \end{align}

  Applying $\calF_{x,y}$ to \eqref{jt.eq.underlying-adjoint-compatibility} gives the following commutative diagram in $\Ho(\sSet)$:
  \begin{align}
    &\begin{tikzpicture}[auto]
        \node (source) at (0,1.5) {$
          \calF_{x,y}(\res(Q_{\proj}W))$};
        \node (dagger-target) at (7,1.5) {$
          \calF_{x,y}
          (\res(t_{\dag}^{!}\N_{\dag}
          (R_{\rmob}\frakC_{\dag}(T_W))))$};
        \node (ordinary-source) at (0,0) {$
          \calF_{x,y}(t^!(\UdagSSet T_W))$};
        \node (ordinary-target) at (7,0) {$
          \calF_{x,y}
          (t^!N\UdagSCat
          (R_{\rmob}\frakC_{\dag}(T_W)))$};
        \draw[->]
          (source) --
          node[above] {$(\res(a_W^{\flat}))_*$}
          (dagger-target);
        \draw[->]
          (source) --
          node[left] {$(\widetilde g_W)_*$}
          (ordinary-source);
        \draw[->]
          (dagger-target) --
          node[right] {$
            (\vartheta_{\N_{\dag}(R_{\rmob}\frakC_{\dag}(T_W))})_*$}
          (ordinary-target);
        \draw[->]
          (ordinary-source) --
          node {$(t^!(\eta_W))_*$}
          (ordinary-target);
      \end{tikzpicture}
  \end{align}
  The morphisms labeled $(\widetilde g_W)_*$ and $(t^!(\eta_W))_*$ are isomorphisms by the preceding paragraph.
  By \cref{jt.lem.tdag-ordered-fibers}, $(\vartheta_{\N_{\dag}(R_{\rmob}\frakC_{\dag}(T_W))})_*$ is also an isomorphism;
  it is represented by a strict isomorphism of fixed-endpoint fibers.
  Therefore $(\res(a_W^{\flat}))_*$ is an isomorphism.
  Since the ordered endpoint fibers of a dagger object are computed after applying $\res$, this says precisely that $a_W^{\flat}$ induces an isomorphism on the derived ordered endpoint fibers over $(x,y)$.

  The simplicial category $R_{\rmob}\frakC_{\dag}(T_W)$ is dagger Bergner fibrant.
  Hence \cref{jt.lem.ordered-lambda-mate} implies that $\Lambda^{\flat}_{R_{\rmob}\frakC_{\dag}(T_W)}$ induces an isomorphism on the derived ordered fixed-endpoint fibers over $(x,y)$.
  The defining mate identity shows that the diagram \eqref{jt.eq.ordered-frame-mate} is commutative.
\end{proof}

\begin{theorem}\label{jt.thm.strict}
  Every $S_{\rev}$-local object $W$ is connected to a frame nerve by a zigzag of local equivalences.
  If $W$ is injectively fibrant, then there exists the zigzag of local equivalences:
  \begin{align}
    W
    \xleftarrow[\sim]{p}
    Q_{\proj}W
    \xrightarrow[\sim]{\varphi}
    \N^{\prism}_{\dag}
    (R_{\rmob}\frakC^{\prism}_{\dag}(Q_{\proj}W))
  \end{align}
\end{theorem}

\begin{proof}
  First suppose that $W$ is injectively fibrant.
  Use the maps fixed in \cref{jt.not.ordered-frame-data}, and put
  \begin{align}
    \psi_W
    :=
    \N^{\prism}_{\dag}
    (a_W\Lambda_{Q_{\proj}W})
    \circ
    \eta^{\prism}_{Q_{\proj}W}:
    Q_{\proj}W
    \to
    \N^{\prism}_{\dag}
    (R_{\rmob}\frakC_{\dag}(T_W)).
  \end{align}
  The mate identity of \cref{jt.lem.strict-ordered-frame-mate} gives
  \begin{align}\label{jt.eq.psi-mate}
    \psi_W
    =
    \Lambda^{\flat}_{R_{\rmob}\frakC_{\dag}(T_W)}
    \circ
    a_W^{\flat}.
  \end{align}
  Hence, for every $x,y\in((Q_{\proj}W)_0)_0$, the map induced by $\psi_W$ on the derived ordered fixed-endpoint fibers over $(x,y)$ is an isomorphism in $\Ho(\sSet)$.

  The map $g_W:t_{\dag,!}(Q_{\proj}W)\to T_W$ preserves vertices, and the map $r_W: \frakC_{\dag}(T_W) \to R_{\rmob}\frakC_{\dag}(T_W)$ preserves objects.
  Together with $\lambda_m^0=\id_{A_m}$, this gives
  \begin{align}
    \Ob(R_{\rmob}\frakC_{\dag}(T_W))
    =
    (\UdagSSet T_W)_0
    =
    ((Q_{\proj}W)_0)_0,
  \end{align}
  and $\psi_W$ is the identity under this identification of vertices.

  Choose an injective fibrant replacement $q : \N^{\prism}_{\dag}(R_{\rmob}\frakC_{\dag}(T_W)) \to Z$ along a rowwise weak equivalence.
  By \cref{jt.cor.local}, $Z$ is local fibrant.
  The map $p$ is a rowwise weak equivalence, hence an isomorphism in $\Ho (L_{S_{\rev}}\sSpace^{\dag}_{\rev,\inj})$, and all objects of $L_{S_{\rev}}\sSpace^{\dag}_{\rev,\inj}$ are cofibrant.
  Therefore the morphism $[q][\psi_W][p]^{-1} : [W] \to [Z]$ is represented by an actual map $f:W\to Z$ between cofibrant--fibrant objects.
  The equality $[f][p] = [q][\psi_W]$ is represented by a simplicial homotopy from $Q_{\proj}W$ to $Z$ by \cref{h1.lem.localized-simplicial} and \cite[Proposition~9.5.24 (2)]{Hir02}.

  For every $x,y\in((Q_{\proj}W)_0)_0$, homotopy invariance of derived endpoint fibers, together with \eqref{jt.eq.psi-mate}, gives a zigzag
  \begin{align}
    \Map_W(px,py)
    &\simeq
    \RFib_{(x,y)}
    ((Q_{\proj}W)_1
      \to
      (Q_{\proj}W)_0\times(Q_{\proj}W)_0)\notag\\
    &\xrightarrow{\simeq}
    \RFib_{(\psi_Wx,\psi_Wy)}
    ((\N^{\prism}_{\dag}
        (R_{\rmob}\frakC_{\dag}(T_W)))_1
      \to
      (\N^{\prism}_{\dag}
        (R_{\rmob}\frakC_{\dag}(T_W)))_0^2)\notag\\
    &\simeq
    \Map_Z(q\psi_Wx,q\psi_Wy).
  \end{align}

  The preceding simplicial homotopy gives paths from $f(px)$ to $q\psi_Wx$ and from $f(py)$ to $q\psi_Wy$.
  By \cref{jt.lem.uequiv-paths}, these paths are represented by coherent unitaries;
  composition with them induces weak equivalences on mapping spaces by \cite[Corollary~11.5]{Rez01}.

  Naturality of the endpoint-fiber comparisons therefore shows that $\Map_W(px,py) \to \Map_Z(f(px),f(py))$ is a weak equivalence.
  The projective cofibrant replacement $p:Q_{\proj}W\to W$ is a projective trivial fibration.
  Hence $p_{0,0}:((Q_{\proj}W)_0)_0\to(W_0)_0$ is surjective, and the preceding argument proves condition~(1) of \cref{jt.lem.rezk-criterion} for every pair of vertices of $W$.

  To verify condition (2), let $\zeta\in(Z_0)_0$.
  Since $\pi_0(q_0)$ is surjective, there are a vertex $d$ in the source of $q_0$ and a path $\zeta\simeq qd$.
  The identification of vertex sets gives $d=\psi_Wx$ for some $x\in((Q_{\proj}W)_0)_0$, while the preceding simplicial homotopy gives a path $q\psi_Wx\simeq f(px)$.
  Thus
  \begin{align}
    \zeta
    \simeq
    qd
    =
    q\psi_Wx
    \simeq
    f(px).
  \end{align}
  Hence $\zeta$ and $f(px)$ lie in the same path component of $Z_0$.
  By \cref{jt.lem.uequiv-paths}, they are coherently unitarily equivalent, which verifies condition~(2) of \cref{jt.lem.rezk-criterion}.
  It follows from \cref{jt.lem.rezk-criterion} that $f$ is a weak equivalence in $L_{S_{\rev}}\sSpace^{\dag}_{\rev,\inj}$.
  Since $W$ and $Z$ are local fibrant, \cite[Theorem~3.2.13 (1)]{Hir02} then implies that $f$ is a weak equivalence in $\sSpace^{\dag}_{\rev,\inj}$, and hence a rowwise weak equivalence.
  Thus $[q][\psi_W][p]^{-1}$ is an isomorphism in $\Ho (L_{S_{\rev}}\sSpace^{\dag}_{\rev,\inj})$.
  Since $p$ and $q$ are local equivalences, $[\psi_W]$ is an isomorphism;
  thus $\psi_W$ is a local equivalence.

  Let $j_W : \frakC^{\prism}_{\dag}(Q_{\proj}W) \to R_{\rmob}\frakC^{\prism}_{\dag}(Q_{\proj}W)$ be the object-preserving fibrant-replacement map whose adjoint is $\varphi$.
  The map $j_W$ is a dagger DK-equivalence by construction.
  The map
  \begin{align}
    a_W\Lambda_{Q_{\proj}W}:
    \frakC^{\prism}_{\dag}(Q_{\proj}W)
    \to
    R_{\rmob}\frakC_{\dag}(T_W)
  \end{align}
  is a dagger DK-equivalence by \cref{jt.lem.frame-comparison,jt.lem.free-inner-trivial}.
  Hence
  \begin{align}
    h_W
    :=
    [a_W\Lambda_{Q_{\proj}W}][j_W]^{-1}:
    R_{\rmob}
    \frakC^{\prism}_{\dag}(Q_{\proj}W)
    \xrightarrow{\cong}
    R_{\rmob}\frakC_{\dag}(T_W)
  \end{align}
  is an isomorphism in $\Ho(\sCat^{\dag}_{\Bergner})$.
  The morphisms $[\varphi]$ and $[\psi_W]$ are the derived adjoints of $[j_W]$ and
  $[a_W\Lambda_{Q_{\proj}W}]$, respectively.
  Under the derived identity equivalence of \cref{h1.lem.rev-models}, naturality of the derived adjunction therefore gives
  \begin{align}
    \bfR\N^{\prism}_{\dag}(h_W)
    \circ
    [\varphi]
    =
    [\psi_W].
  \end{align}
  The first factor and the right-hand side are isomorphisms.
  It follows that $[\varphi]$ is an isomorphism, so $\varphi$ is a local equivalence.

  For an arbitrary $S_{\rev}$-local $W$, first choose an injectively fibrant replacement along a local equivalence and apply the preceding construction to its fibrant replacement object.
\end{proof}

\begin{corollary}\label{jt.cor.frame-detection}
  Let $u:X\to Y$ be a morphism in $\sSpace^{\dag}_{\rev}$.
  Then $u$ is an $S_{\rev}$-local equivalence if and only if, for every fibrant $\calC\in\sCat^{\dag}_{\Bergner}$, the induced map
  \begin{align}
    \RMap(Y,\N^{\prism}_{\dag}(\calC))
    \to 
    \RMap(X,\N^{\prism}_{\dag}(\calC))
  \end{align}
  is a weak equivalence, where the derived mapping spaces are computed in
  $L_{S_{\rev}}\sSpace^{\dag}_{\rev,\inj}$.
\end{corollary}

\begin{proof}
  Suppose first that $u$ is an $S_{\rev}$-local equivalence.
  Choose an injectively fibrant replacement $\N^{\prism}_{\dag}(\calC) \xrightarrow{\sim} Z_{\calC}$ along a rowwise weak equivalence.
  By \cref{jt.cor.local}, $Z_{\calC}$ is $S_{\rev}$-local fibrant.
  Thus the displayed map is a weak equivalence.

  Conversely, suppose that the displayed map is a weak equivalence for every fibrant $\calC\in\sCat^{\dag}_{\Bergner}$.
  Let $Z$ be arbitrary fibrant of $L_{S_{\rev}}\sSpace^{\dag}_{\rev,\inj}$.
  By \cref{jt.thm.strict}, there are a fibrant $\calC\in\sCat^{\dag}_{\Bergner}$ and a zigzag of local equivalences connecting $Z$ to $\N^{\prism}_{\dag}(\calC)$.
  Homotopy invariance in the target variable identifies the map induced by $u$ into $Z$ with the map induced by $u$ into $\N^{\prism}_{\dag}(\calC)$ in $\Ho(\sSet)$.
  The latter is a weak equivalence by hypothesis.
  Since $Z$ was arbitrary, $u$ is an $S_{\rev}$-local equivalence.
\end{proof}

\subsection{The main theorem}

We now return to the constant--vertex adjunction (\cref{jt.lem.star}).
The rows formula gives reflection by $i_0^{*}$ (\cref{jt.cor.i0-reflects}), and the derived unit is identified with that of the dagger Joyal--Bergner adjunction (\cref{jt.lem.unit}).
These ingredients yield the main Quillen equivalence (\cref{jt.thm.main}).

\begin{lemma}\label{jt.lem.delta-local}
  For every weak equivalence $f:X\to Y$ between cofibrant objects of $\sSet^{\dag}_{\Joyal}$, the map $\delta(f)$ is a weak equivalence in $L_{S_{\rev}}\sSpace^{\dag}_{\rev,\inj}$.
\end{lemma}

\begin{proof}
  By \cref{jt.cor.frame-detection}, it suffices to show that, for every fibrant object $\calC\in\sCat^{\dag}_{\Bergner}$, the induced map
  \begin{align}
    (\delta f)^{*} : 
    \RMap(\delta Y,\N^{\prism}_{\dag}(\calC))
    \to 
    \RMap(\delta X,\N^{\prism}_{\dag}(\calC))
  \end{align}
  is a weak equivalence.
  By \cref{jt.cor.transport-local}, this map is identified with
  \begin{align}
    f^{*} : 
    \RMap(Y,\N_{\dag}(\calC))
    \to 
    \RMap(X,\N_{\dag}(\calC)).
  \end{align}
  The latter is a weak equivalence because $f:X\to Y$ is a weak equivalence between cofibrant objects and $\N_{\dag}(\calC)$ is fibrant.
\end{proof}

\begin{lemma}\label{jt.lem.star}
  The adjunction $\delta \dashv i_{0}^{*}$ of \cref{jt.lem.c-delta} induces a Quillen adjunction
  \begin{align}
    \delta:
    \sSet^{\dag}_{\Joyal}
    \rightleftarrows
    L_{S_{\rev}}\sSpace^{\dag}_{\rev,\inj}
    :i_{0}^{*}.
  \end{align}
\end{lemma}

\begin{proof}
  By \cite[Proposition~7.15]{JT06}, it is enough to prove that $\delta$ preserves cofibrations and that $i_{0}^{*}$ preserves fibrations between fibrant objects.

  The functor $\delta$ sends every free cofibration to a levelwise injection, hence to a cofibration of $L_{S_{\rev}}\sSpace^{\dag}_{\rev,\inj}$.

  The generating cofibrations $I^{\sSet}_{\dag}$ of $\sSet^{\dag}_{\Joyal}$ have cofibrant domains.
  Since $\sSet^{\dag}_{\Joyal}$ is combinatorial, \cite[Corollary~1.12]{Bar07} implies that it is tractable.
  
  Choose a set $J$ of generating trivial cofibrations with cofibrant domains.
  The codomain of every member of $J$ is also cofibrant, because every member of $J$ is a cofibration.

  For every $j:A\to B$ in $J$, \cref{jt.lem.delta-local} shows that $\delta(j)$ is a local equivalence, because $j$ is a weak equivalence between cofibrant objects.
  It is also a levelwise injection, hence a cofibration.
  Thus $\delta(j)$ is a trivial cofibration of $L_{S_{\rev}}\sSpace^{\dag}_{\rev,\inj}$.
  
  Let $q:Z\to Z'$ be a fibration between fibrant objects of $L_{S_{\rev}}\sSpace^{\dag}_{\rev,\inj}$.
  We show that $i_{0}^{*}(q)$ has the right lifting property with respect to every $j \in J$.
  By adjunction, it suffices to show that $q$ has the right lifting property with respect to $\delta(j)$.
  This follows from the preceding argument.
  Therefore $i_{0}^{*}(q)$ is a fibration in $\sSet^{\dag}_{\Joyal}$.
\end{proof}

\begin{lemma}\label{jt.lem.rows}
  For a fibrant $Z\in L_{S_{\rev}}\sSpace^{\dag}_{\rev,\inj}$, there exists a natural weak equivalence
  \begin{align}
    Z_m
    \simeq \RMap_{\sSet^{\dag}_{\Joyal}}(\FdagSSet(\Delta^m),\,i_{0}^{*}Z).
  \end{align}
\end{lemma}

\begin{proof}
  By \cref{h1.lem.localized-simplicial}, $Z_m=\underline{\Map}(\yo_{\rev}[m],Z)$ computes $\RMap_{L_{S_{\rev}}\sSpace^{\dag}_{\rev,\inj}}(\delta \FdagSSet\Delta^m,Z)$:
  every object of $L_{S_{\rev}}\sSpace^{\dag}_{\rev,\inj}$ is cofibrant, $Z$ is fibrant, and $\yo_{\rev}[m]=\delta \FdagSSet\Delta^m$.
  The derived adjunction of \cref{jt.lem.star} identifies this with $\RMap(\FdagSSet\Delta^m,i_{0}^{*}Z)$;
  here $\FdagSSet(\Delta^m)$ is cofibrant and $i_0^*Z$ is fibrant because $i_0^*$ is right Quillen.
  Therefore we have 
  \begin{align}
    Z_{m}
    \simeq \RMap_{L_{S_{\rev}}\sSpace^{\dag}_{\rev,\inj}}(\delta \FdagSSet\Delta^m,Z) 
    \simeq \RMap_{\sSet^{\dag}_{\Joyal}}(\FdagSSet\Delta^m,i_{0}^{*}Z).
  \end{align}
\end{proof}

\begin{proposition}\label{jt.cor.i0-reflects}
  The right adjoint $i_{0}^{*}$ of \cref{jt.lem.star} reflects weak equivalences in $L_{S_{\rev}}\sSpace^{\dag}_{\rev,\inj}$ between fibrant objects.
\end{proposition}

\begin{proof}
  Let $f:Z\to Z'$ be a map between fibrant objects and suppose that
  $i_{0}^{*}f$ is a weak equivalence in $\sSet^{\dag}_{\Joyal}$.
  For every $m$, it induces a weak equivalence
  \begin{align}
    \RMap_{\sSet^{\dag}_{\Joyal}}
    (\FdagSSet(\Delta^m),i_{0}^{*}Z)
    \to
    \RMap_{\sSet^{\dag}_{\Joyal}}
    (\FdagSSet(\Delta^m),i_{0}^{*}Z').
  \end{align}
  By naturality of the weak equivalences of \cref{jt.lem.rows}, the preceding map corresponds to $f_m:Z_m\to Z'_m$.
  Hence $f_m$ is a weak equivalence for every $m$.
  Thus $f$ is a rowwise weak equivalence, hence an injective weak equivalence.
  Every injective weak equivalence is an $S_{\rev}$-local equivalence, so $f$ is a weak equivalence in $L_{S_{\rev}}\sSpace^{\dag}_{\rev,\inj}$.
\end{proof}

\begin{notation}
  We recall the Quillen adjunctions which we have obtained.
  We identify $\Ho(L_{\widetilde S_{\rev}}\sSpace^{\dag}_{\rev,\proj})$ with $\Ho(L_{S_{\rev}}\sSpace^{\dag}_{\rev,\inj})$ by the derived identity Quillen equivalence of \cref{h1.lem.rev-models}.
  The relevant derived adjunctions are
  \begin{align}
    \bfL\delta:
    \Ho(\sSet^{\dag}_{\Joyal})
    &\rightleftarrows
    \Ho(L_{S_{\rev}}\sSpace^{\dag}_{\rev,\inj})
    :\bfR i_0^*,\\
    \bfL\frakC^{\prism}_{\dag}:
    \Ho(L_{\widetilde S_{\rev}}\sSpace^{\dag}_{\rev,\proj})
    &\rightleftarrows
    \Ho(\sCat^{\dag}_{\Bergner})
    :\bfR\N^{\prism}_{\dag},\\
    \bfL\frakC_{\dag}:
    \Ho(\sSet^{\dag}_{\Joyal})
    &\rightleftarrows
    \Ho(\sCat^{\dag}_{\Bergner})
    :\bfR \N_{\dag}.
  \end{align}
  The first two give the composite derived adjunction
  \begin{align}
    \bfL\frakC^{\prism}_{\dag} \circ\bfL\delta
    : \Ho(\sSet^{\dag}_{\Joyal})
    \rightleftarrows
    \Ho(\sCat^{\dag}_{\Bergner}) : 
    \bfR i_0^* \circ\bfR\N^{\prism}_{\dag}.
  \end{align}
  We let 
  \begin{itemize}
    \item $\overline\Theta_X:\bfL\frakC^{\prism}_{\dag}(\delta X) \xrightarrow{\simeq} \bfL\frakC_{\dag}(X)$ denote the derived transport comparison;
    \item $\overline\vartheta_{\calC}: \bfR i_0^*\bfR\N^{\prism}_{\dag}(\calC) \xrightarrow{\simeq} \bfR \N_{\dag}(\calC)$ denote the inverse of the derived right mate of $\overline\Theta$.
  \end{itemize}
\end{notation}

\begin{lemma}\label{jt.lem.unit-triangle}
  Let $X$ be cofibrant in $\sSet^{\dag}_{\Joyal}$ and let $\calC$ be fibrant in $\sCat^{\dag}_{\Bergner}$.
  For a morphism $a:\bfL\frakC_{\dag}(X)\to\calC$ in $\Ho(\sCat^{\dag}_{\Bergner})$, let $\widetilde a:\delta X\to\bfR\N^{\prism}_{\dag}(\calC)$ be the derived frame adjoint of $a\circ\overline\Theta_X$.
  Then we have 
  \begin{align}
    \overline\vartheta_{\calC}\circ\bfR i_0^*(\widetilde a)\circ\eta_X^{\delta} = a^{\sharp}
  \end{align}
  in $\Ho(\sSet^{\dag}_{\Joyal})$.
  Here $\eta_X^{\delta}$ is the derived unit of $\delta\dashv i_0^*$, and $a^{\sharp}$ is the adjoint of $a$ under $\frakC_{\dag}\dashv \N_{\dag}$.
\end{lemma}

\begin{proof}
  For the bar augmentation $\varepsilon_X: B_{\proj}(\delta X) \to \delta X$, \cref{jt.lem.density} gives
  \begin{align}
    \beta_X = \Theta_X \circ \frakC^{\prism}_{\dag}(\varepsilon_X).
  \end{align}
  The equality of \cref{jt.lem.density} identifies $\overline\Theta$ with the total left derived natural transformation represented by the strict comparison $\Theta$.
  The strict right-mate statement is the assertion of \cref{jt.lem.transport}.
  It identifies $\vartheta^{-1}:\N_{\dag}\to i_0^*\N^{\prism}_{\dag}$ as the strict right mate of $\Theta$.
  Consequently, the derived right mate has components $\overline\vartheta_{\calC}^{-1}: \bfR \N_{\dag}(\calC) \to \bfR i_0^*\bfR\N^{\prism}_{\dag}(\calC)$.

  Therefore the following square of adjunction bijections commutes:
  \begin{align}
    &\begin{tikzpicture}[auto]
        \node (rigidification-hom) at (0,1.5) {$
          \Hom_{\Ho(\sCat^{\dag}_{\Bergner})}
          (\bfL\frakC_{\dag}(X),\calC)$};
        \node (frame-rigidification-hom) at (7,1.5) {$
          \Hom_{\Ho(\sCat^{\dag}_{\Bergner})}
          (\bfL\frakC^{\prism}_{\dag}(\bfL\delta(X)),\calC)$};
        \node (nerve-hom) at (0,0) {$
          \Hom_{\Ho(\sSet^{\dag}_{\Joyal})}
          (X,\bfR \N_{\dag}(\calC))$};
        \node (frame-nerve-hom) at (7,0) {$
          \Hom_{\Ho(\sSet^{\dag}_{\Joyal})}
          (X,\bfR i_0^*\bfR\N^{\prism}_{\dag}(\calC))$};
        \draw[->]
          (rigidification-hom) --
          node[above] {$-\circ\overline\Theta_X$}
          (frame-rigidification-hom);
        \draw[->]
          (rigidification-hom) --
          node[left] {$\simeq$}
          (nerve-hom);
        \draw[->]
          (frame-rigidification-hom) --
          node[right] {$\simeq$}
          (frame-nerve-hom);
        \draw[->]
          (nerve-hom) --
          node {$\overline\vartheta_{\calC}^{-1}\circ-$}
          (frame-nerve-hom);
      \end{tikzpicture}%
  \end{align}

  Let $\eta^{\prism}$ denote the derived unit of $\bfL\frakC^{\prism}_{\dag}\dashv\bfR\N^{\prism}_{\dag}$.
  Since $\widetilde a$ is the derived frame adjoint of $a\circ\overline\Theta_X$, we have
  \begin{align}
    \widetilde a
    =
    \bfR\N^{\prism}_{\dag}
    (a\circ\overline\Theta_X)
    \circ
    \eta^{\prism}_{\bfL\delta(X)}.
  \end{align}
  Under the convention identifying $\delta X$ with the chosen representative of $\bfL\delta(X)$, this is the morphism denoted by $\widetilde a$ in the statement.
  Its adjoint under the composite derived adjunction is therefore
  \begin{align}
    X
    \xrightarrow{\eta_X^{\delta}}
    \bfR i_0^*\bfL\delta(X)
    \xrightarrow{\bfR i_0^*(\widetilde a)}
    \bfR i_0^*\bfR\N^{\prism}_{\dag}(\calC).
  \end{align}

  Apply the preceding commutative square to $a:\bfL\frakC_{\dag}(X) \to \calC$.
  The left vertical adjunction sends $a$ to $a^{\sharp}$, while the right vertical adjunction sends $a\circ\overline\Theta_X$ to $\bfR i_0^*(\widetilde a)\circ\eta_X^{\delta}$.
  Hence we have
  \begin{align}
    \overline\vartheta_{\calC}^{-1}
    \circ
    a^{\sharp}
    =
    \bfR i_0^*(\widetilde a)
    \circ
    \eta_X^{\delta}.
  \end{align}
  Composing with $\overline\vartheta_{\calC}$ gives
  \begin{align}
    \overline\vartheta_{\calC}
    \circ
    \bfR i_0^*(\widetilde a)
    \circ
    \eta_X^{\delta}
    =
    \overline\vartheta_{\calC}
    \circ
    \overline\vartheta_{\calC}^{-1}
    \circ
    a^{\sharp}
    =
    a^{\sharp}.
  \end{align}
\end{proof}

\begin{proposition}\label{jt.lem.unit}
  For every cofibrant $X\in\sSet^{\dag}_{\Joyal}$, the derived unit $\eta_X^{\delta} : X\to i_{0}^{*}(R\delta X)$ is a weak equivalence of $\sSet^{\dag}_{\Joyal}$.
  Here $R$ denotes a fixed functorial fibrant replacement in $L_{S_{\rev}}\sSpace^{\dag}_{\rev,\inj}$.
\end{proposition}

\begin{proof}
  Let $r_X: \delta X \to R\delta X$ be the fibrant-replacement map.
  By \cref{h1.lem.rev-models}, it is also a weak equivalence in $L_{S_{\rev}}\sSpace^{\dag}_{\rev,\inj}$.
  Hence \cref{jt.cor.localized-adjunction,jt.lem.transport-realization} give isomorphisms
  \begin{align}
    \bfL\frakC^{\prism}_{\dag}(R\delta X)
    &\xleftarrow[\simeq]{\bfL\frakC^{\prism}_{\dag}(r_X)}
    \bfL\frakC^{\prism}_{\dag}(\delta X)
    \xrightarrow[\simeq]{\overline\Theta_X}
    \bfL\frakC_{\dag}(X).
  \end{align}

  Choose a fibrant replacement $a : \frakC_{\dag}X \to \calD$ in $\sCat^{\dag}_{\Bergner}$.
  Since $X$ is cofibrant, $a$ represents an isomorphism $a : \bfL\frakC_{\dag}(X) \xrightarrow{\simeq} \calD$ in $\Ho(\sCat^{\dag}_{\Bergner})$.
  Consequently, the composite
  \begin{align}\label{jt.eq.unit-realization-comparison}
    \bfL\frakC^{\prism}_{\dag}(R\delta X)
    &\xrightarrow{
      (\bfL\frakC^{\prism}_{\dag}(r_X))^{-1}
    }
    \bfL\frakC^{\prism}_{\dag}(\delta X)
    \xrightarrow{\overline\Theta_X}
    \bfL\frakC_{\dag}(X)
    \xrightarrow{a}
    \calD
  \end{align}
  is an isomorphism.
  Apply \cref{jt.thm.strict} to $R\delta X$, and denote the resulting projective replacement and strictification map by
  \begin{align}
    R\delta X
    \xleftarrow[\sim]{p_R}
    Q_{\proj}(R\delta X)
    \xrightarrow{\varphi_R}
    \N^{\prism}_{\dag}(R_{\rmob}\frakC^{\prism}_{\dag}(Q_{\proj}(R\delta X))).
  \end{align}
  After identifying the displayed target with the chosen representative of the derived frame nerve, the derived unit is represented in $\Ho(L_{S_{\rev}}\sSpace^{\dag}_{\rev,\inj})$ by $\eta^{\prism}_{R\delta X} = [\varphi_R][p_R]^{-1}$.
  In particular, $\eta^{\prism}_{R\delta X}$ is an isomorphism.

  Applying $\bfR\N^{\prism}_{\dag}$ to \eqref{jt.eq.unit-realization-comparison} and composing with this unit therefore gives an isomorphism
  \begin{align}\label{jt.eq.unit-nerve-comparison}
    R\delta X
    \xrightarrow{\eta^{\prism}_{R\delta X}}
    \bfR\N^{\prism}_{\dag}
    \bfL\frakC^{\prism}_{\dag}(R\delta X) 
    \xrightarrow{\bfR\N^{\prism}_{\dag}(a \circ \overline\Theta_X \circ (\bfL\frakC^{\prism}_{\dag}(r_X))^{-1})}
    \bfR\N^{\prism}_{\dag}(\calD).
  \end{align}
  Let $\widetilde a: \delta X \to \bfR\N^{\prism}_{\dag}(\calD)$ be the derived frame adjoint of $a\circ\overline\Theta_X$.
  Naturality of $\eta^{\prism}$ gives
  \begin{align}
    &\bfR\N^{\prism}_{\dag}(a \circ \overline\Theta_X \circ (\bfL\frakC^{\prism}_{\dag}(r_X))^{-1}) \circ \eta^{\prism}_{R\delta X} \circ r_X \\
    &=
    \bfR\N^{\prism}_{\dag}(a \circ \overline\Theta_X \circ (\bfL\frakC^{\prism}_{\dag}(r_X))^{-1}) \circ \bfR\N^{\prism}_{\dag} (\bfL\frakC^{\prism}_{\dag}(r_X)) \circ \eta^{\prism}_{\delta X} \\
    &=
    \bfR\N^{\prism}_{\dag}(a \circ \overline\Theta_X) \circ \eta^{\prism}_{\delta X} \\
    &= 
    \widetilde a.
  \end{align}

  Let $\gamma_X$ denote the isomorphism in \eqref{jt.eq.unit-nerve-comparison}.
  The preceding naturality calculation is $\gamma_X \circ [r_X] = \widetilde a$.
  Both $\gamma_X$ and $[r_X]$ are isomorphisms, so $\widetilde a$ is an isomorphism, and hence so is $\bfR i_0^*(\widetilde a)$.
  The morphism $\overline\vartheta_{\calD}$ is an isomorphism, while $a^{\sharp}$ is the derived unit of the dagger Joyal--Bergner adjunction at $X$ and is an isomorphism by the Quillen equivalence of \cite[Theorem~3.5.6]{Aka26}.

  The identity of \cref{jt.lem.unit-triangle} therefore implies by two-out-of-three that $\eta_X^{\delta}$ is an isomorphism in $\Ho(\sSet^{\dag}_{\Joyal})$.
  Equivalently, $\eta_X^{\delta}$ is represented by a dagger Joyal weak equivalence.
\end{proof}

\begin{theorem}[The dagger Joyal--Tierney equivalence]\label{jt.thm.main}
  The Quillen adjunction
  \begin{align}
    \delta:
    \sSet^{\dag}_{\Joyal}
    \rightleftarrows
    L_{S_{\rev}}\sSpace^{\dag}_{\rev,\inj}
    :i_{0}^{*}
  \end{align}
  is a Quillen equivalence.
\end{theorem}

\begin{proof}
  It is a Quillen adjunction by \cref{jt.lem.star}.
  By \cite[Corollary 1.3.16]{Hov99}, it is a Quillen equivalence, since $i_{0}^{*}$ reflects weak equivalences between fibrant objects (by \cref{jt.cor.i0-reflects}) and the derived unit is a weak equivalence at every cofibrant object (by \cref{jt.lem.unit}).
\end{proof}

\begin{corollary}\label{jt.cor.goal}
  There exists a chain of Quillen equivalences
  \begin{align}
    \sCat^{\dag}_{\Bergner}
    \underset{\frakC_{\dag}}{\overset{\N_{\dag}}{\rightleftarrows}}
    \sSet^{\dag}_{\Joyal}
    \underset{i_0^*}{\overset{\delta}{\rightleftarrows}}
    L_{S_{\rev}}\sSpace^{\dag}_{\rev,\inj} 
    \underset{\id}{\overset{\id}{\rightleftarrows}}
    L_{\widetilde S_{\rev}}\sSpace^{\dag}_{\rev,\proj}
    \underset{u^*}{\overset{u_!}{\rightleftarrows}}
    L_{u_!(\widetilde S_{\rev})}\sSpace^{\dag}_{\proj}.
  \end{align}
\end{corollary}

\begin{proof}
  The first is the dagger Joyal--Bergner equivalence of \cite[Theorem~3.5.6]{Aka26}, the second is \cref{jt.thm.main} together with the identity Quillen equivalence of \cref{h1.lem.rev-models}, and the third is \cref{h1.thm.main}.
\end{proof}

\section*{References}
\begingroup
\printbibliography[heading=none]
\endgroup

\end{document}